\documentclass[a4paper;11pt]{amsart}
\usepackage{mathrsfs}
 \usepackage[arrow,matrix]{xy}
\usepackage{amsmath,amssymb,amscd,bbm,amsthm,mathrsfs}
\newtheorem{thm}{Theorem}[section]
\newtheorem{lem}{Lemma}[section]
\newtheorem{cor}{Corollary}[section]
\newtheorem{prop}{Proposition}[section]
\newtheorem{rem}{Remark}[section]

\theoremstyle{definition}

\begin{document}
\numberwithin{equation}{section}

\title[On  the boundary complex   of   the $k$-Cauchy-Fueter complex          ]
 {   On  the  boundary complex   of    the $k$-Cauchy-Fueter complex      }
\author{  Wei Wang}
\thanks{
Supported by National Nature Science Foundation in China (No.
11971425)  }\thanks{   Department of Mathematics,
Zhejiang University, Zhejiang 310027,
 P. R. China, Email:   wwang@zju.edu.cn}

\begin{abstract}   The $k$-Cauchy-Fueter complex, $k=0,1,\ldots$, in quaternionic analysis are  the counterpart  of the Dolbeault complex in the theory of several complex variables. In this paper, we construct explicitly    boundary  complexes  of these complexes on   boundaries   of    domains, corresponding to  the tangential Cauchy-Riemann complex in complex analysis.
They are only known boundary complexes outside of complex analysis that have  interesting applications to the function theory. As an application, we establish the  Hartogs-Bochner extension    for $k$-regular functions,  the  quaternionic counterpart of holomorphic functions. These  boundary  complexes have a very simple form on a kind of  quadratic hypersurfaces, which have the structure of  right-type nilpotent Lie groups  of step two. They allow  us to introduce  the quaternionic Monge-Amp\`{e}re operator and     opens the door to investigate   pluripotential theory on such
  groups.  We also apply   abstract duality theorem to   boundary  complexes to obtain the generalization of Malgrange's vanishing theorem and   Hartogs-Bochner extension   for $k$-CF functions,    the  quaternionic counterpart of CR functions, on this kind of groups.
\end{abstract}

\maketitle
\section{Introduction}
 The Cauchy-Riemann operator and the Dolbeault complex  play   central  roles in the theory of several complex variables.  In quaternionic analysis, we have a family of operators,    the $k$-Cauchy-Fueter   operator, acting on  $ \odot^{k} \mathbb{C}^2 $-valued functions,
$k=0,1,\ldots$. This is because the group ${\rm SU}(2)$  of unit quaternionic numbers has a family of  irreducible representations $ \odot^{k} \mathbb{C}^2 $,  while the group   of unit complex numbers  has only one irreducible representation space $\mathbb{C}$.
The corresponding  complexes are  the $k$-Cauchy-Fueter complexes, which are already   known explicitly (cf. \cite{adams2,Ba,bS,bures,CSS,Wang,wang-mfd} and references therein), and   used to show several interesting properties of $k$-regular functions,  the  quaternionic counterpart of holomorphic functions.  The $0$-Cauchy-Fueter complex also has important applications  to  the  quaternionic Monge-Amp\`{e}re operator and  quaternionic
 plurisubharmonic functions (cf. \cite{wan-wang,wang21} and references therein). For a  differential complex, a fundamental  problem is to characterize domains on which the complex is exact, i.e.  the Poincar\'e Lemma holds or its cohomology groups vanish.  The Neumann problem associated to   the $k$-Cauchy-Fueter complex  on  $k$-pseudoconvex domains was investigated in \cite{wang19}. It is expected that the nonhomogeneous $k$-Cauchy-Fueter equation is solvable if and only if the domain is $k$-pseudoconvex.

In   the complex  case, when the Dolbeault complex   is restricted to a CR submanifold, ones obtain   the tangential Cauchy-Riemann complex, which is a powerful tool to investigate holomorphic functions on   domains and
  the Dolbeault complex.
 One way to study  the $k$-Cauchy-Fueter complex  and $k$-regular functions is to study its boundary  complex. The theory of boundary  complexes  of
  general differential complexes began in 1970s by Andreotti,   Hill,   Lojasiewicz,   Mackichan, and  Nacinovich et  al. (cf., e.g. \cite{Andreotti,Andreotti2,Nacinovich,Nacinovich85} and references therein).   In this paper, we will write down explicitly  the boundary  complex   of  the $k$-Cauchy-Fueter complex  on   boundaries  of    domains, and apply it to establish the  Hartogs-Bochner extension     for $k$-regular functions, and   construct the quaternionic Monge-Amp\`{e}re operator on right-type   nilpotent Lie groups  of step two, corresponding to a kind of rigid quadratic hypersurfaces. On this kind of groups, we also apply abstract duality theorem to   boundary  complexes to obtain the generalization of Malgrange's vanishing theorem and  the  Hartogs-Bochner extension    for $k$-CF functions, the  quaternionic counterpart of CR functions,  under the momentum condition. They are only known  boundary complexes outside of  complex analysis that have  interesting applications to the function theory.

\subsection{     The $k$-Cauchy-Fueter complex   }
Denote
\begin{equation*}\begin{aligned}\label{eq:Vab}
\mathcal{ {V}}^{\sigma,\tau}:=&\odot^{\sigma}\mathbb{C}^{2}\otimes
\wedge^\tau\mathbb{C}^{2(n+1)},
\end{aligned}\end{equation*} where $\odot^{\sigma}\mathbb{C}^{2}$ is the $\sigma$-th symmetric product of $\mathbb{C}^2 $, and    $\wedge^{\tau}\mathbb{C}^{2(n+1)}$  is the  $\tau$-th exterior product of $\mathbb{C}^{2(n+1)}$. For fixed $k=0,1,\cdots$,
the  {\it  $k$-Cauchy-Fueter complex} on $\mathbb{H}^{n+1}$ is  given by
\begin{equation}\begin{aligned}\label{cf}
0\rightarrow \Gamma(D, \mathcal{{V}}_0)
&\xrightarrow{\mathcal{ {D}}_{0}}\cdots \rightarrow \Gamma(D, \mathcal{{V}}_j)\xrightarrow{ \mathcal{{D}}_{j}} \Gamma(D, \mathcal{{V}}_{j+1})\rightarrow \cdots\xrightarrow{ \mathcal{{D}}_{2n }} \Gamma(D, \mathcal{{V}}_{2n+1})\rightarrow0,
\end{aligned}\end{equation}
for  a domain  $D $   in  $\mathbb{H}^{n+1}$, where $\Gamma(D, \mathcal{{V}}_j)$ is the space of smooth $ \mathcal{{V}}_j$-valued functions with
$
\mathcal{ {V}}_j:=  \mathcal{{V}}^{\sigma_j ,\tau_j } $ (see \eqref{eq:sigma-tau} for $\sigma_j $ and $\tau_j $).

To   write down   operators  in the complex \eqref{cf}, we need complex  vector fields \cite{Wang}
\begin{align}\label{nabla}
\left(\nabla_{\dot{A}A'}\right):=\left(\begin{array}{rr} \partial_{ {1}}+\textbf{i}\partial_{ {2}}& -\partial_{ {3}}-\textbf{i}\partial_{ {4}}\\ \partial_{ {3}}-\textbf{i}\partial_{ {4}}&\ \  \partial_{ {1}}-\textbf{i}\partial_{ {2}}\\  \vdots\qquad&\ \ \ \ \ \ \ \vdots\qquad\\ \partial_{ {4l+1}}+\textbf{i}\partial_{ {4l+2}}& -\partial_{ {4l+3}}-\textbf{i}\partial_{ {4l+4}}\\ \partial_{ {4l+3}}
-\textbf{i}\partial_{ {4l+4}}& \ \ \partial_{ {4l+1}}-\textbf{i}\partial_{ {4l+2}}\\ \ \ \ \ \vdots\qquad&\ \ \ \ \ \ \ \vdots\qquad\end{array}\right),
\end{align}
  where $\dot{A}=0,\ldots, 2n+1$, $A'=0',1'$, and $\partial_j=\frac {\partial}{\partial x_j}$. It is motivated by the embedding $\tau$ of  quaternionic algebra $\mathbb{H}$ into  $\mathfrak{gl}(2,\mathbb{C}):$
\begin{align}\label{tau}
\tau(x_{1}+x_{2}\textbf{i}+x_{3}\textbf{j}+x_{4}\textbf{k})=
\left(\begin{array}{rr} x_{1}+\textbf{i}x_{2}& -x_{3}-\textbf{i}x_{4}\\ x_{3}-\textbf{i}x_{4}& x_{1}-\textbf{i}x_{2}\end{array}\right).
\end{align} The quaternionic structure of $\mathbb{H}^{n+1}$ is encoded in these vector fields. In the sequel,  we identify $\mathbb{H}^{n+1}$ with the underlying space $\mathbb{R}^{4(n+1)}$.
  For a fixed basis
$\{\omega^0, \ldots$, $\omega^{2n+1}\}$ of $\mathbb{C}^{2(n+1)}$, define two differential operators $d_{A'}:\Gamma( D ,\wedge^{\tau}\mathbb{C}^{2(n+1)})\rightarrow \Gamma ( D ,\wedge^{\tau+1}\mathbb{C}^{2(n+1)})$  by
\begin{equation}\label{eq:d01}\begin{split}&d_{A'}F:=\sum_{\mathbf{\dot{A}}}\sum_{\dot{A}=0}^{2n+1}\nabla_{\dot{A}A' }f_{\mathbf{\dot{A}}}~\omega^{\dot{A}}\wedge\omega^{\mathbf{\dot{A}}},
\end{split}\end{equation}$A'=0',1' $, for
   $F=\sum_{\mathbf{\dot{A}}}\ f_{\mathbf{\dot{A}}}~ \omega^{\mathbf{\dot{A}}}\in \Gamma( D ,\wedge^{\tau}\mathbb{C}^{2(n+1)})$, where
$\omega^{\mathbf{\dot{A}}}:=\omega^{\dot{A}_1}\wedge\ldots\wedge\omega^{\dot{A}_{\tau}}$ for the multi-index
${\mathbf{\dot{A}}}= \dot{A}_1 \cdots \dot{A}_{\tau} $. As $ {\partial}$ and $\overline{\partial}$ in complex analysis,
 $d_{0'}$ and $d_{1'}$ introduced in  \cite{wan-wang} are a pair of anti-commutative operators behaving like exterior differentials:
 \begin{equation}\label{eq:d20}
    d_{0'}^2=d_{1'}^2=0,\qquad
 d_{0'}d_{1'}=-d_{1'}d_{0'}.
 \end{equation}
 They
give a very useful   expression of the quaternionic Monge-Amp\`{e}re operator, and allow us to prove many important results in quaternionic pluripotential theory (cf. \cite{wan20,wan-wang,wang21}  and references therein).

By raising primed indices,  we have operator $d^{A'}$ \eqref{eq:d01'}.
It is convenient to identify $\odot^{\sigma}\mathbb{C}^{2 } $ with the space $   \mathcal{{P}}_\sigma(\mathbb{C}^2)$ of homogeneous polynomials of degree $\sigma$ on $\mathbb{C}^2$ \cite{LSW}. A  $ \mathcal{{V}}_j $-valued function $f$ can be viewed as a function in variables $x\in \mathbb{R}^{4(n+1)}$, $s^{ {A}'}\in \mathbb{C}^2$ and Grassmannian variables $\omega^A$:
\begin{equation*}
  f= f_{\mathbf{A}'\dot{\mathbf{A}}}(x)s^{\mathbf{A}'}\omega^{\dot{\mathbf{A} }}
\end{equation*}
where $s^{\mathbf{A}'}:=s^{A_1'}\cdot\ldots\cdot s^{A_\sigma'}$ for the multi-index $\mathbf{A}' = A_1' \cdots A_\sigma'$. Let
      $\partial_{A'} =\frac \partial{\partial s^{A'} }$.
Under this identification, differential operators in   the $k$-Cauchy-Fueter  complex   have a very simple form:
\begin{equation}\label{eq:operator-k-CF}    \mathcal{D}_j
 = \left\{
    \begin{array}{ll}  \partial_{A'} d^{ {A} '} ,\quad &{\rm if}\quad j=0, \ldots, k-1,\\
      d^{0 '}  d^{ 1'} ,\quad &{\rm if}\quad j=  k ,\\
      s_{A'} d^{ {A} '} ,\quad &{\rm if}\quad  j=k+1,\ldots,2n ,
    \end{array}\right.
\end{equation}

\subsection{The boundary complex} Consider a domain
\begin{equation}\label{eq:domain}
   D=\{\mathbf{q}=(\mathbf{q}',q_{n+1 }) \in \mathbb{H}^{n  }\times \mathbb{H} ;\varrho(\mathbf{q})>0\} .
\end{equation}
By rotation if necessary, we can assume the defining function $\varrho$ has the following form  near the origin: \begin{equation}\label{eq:quadratic-hypersurface}\begin{split}
\varrho (\mathbf{q} ) &=
 \operatorname{ Re}  q_{n +1  }-\phi (\mathbf{q}', \operatorname{Im} q_{n +1  }) ,
\end{split}\end{equation}  where  $\phi(\mathbf{q}', \operatorname{Im} q_{n  +1 } )= O(|\mathbf{q}',\operatorname{Im} q_{n  +1 }|^2)$.
As the boundary version of operators $d^{A'}$,  we introduce  operators $\mathfrak  d^{A'}:  \Gamma(bD,\wedge^\tau\mathbb{C}^{2n})$ $\rightarrow \Gamma(bD,\wedge^{\tau+ 1}\mathbb{C}^{2n})$ by
\begin{equation}\label{eq:mathfrak-d}\begin{split}
\mathfrak  d^{A'}f&=\sum_{ A,  {\mathbf{A}} } Z _{A  }^{A'}f_{ {\mathbf{A}}} \omega^A \wedge\omega^{ {\mathbf{A}}},
 \end{split} \end{equation}   for $f =\sum_{  |{\mathbf{A}}|=\tau} f_{ {\mathbf{A}}}  \omega^{ {\mathbf{A}}} $, where complex vector field   $
    Z_{A}^{A'}
$  is tangential to the boundary:
$
    Z_{A}^{A'}\varrho=0
$,  for $ {A}=0,\ldots, 2n-1$, $A'=0',1'$ (cf. Subsection \ref{subsection:Z-AA'}). They are.
Denote
\begin{equation*}\begin{aligned}\label{eq:Vab-boundary}
\mathscr {V}^{\sigma,\tau}:=&\odot^{\sigma}\mathbb{C}^{2}\otimes
\wedge^\tau\mathbb{C}^{2n }.
\end{aligned}\end{equation*}

\begin{thm} \label{thm:bd-complex-main:1}The   boundary complex of the    the $k$-Cauchy-Fueter complex is the differential complex
 \begin{equation} \label{eq:bdoperator}0\rightarrow
  \Gamma\left(bD,\mathscr{V}_{0}\right)\xrightarrow{\mathscr{D} _0 }   \cdots\longrightarrow
 \Gamma\left(bD,\mathscr{V}_j\right)\xrightarrow{\mathscr{D} _j }
 \Gamma\left(bD,\mathscr{V}_{j+1}\right)\rightarrow \cdots\xrightarrow{ \mathscr{D} _{2n-2}} \Gamma(bD,\mathscr{V}_{2n-1})\rightarrow0,
 \end{equation}
where $\mathscr{V}_j:=\mathscr{V}_j^{(1)}\oplus\mathscr{V}_j^{(2)}$ with
\begin{equation}\label{eq:Q}\begin{array}{llll}
\mathscr  {V}_j^{(1)}:=&\mathscr{V} ^{\sigma_j,\tau_j } ,\qquad\qquad & \mathscr  {V}_j^{(2)}:=\mathscr{V} ^{\sigma_{j+1 } ,\tau_j -1}  ,\qquad    & {\rm if}\quad j \neq k  ,\\\mathscr  {V}_k^{(1)}:=&\mathscr{V} ^{0,k } ,\qquad\qquad & \mathscr  {V}_k^{(2)}:=\mathscr{V} ^{0,k }   ,
\end{array}\end{equation}
and $\mathscr{V}_0^{(2)}=\emptyset$. For $\mathbb{ F}=(\mathbb{F}_1,  \mathbb{F}_2 )\in \Gamma ( bD,\mathscr{V}_ j^{(1)}  )\oplus \Gamma (bD,\mathscr{V}_ j^{(2)}   )$,
$
    {\mathscr{D} }_{j}  \mathbb{F} = ({\mathscr{D} }_{j}^{(1)} \mathbb{F} , {\mathscr{D} }_{j}^{(2)} \mathbb{F}  )
$
with
\begin{equation}  \label{eq:bdoperator0} {\mathscr{D} }_{j}^{(1)} \mathbb{F}
 = \left\{
    \begin{array}{ll}  \partial_{A'}\mathfrak  d^{ {A} '}\mathbb{F}_1+\mathcal E_0\wedge  \mathbb{F}_2,\quad &{\rm if}\quad j=0, \ldots, k-1,\\
     \mathfrak  d^{0 '}\mathfrak   d^{ 1'}\mathbb{F}_1 -\mathcal E_0  \wedge  ( {\mathbf T}^{1'0 '} \mathbb{F}_1 +  \mathbb{F}_2 ) ,\quad &{\rm if}\quad j=  k ,\\
      s_{A'}\mathfrak  d^{ {A} '}\mathbb{F}_1+\mathcal E_0\wedge  \mathbb{F}_2 ,\quad &{\rm if}\quad  j=k+1,\ldots,2n-2,
    \end{array}\right.
\end{equation}
where
$
   \mathcal E_0  :=-\mathfrak  {d}^{0'}\mathfrak  {d}^{1'}\varrho.
$
\end{thm}  Explicit formulae for ${\mathscr{D} }_{j}$'s are given by \eqref{eq:bdoperator<} and \eqref{eq:bdoperator>=}.
  Compared to $d^{A'}$ in \eqref{eq:d20},  operators $\mathfrak d^{A'}$'s usually do not behave like  anti-commutative  exterior differentials.  We call  a hypersurface  $bD$
  {\it right-type} if   $\mathcal E_0  $   vanishes, because on such groups, operator $\mathfrak d^{A'}$'s behave like that on  the  right quaternionic Heisenberg group \eqref{eq:right}. In this case,
   ${\mathscr{D} }_{j}$ maps  a $ \mathscr {V}_j^{(1)} $-valued function   to a    $ \mathscr {V}_{j+1}^{(1)} $-valued one    by \eqref{eq:bdoperator0},  i.e. we obtain    a  differential  subcomplex:
 \begin{equation} \label{eq:subcomplex}
0\longrightarrow \Gamma\left(bD,\mathscr  {V}_0^{(1)}\right)\xrightarrow{\mathscr{D} _0 } \cdots\longrightarrow \Gamma\left(bD,\mathscr {V}_j^{(1)}\right)\xrightarrow{\mathscr{D} _j }
    \cdots\xrightarrow{\mathscr{D} _{2n-2} }
   \Gamma\left(bD,\mathscr{V}_{2n-1}^{(1)}\right)\longrightarrow 0
 \end{equation}
with
\begin{equation*} \mathscr   D_j
 = \left\{
    \begin{array}{ll}  \partial_{A'}\mathfrak  d^{ {A} '} ,\quad &{\rm if}\quad j=0, \ldots, k-1,\\
     \mathfrak  d^{0 '}\mathfrak   d^{ 1'} ,\quad &{\rm if}\quad j=  k ,\\
      s_{A'}\mathfrak  d^{ {A} '} ,\quad &{\rm if}\quad  j=k+1,\ldots,2n-2.
    \end{array}\right.
\end{equation*}

This subcomplex     and its operators  are very similar to the    $k$-Cauchy-Fueter complex on $\mathbb{H}^{ n+1 }$.
  \subsection{The Hartogs-Bochner extension for $k$-regular functions}
On a domain $D\subset \mathbb{H}^{n+1}$,
a function   $f\in \Gamma( D,  \odot^k \mathbb{C}^2 )$ is   called {\it $k$-regular} if
$
  \mathcal{ D}_0f=0.
$
The space of all $k$-regular functions on a domain $D $ is denoted by $\mathscr{O}_k(D)$. On a domain $\Omega$ in $ bD$, a function   $F\in \Gamma( \Omega,  \odot^k \mathbb{C}^2 )$ is   called {\it $k$-CF} if
$
    \mathscr{D}_0F=0.
$
We  have the following    Hartogs-Bochner extension for $k$-regular functions.

\begin{thm} \label{thm:Hartogs-Bochner}  Let $D $ be a bounded domain in   $   \mathbb{H}^{n+1} $ ($n\geq 1$)   with
smooth boundary such that $ \mathbb{H}^{n+1} \setminus \overline{D} $ is connected. If $f $ is a smooth $k$-CF function on  $bD$, then there exists $\widetilde{f}\in \mathscr{ O}_k( D) $ smooth up to the boundary
 such that $\widetilde{f} =f $ on $bD$.
 \end{thm}

 This theorem for $k=1$ was proved by Maggesi-Pertici-Tomassini
\cite{MPT} recently. They introduced the notion of admissible functions on the boundary, which  coincides with the notion of $1$-CF functions here, although it is written in a different form.

For a $k$-CF function $f$ on $bD $, the  boundary complex allows us to construct a representative $ \widehat{ f}\in \Gamma(D, \odot^k \mathbb{C}^2)$ such that $ \widehat{ f}|_{bD}=f$ and
$
  \mathcal{ D}_0f
$ is flat on $bD$. Then by using the solution to the nonhomogeneous
 $k$-Cauchy-Fueter equation  with compact support, we can construct  the $k$-regular function $\widetilde{f}  $.

\subsection{   Right-type
  groups   and  the quaternionic Monge-Amp\`{e}re operator } As in the complex case,  we call a domain $D$      a {\it rigid   domains}  if  it has a defining function of the following form: \begin{equation}\label{eq:rigid-hypersurface}\begin{split}
\varrho (\mathbf{q})&=
{\rm Re}\, q_{n+1  }-\phi (\mathbf{q}') ,
\end{split}\end{equation}i.e. $\phi$ is independent of $\operatorname{Im} q_{n +1  }$.
If we take $\phi(\mathbf{q}' )= \sum_{j,k=1 }^{4n}  \mathbb{S}_{jk }x_jx_k$ for a real symmetric $ 4n \times  4n $-matrix $\mathbb S $,    the boundary is the {\it  rigid  quadratic hypersurface}. It has the structure of a
nilpotent Lie group  of step two, denoted by $\mathcal{N}_{\mathbb{S}}$, with
the multiplication given by
\begin{equation} \label{eq:mul}
(\mathbf{x} , \mathbf{t})\cdot(\mathbf{{y}}, \mathbf{{s}})= {\left(\mathbf{x}+\mathbf{y}  ,  t_{\beta}+s_{\beta}+2
 \sum_{a,b=1}^{4 n }B_{ab}^{\beta}x_{a}y_{b}\right)},
\end{equation}
for $\mathbf{x},\mathbf{y}\in \mathbb{R}^{4n},\ \mathbf{t},\mathbf{s}\in \mathbb{R}^{3},\ \beta=1,2,3,$ where skew symmetric matrices  $B^{\beta}=(B_{ab}^{\beta})$
are determined by $ \mathbb{S}$ in \eqref{eq:B-S}. We call the associated group $\mathcal{N}_{\mathbb{S}}$
  {\it right-type} if the quadratic hypersurface   is. A simple characterization of   right-type groups in terms of matrices $B^\beta$ is given in
   Proposition \ref{prop:right-type}. This kind of groups are abundant.

The \emph{right quaternionic Heisenberg group} is $ \mathbb{H}^n\times\rm{Im}\ \mathbb{H}$ with  the multiplication given by
\begin{align}\label{eq:right}
(\mathbf{p} , \mathbf{t})\cdot(\mathbf{q} , {\mathbf{s}} )=  \left(\mathbf{p}+ \mathbf{ q } ,  \mathbf{t} +  {\mathbf{s}}  +2{\rm{Im}} (\mathbf{{p}} \overline{\mathbf{q} })\right),
\end{align}
where $ \mathbf{ p },\mathbf{q } \in \mathbb{H}^{n }, \mathbf{t}, {\mathbf{s}}\in {\rm{Im}}\, \mathbb{H} $.
 This group  is right-type  and the tangential $k$-Cauchy-Fueter complex   constructed in \cite{shi-wang}  is a special case of  the subcomplex
\eqref{eq:subcomplex}.
While the  {\it left quaternionic Heisenberg group} is $ \mathbb{H}^n\times\rm{Im}\ \mathbb{H}$ with the multiplication given by
 \begin{equation}\label{eq:left}
(\mathbf{p} , \mathbf{t})\cdot(\mathbf{q} , {\mathbf{s}} )=  \left(\mathbf{p}+ \mathbf{ q}  ,  \mathbf{t} +  {\mathbf{s}}  +2{\rm{Im}} (\overline{\mathbf{p}}{\mathbf{q} })\right).
 \end{equation} This group  is not right-type, but we have already constructed the tangential $k$-Cauchy-Fueter complex on this group  by using the twistor method (cf. \cite[Theorem 1.0.1]{wang1}), which is   different from  the complex on the right one. To understand this difference,  which is explained by  Theorem \ref{thm:bd-complex-main:1},   is one of  main motivations of this paper.

  On   a    right-type
  group,  operators $\mathfrak d^{A'}$'s  behave like  anti-commutative  exterior differentials:
\begin{equation}\label{eq:delta20}
   \mathfrak d_{0'}^2=\mathfrak  d_{1'}^2=0,\qquad
 \mathfrak    d_{0'}\mathfrak  d_{1'}=-\mathfrak  d_{1'}   \mathfrak d_{0'} .
\end{equation}
 This kind of  nice operators  was first observed  for $(4n+1)$-dimensional   Heisenberg group in \cite{wang21}.
$\mathfrak{d}^{A'}$'s allow  us to introduce the {\it quaternionic Monge-Amp\`{e}re operator} on a right-type  group as
\begin{equation}\label{eq:wedge-delta}
   \triangle
u \wedge\ldots\wedge\triangle u,
\qquad {\rm where
}\quad
   \triangle
u:=\mathfrak d_{0'}\mathfrak d_{1'}u,
\end{equation}and  plurisubharmonicity of a  function $u$  in terms of the positivity of the $2$-current $\triangle u$ ($2$-form for a $ C^2 $-function). A direct application of \eqref{eq:delta20}
gives us a key identity for $\mathfrak d_{0'}$, $\mathfrak d_{1'}$ and $\triangle$, by which we can show  important  Chern-Levine-Nirenberg type estimate in Theorem \ref{thm:estimate-C0}, and obtain the   existence  of the   Monge-Amp\`{e}re measure   for a continuous
plurisubharmonic function.
This opens the door to investigate the quaternionic Monge-Amp\`{e}re equation and pluripotential theory on   right-type
  groups, generalizing the theory
    on the Heisenberg group \cite{wang21}.

\subsection{The generalization of Malgrange's vanishing theorem and the  Hartogs-Bochner  extension   for   $k$-CF functions}
It is  important  to investigate the
validity of the Poincar\'e lemma, i.e. vanishing of its cohomology groups, for     boundary  complexes   in terms of the Levi-type forms, which were introduced in the study of
the Neumann problem for the $k$-Cauchy-Fueter complex in \cite{wang19}, and the behavior of $k$-CF functions  related to hypoellipticity,   unique continuation, the maximum modulus principle, hypoanaliticity etc., as in the theory of CR functions and the tangential Cauchy-Riemann complex  on   CR manifolds  (see e.g. \cite{AH72,Brinkschulte10,Brinkschulte,HL,HN00,Hill03,Hill,LL99,LL00,Nacinovich,Nacinovich17}). In this paper, we discuss the most simple case: right-type
  groups. We have subelliptic estimate.

\begin{prop} \label{prop:Subelliptic-estimate} Suppose that  $\mathcal{N}_{\mathbb{S}}$ is a stratified  right-type  group    and $K\Subset \mathcal{N}_{\mathbb{S}}$. Then  there are positive
constants $  C_1, C_2 $ only depending on $K$  such that
    \begin{equation}\label{eq:Subelliptic-estimate}
\left\|\mathscr{D  }_{0} f\right\|_0^2\geq C_1\| f\|_{\frac 12}^2-C_2\| f\|_{0}^2 ,
    \end{equation}
    for any $f\in C_0^\infty(K,\mathscr{V}_0)$.
 \end{prop}

Abstract duality theorem  for  a Fr\'echet-Schwartz space or  the dual of a Fr\'echet-Schwartz space with   topological homomorphisms can be applied to our case.
Let $\widehat{\mathscr{V}}_\bullet $ be
the differential complex dual to the $k$-Cauchy-Fueter complex.
We  have the following generalization of Malgrange's vanishing theorem.
\begin{thm} \label{thm:Malgrange} On a  right-type
  group  $ \mathcal{{N}}_{\mathbb{S}}$   satisfying condition (H), the homology groups  $H_0( \mathscr E(\mathcal{{N}}_{\mathbb{S}},   \widehat{\mathscr{V}}_\bullet))$ and $H_0(\mathscr D'(\mathcal{{N}}_{\mathbb{S}},  \widehat{\mathscr{V}}_\bullet))$ vanish.
 \end{thm}

In the step two case,  a stratified    group is exactly a group satisfying H\"ormander's condition, which is used to promise hypoellipticity of the SubLaplacian,
while the assumption of condition (H) is a technique condition to  promise unique continuation.
 We also prove  the Hartogs-Bochner  extension   for $k$-CF functions under the momentum condition  in Theorem \ref{thm:Hartogs-Bochner-CF}. See \cite{Brinkschulte,LL99,LL00,Nacinovich,Nacinovich17}  for Malgrange's vanishing theorem and the Hartogs-Bochner  extension   for CR functions on CR manifolds.

This paper is organized as follows. In Section 2, we give the basic notations, complex tangential vector fields  $Z_{AA'}$, and  recall the definition of the boundary  complex  of a
  general differential complex. In Section 3, we determine   vector spaces of the boundary complex of the   $k$-Cauchy-Fueter complex and   induced   operators  for $j< k-1$. For the case $j\geq k-1$, it is done
in Section 4. In Section 5, the  Hartogs-Bochner extension theorem   for $k$-regular functions   is established. In Section 6, we prove the characterization of
   right-type  groups in Proposition \ref{prop:right-type} and the  Chern-Levine-Nirenberg type estimate,  and   construct quaternionic Monge-Amp\`{e}re measure
  on    right-type
  groups. In Section 7,  after establish subelliptic estimate,  we apply   abstract duality theorem to  boundary  complexes on     right-type
  groups  to obtain the generalization of Malgrange's vanishing theorem and the  Hartogs-Bochner extension    for $k$-CF functions under the momentum condition. In the appendix,
we show differential operators \eqref{eq:operator-k-CF} in the $k$-Cauchy-Fueter  complex  coincide  with the usual form used before (e.g. in \cite{Wang,wang-mfd,wang19}).
\section{preliminaries }\subsection{Notations }We adopt the following index notations:
  \begin{equation*}
  \begin{split}
  \dot{A} , \dot{ B} ,  \dot{C} ,  \cdots&\in\{0,1, \cdots,2n+1\},\\A,B,C,  \cdots&\in\{ 0,1,\cdots,2n-1\},\\
  A',B',C',D', \cdots&\in\{0',1' \}.
  \end{split}
  \end{equation*}
We will use the Einstein convention of taking summation for repeated indices.

It  is similar to lower  or raise  indices by a metric  in differential geometry that we use
 \begin{equation} \label{eq:varepsilon} (\varepsilon_{A'B'})=\left( \begin{array}{cc} 0&
 1\\-1& 0\end{array}\right) \qquad {\rm and } \qquad \left (\varepsilon^{A'B'}\right) =\left( \begin{array}{cc} 0&
- 1\\1& 0\end{array}\right)
 \end{equation} to lower or raise primed  indices,  where $(\varepsilon^{A'B'})$ is the inverse of $(\varepsilon_{A'B'})$.
 For example,
\begin{equation} \label{eq:varepsilon1}
    f^{ A'  } =   f_{ {B '} }\varepsilon^{B' A' },\qquad  f^{ {A' } } \varepsilon_{A ' C' }=  f_{ {C' } }.
\end{equation} It is the same when an index is raised (or lowered) and then lowered  (or raised) \cite{PR1,PR2}. Here $(\varepsilon^{A'B'})$ is a volume element in  $\mathbb{C}^2$. The contraction of an upper and a lower primed  indices is invariant under the action of ${\rm SL}(2,\mathbb{C})$ on $\mathbb{C}^2$ (cf. e.g. \cite{wang-mfd,wang19}). These primed  and unprimed  indices are the generalization of Penrose's two-spinor  notations (cf.  \cite{PR1,PR2}).  By raising primed indices, we have
 \begin{equation}\label{eq:k-CF-raised-flat} ( \nabla_{\dot{A}}^{A'})=  \left(
                                      \begin{array}{rr}\vdots\hskip 13mm&\vdots\hskip 13mm\\- \partial_{ {4l+3}} -\textbf{i}\partial_{ {4l+4}}
                                    & -\partial_{ {4l+1}} -\textbf{i}\partial_{ {4l+2}} \\ \partial_{ {4l+1}}-\textbf{i}\partial_{ {4l+2}}
                                  &   - \partial_{ {4l+3}}+\textbf{i}\partial_{ {4l+4}}\\\vdots\hskip 13mm&\vdots\hskip 13mm
                                                                    \\- \partial_{ {4n+3}} -\textbf{i}\partial_{ {4n+4}}
                                    & -\partial_{ {4n+1}} -\textbf{i}\partial_{ {4n+2}} \\ \partial_{ {4n+1}}-\textbf{i}\partial_{ {4n+2}}
                                  &   - \partial_{ {4n+3}}+\textbf{i}\partial_{ {4n+4}}           \end{array}
                                    \right),
\end{equation}i.e.
$
   \nabla_{{\dot{A}}}^{0' }=\nabla_{{\dot{A}}1' },   \nabla_{{\dot{A}}}^{1' }=-\nabla_{{\dot{A}}0' }
$ by \eqref{eq:varepsilon1}. Here $\dot{A}  $ is the row index,  while ${A'}$ is the column index. For  $ f=  f_{\mathbf{\dot{A}}} \omega^{\mathbf{\dot{A}}}\in \Gamma( {D},\wedge^{\tau}\mathbb{C}^{2(n+1)}) ,$
\begin{equation}\label{eq:d01'}\begin{split}&d^{A'}f:= \nabla_{{\dot{A}}}^{A' }f_{\mathbf{\dot{A}}}~\omega^{\dot{A}}\wedge\omega^{\mathbf{\dot{A}}}.
\end{split}\end{equation}

\begin{prop}\label{prop:dd} {\rm (Proposition 2.2 in \cite{wan-wang})}  (1) $d^{(A'} d^{B')}=0$, i.e. \eqref{eq:d20} holds. \\
(2) For $F\in \Gamma( \wedge^{\tau}\mathbb{C}^{2n})$, $G\in \Gamma( \wedge^{\varsigma}\mathbb{C}^{2n})$, we have\begin{equation*}  d^{A'}(F\wedge G)=  d^{A'}
F\wedge
G+(-1)^{\tau}F\wedge   d^{A'} G,\qquad  {A'}=0' ,1' .\end{equation*}
\end{prop}
\begin{proof} See \eqref{eq:sym-0}  for  symmetrization. For $F=F_{\mathbf{A} }~ \omega^{\mathbf{A}}$ with $|\mathbf{A}|=\tau$,
\begin{equation*}
   d^{A'}  d^{B'}F= \nabla_{\dot{A}  }^{A'} \nabla_{\dot{B}  }^{B'} F_{\mathbf{A} }~ \omega^{\dot{A}  }\wedge\omega^{\dot{B}  }\wedge\omega^{\mathbf{A}}=-\nabla_{\dot{B}  }^{B'}\nabla_{\dot{A } }^{A'}F_{\mathbf{A} }~
  \omega^{\dot{B}  }\wedge\omega^{\dot{A}  }\wedge\omega^{\mathbf{A}}=-  d^{B'}  d^{A'}F,
\end{equation*}by $\nabla_{\dot{A } }^{A'}$'s commuting each other as differential operators of constant coefficients.
The proof of (2) is the same as \eqref{eq:Leibnitz2}.
\end{proof}

The Leibnitz law in Proposition \ref{prop:dd} (2) will be frequently used.
 It is convenient to identify $\odot^{\sigma}\mathbb{C}^{2 } $ with the space $   \mathcal{{P}}_\sigma(\mathbb{C}^2)$ of homogeneous polynomials of degree $\sigma$ on $\mathbb{C}^2$ \cite{LSW}.   $ \mathcal{{V}}_j$ is realized as $  \mathcal{{P}}_{\sigma_j }(\mathbb{C}^2) \otimes \wedge^{\tau_j}\mathbb{C}^{2(n +1)} $.
Let $s^{0'}, s^{1'}$ be coordinate functions of $\mathbb{C}^2$.  We choose
      \begin{equation*}
        \mathbf{S}_\sigma^a:= \frac {(s^{0'})^{\sigma-a}}{(\sigma-a)!}  \frac {(s^{1'})^a}{a!} ,
      \end{equation*}$a=0,\ldots,\sigma,$
     as a basis of $ \mathcal{{P}}_{\sigma  }(\mathbb{C}^2)$. The advantage of this basis is that
      \begin{equation}\label{eq:partial-S}
       \partial_{0'}  \mathbf{S}_\sigma^a=\mathbf{S}_{\sigma-1}^{a},\qquad \partial_{1'}  \mathbf{S}_\sigma^a=\mathbf{S}_{\sigma-1}^{a-1 }.
      \end{equation}We   write
       $s^{\mathbf{A}'}:=s^{A_1'}\cdots s^{A_\sigma'}$ for $\mathbf{A}' = A_1' \ldots A_\sigma'$ and set $|\mathbf{A }'|=\sigma$ and $o(\mathbf{A}')$ to be the number of $1'$  in $\mathbf{A}'$.

     For $j\geq k$, we also use
$
        \widetilde{ \mathbf{S}}_\sigma^a:= (s_{0'})^{\sigma-a} (s_{1'})^a  ,$ $ a=0,\ldots,\sigma,
$
     as a basis of $ \mathcal{{P}}_{\sigma  }(\mathbb{C}^2)$. Here $s_{A'} $'s are  obtained by lowering primed  indices. The advantage of this basis is that
      \begin{equation}\label{eq:partial-S'}
       s_{0'}\widetilde{ \mathbf{S}}_\sigma^a=\widetilde{ \mathbf{S}}_{\sigma+1}^{a},\qquad s_{1'}\widetilde{ \mathbf{S}}_\sigma^a=\widetilde{ \mathbf{S}}_{\sigma+1}^{a +1}.
      \end{equation}
        Here and in the sequel, we use the convention $\mathbf{S}_{\sigma-1}^{\sigma }=0=\mathbf{S}_{\sigma-1}^{-1 }$,  $\widetilde{ \mathbf{S}}_{\sigma-1}^{\sigma }=0=\widetilde{ \mathbf{S}}_{\sigma-1}^{-1 }$.

     For fixed $k$, indices in  $
\mathcal{ {V}}_j:= \mathcal{ {V}}^{\sigma_j ,\tau_j } $ in \eqref{cf} are given by
 \begin{equation}\label{eq:sigma-tau} \begin{split}&
    \sigma_j  :=\left\{
    \begin{array}{ll} k-j  ,\quad &{\rm if}\quad j=0, \ldots, k,\\
       j-k-1  ,\quad &{\rm if}\quad  j=k+1,\ldots,2n+1,
    \end{array}\right.
    \\& \tau_j :=\left\{
    \begin{array}{ll}  j ,\qquad \qquad \quad &{\rm if}\quad j=0, \ldots, k,\\
        j+1 ,\quad &{\rm if}\quad  j=k+1,\ldots,2n+1.
    \end{array}\right.
 \end{split}\end{equation}

  Under this realization, we can easily see \eqref{cf} is a complex, i.e. $ \mathcal{{D}}_{j+1} \mathcal{{D}}_{j}=0$, by
\begin{equation}\label{eq:DD1}
   \partial_{A'}\partial_{B'} d^{ {A} '} d^{ {B} '}=0,\qquad  s_{A'}s_{B'} d^{ {A} '} d^{ {B} '}=0,
\end{equation}since $d^{ {A} '} d^{ {B} '}$ is skew-symmetric in ${A} ', {B} '$ by Proposition \ref{prop:dd} (1), while $\partial_{A'}\partial_{B'}$ and $s_{A'}s_{B'}$
are both symmetric in ${A} ', {B} '$, and
\begin{equation}\label{eq:DD2}
   d^{ 0'}d^{ 1'} \partial_{A'}d^{ {A} '} =0 ,\qquad   s_{A'} d^{ {A} '} d^{ 0'}d^{ 1'} =0.
\end{equation}

 \subsection{Complex tangential vector fields  $Z _{  AA'}$} \label{subsection:Z-AA'} Write  $q_{l+1 }:=x_{4l+1}+\textbf{i}x_{4l+2}+\textbf{j}x_{4l+3}+\textbf{k}x_{4l+4},$  $l=0,\cdots,n-1 $.       The {\it  Cauchy-Fueter operator} on $\mathbb{H}^{n+1}$ is
  $$\overline{\partial}_{q_{l +1}}=\partial_{x_{4l+1}}+\textbf{i}
\partial_{x_{4l+2}}+\textbf{j}\partial_{x_{4l+3}}+\textbf{k}
\partial_{x_{4l+4}}, $$
$l=0,\ldots,n $.
We have {\it quaternionic tangential vector fields} on the boundary:
\begin{equation}\label{eq:Z-S}
    \overline{Q}_{l}=\overline{\partial}_{q_l}-\overline\partial_{q_l} \varrho\cdot(\overline\partial_{q_{n+1}}\varrho)^{-1}\cdot
    \overline\partial_{q_{n+1}}
\end{equation}$l=1,\ldots,n  $, since $\overline{Q}_l\varrho=0$ by definition.

The definition of  the map  $\tau$ in \eqref{tau}   can extended  to a mapping from quaternionic $l \times m$-matrices to complex
$ 2l \times 2m $-matrices by setting $\tau(\mathbf{a}):= (\tau(\mathbf{a}_{jk}))$ for a quaternionic  matrix $\mathbf{a} = (\mathbf{a}_{jk})$.
It is known $\tau(\mathbf{a} \mathbf{b})=\tau(\mathbf{a})\tau(\mathbf{b})$ for a quaternionic $ p\times m $-matrix $\mathbf{a}$ and a quaternionic $(m\times l)$-matrix $\mathbf{b}$ \cite[Lemma 2.1]{wan-wang}. Apply  it to  get
\begin{equation*}
   \tau(\overline\partial_{q_l})=\left(\begin{array}{ll} \nabla_{(2l)0'}& \nabla_{(2l)1'}\\ \nabla_{(2l+1)0'}&\nabla_{(2l+1)1'} \end{array}\right),\qquad \tau(\overline\partial_{q_l}\varrho)=\left(\begin{array}{ll} \nabla_{(2l)0'}\varrho& \nabla_{(2l)1'}\varrho\\ \nabla_{(2l+1)0'}\varrho&\nabla_{(2l+1)1'}\varrho \end{array}\right),
\end{equation*}by definition of  $\nabla_{AA'}$ in \eqref{nabla}, and so
\begin{equation}\label{eq:n-AA}
  (\nabla_{AA'}\varrho)=\tau\left(\begin{array}{c} \overline\partial_{q_1}\varrho
 \\ \vdots \\  \overline\partial_{q_n}\varrho\end{array}\right)  \qquad{\rm and }\qquad(Z_{AA'}): =\tau\left(\begin{array}{c}   \overline{Q}_1
 \\ \vdots \\    \overline{Q}_n  \end{array}\right).
 \end{equation} The vector fields
  $ Z_{AA'}$'s are   tangential to the    boundary     (\ref{eq:quadratic-hypersurface}), i.e. $Z_{AA'}\varrho=0$,  since $\overline{Q}_l\varrho=0$.
 If we write $Q_l=X_{4l+1} +\textbf{i}X_{4l+2}
                                     +\textbf{j} X_{4l+3}  +\textbf{k}X_{4l+4}$, then we have
\begin{equation}\label{eq:t-k-CF} ( Z_{A A'})=  \left(
                                      \begin{array}{rr}\vdots\hskip 13mm&\vdots\hskip 13mm\\  X_{4l+1} +\textbf{i}X_{4l+2}
                                    &- X_{4l+3}  -\textbf{i}X_{4l+4} \\
                                  X_{ 4l+3}  -\textbf{i}X_{4l+4}&X_{4l+1} -\textbf{i}X_{4l+2}  \\\vdots\hskip 13mm&\vdots\hskip 13mm
                                                                               \end{array}
                                    \right).
\end{equation}

Denote $\mathbf{N}_{B'C'}:=\nabla_{(2n+o(B'))C'}.$ Then,
\begin{equation*}
  (\mathbf{N}_{B'C'})=\left(\begin{array}{ll} \nabla_{(2n)0'} & \nabla_{(2n)1'}\\ \nabla_{(2n+1)0'} &\nabla_{(2n+1)1'}  \end{array}\right) ,
\end{equation*}
and $ \mathbf{N} \varrho $ is the $2\times 2$-matrix $\tau\left( \overline\partial_{q_{n+1}}\varrho \right)$. By applying $\tau$ to \eqref{eq:Z-S}, $ Z_{AA'}$ can be written as
\begin{equation}\label{eq:Z-AA'1}
   Z_{AC'}=\nabla_{AC'}-n_{AB'} \mathbf{N}_{B'C'},
\end{equation}
with
\begin{equation}\label{eq:n}
  n_{A B'}=\nabla_{A D'}\varrho\cdot (\mathbf{N}\varrho)^{-1}_{D'B' }.
\end{equation}
 Then by raising primed indices  \eqref{eq:varepsilon1}, we get
 \begin{equation}\label{eq:Z}
   Z _{A}^{A'}=\nabla_{A}^{A'}-n_{A B'} \mathbf{N}_{B'}^{A'},
\end{equation}with
\begin{equation}\label{eq:t-k-CF-raised} ( Z_{A}^{ A'})=  \left(
                                      \begin{array}{rr}\vdots\hskip 13mm&\vdots\hskip 13mm\\ - X_{4l+3}  -\textbf{i}X_{4l+4}
                                    &- X_{4l+1} -\textbf{i}X_{4l+2}\\X_{4l+1} -\textbf{i}X_{4l+2}
                           &    -   X_{ 4l+3}+\textbf{i}X_{4l+4}\\\vdots\hskip 13mm&\vdots\hskip 13mm
                                                                               \end{array}
                                    \right).
\end{equation} It also also direct to check that
$
    Z_{A}^{ A'}\varrho=0
$. The operator $ d^{ {A} '}$ in \eqref{eq:d01'} can be rewritten as
 \begin{equation}\label{eq:dA'0}\begin{split}
  d^{ {A} '}f   =& \left\{ \nabla_{  {A} }^{ A '} f_{ \mathbf{\dot {A}} }   \omega^{ A}   +\mathbf{N}_{  C' }^{ A '} f_{ \mathbf{\dot {A}}  }\omega^{  2n+ o(C')  }\right\}\wedge  \omega^{\mathbf{\dot {A}} }
   .
\end{split}\end{equation}

\subsection{Definition of the boundary complex} Let us recall the definition of the boundary  complex  of a
  general differential complex (cf. e.g. \cite{Andreotti,Andreotti2,Nacinovich}). Let $D$ be a domain in $ \mathbb{R}^N$.
Suppose that we have a differential complex
on $ \mathbb{R}^N$:
\begin{equation*}
\Gamma(  \mathbb{R}^N , E^{(0)} )\xrightarrow{A_0(x, \partial)}\Gamma( \mathbb{R}^N, E^{(1)} )\xrightarrow{A_1(x,\partial)}\Gamma( \mathbb{R}^N, E^{(2)})\xrightarrow{A_2(x, \partial)}\cdots.
\end{equation*} We say
$u \in \Gamma(\overline{D} , E^{(j )} )$ has {\it zero Cauchy
data} on the boundary $bD$ for $A_j(x, \partial)$  if for any $\psi\in \Gamma(U , E^{(j+1)} )$  compactly  supported in $U$, we have
\begin{equation}\label{eq:zero-Cauchy}
   \int_D\langle A_j(x, \partial)u, \psi\rangle dV=\int_D \langle u,   A_j^*(x, \partial)\psi \rangle dV,
\end{equation} where $\langle \cdot, \cdot\rangle$ is the inner product of $ E^{(j)}$, and $A_j^*(x, \partial)$ is the formal adjoint operator of $A_j(x, \partial)$.
  Set
\begin{equation}\label{eq:zero-Cauchy-sheaf}
   \mathcal{J}_{A_j}(bD,U):=\{u \in \Gamma (U,  E^{(j)} ); u  \hskip 2mm{\rm has \hskip 2mm zero\hskip 2mm Cauchy \hskip 2mm data\hskip 2mm   on \hskip 2mm}   bD \hskip 2mm  \hskip 2mm{\rm for}\hskip 2mm  A_j(x, \partial)\},
\end{equation}for any open set $U$.
 Then $U\rightarrow  \mathcal{J} _{A_j}(bD ,U)$ is a sheaf.
We must have
\begin{equation}\label{eq:inclusion}
   A_j(x, \partial)\mathcal{J} _{A_j}(bD ,U)\subset   \mathcal{J}_{A_{j+1}}(bD ,U).
\end{equation}
This is because
\begin{equation*}\begin{split}
   \int_D\langle A_{j+1}(x, \partial)A_j(x, \partial)u ,\psi\rangle dV &=0=\int_D \langle u , A_j^*(x, \partial) A_{j+1}^*(x, \partial)\psi \rangle dV\\
   &=\int_D \langle A_j (x,\partial)u ,  A_{j+1}^*(x, \partial)\psi\rangle dV,
 \end{split} \end{equation*}for any $u\in \mathcal{J} _{A_j}(bD ,U)$ and compactly  supported $\psi\in \Gamma(U , E^{(j+1)} )$, by $A_{j+1} A_j=0,  A_j^* A_{j+1}^* =0$.
Setting
 \begin{equation*}
    Q^{(j)}(bD)=\frac { \Gamma( \mathbb{R}^N , E^{(j)} )}{\mathcal{J} _{A_j}(bD, \mathbb{R}^N)},
 \end{equation*}
we obtain   a quotient complex of the form
\begin{equation*}
  Q^{(0)}(bD)\xrightarrow{\widehat{A}_{0 } } Q^{(1)}(bD)\xrightarrow{\widehat{A}_{1 }} Q^{(2)}(bD)\xrightarrow{\widehat{A}_{2 }}\cdots
\end{equation*}
where  $\widehat{A}_{j }$ is induced  by the differential operator  $A_j(x, \partial)$,  but is not necessarily
a differential operator  (cf. \cite[Section 6 (d)-(g)]{Andreotti2}).

Note that when the differential operator is of first order,
 (\ref{eq:zero-Cauchy}) is satisfied if and only if
 \begin{equation*}
    \int_{bD}\langle A(x, \nu)u ,\psi\rangle dS=0
 \end{equation*}
 by Stokes' formula, where $\nu$ is the unit vector outer normal to the  boundary $bD$. Since $\psi$ is arbitrarily chosen, it is equivalent to
 \begin{equation*}
    A(x, \nu)u|_{bD}= 0
 \end{equation*}

For the $k$-Cauchy-Fueter complex, $ \mathcal{{D}}_{j}$ in \eqref{eq:operator-k-CF} with  $j\neq k$ is a differential operator of the first order, and so
$
   Q^{(j)}(bD)\cong \Gamma(bD, \ker \sigma_j^*(\nu ))
$
by   the exactness of its symbol sequence \cite[Proposition 3.2]{wang-mfd}
\begin{equation}\begin{split}
0\longrightarrow &  \mathcal{{V}}_0
\xrightarrow{\sigma_0(v )} \mathcal{{V}}_1
\xrightarrow{\sigma_1(v )} \mathcal{{V}}_2 \longrightarrow  \cdots
\xrightarrow{\sigma_{2n } (v ) }\mathcal{ {V}}_{2n+1}\longrightarrow
0\end{split}\label{eq:symbol}
\end{equation}  for any $0\neq v\in \mathbb{R}^{4 (n+1)}$,
where $\sigma_j (v )$ is the symbol of $ \mathcal{{D}}_{j}$ at $v $.

  \section{The boundary complex  for $j< k-1$   }

 \subsection{Vector spaces of the boundary complex    }
 The first step is to identify $\mathcal{J} _{j}(U):= \mathcal{J} _{ D_j}(bD,U)$   for the    $k$-Cauchy-Fueter complex. To do so, we need the formula of $ d^{A'}$ in terms of vector fields tangential to $bD$. Here and in the sequel, we extend the definition \eqref{eq:mathfrak-d} of $\mathfrak  d^{A'}$ to $  \Gamma( U,\wedge^\tau\mathbb{C}^{2n+2})\rightarrow \Gamma( U,\wedge^{\tau+ 1}\mathbb{C}^{2n+2})$ on the domain by
\begin{equation}\label{eq:mathfrak-d-2}\begin{split}
\mathfrak  d^{A'}F&= Z _{A  }^{A'}f_{ {\mathbf{\dot{A}}}} \omega^A \wedge\omega^{ {\mathbf{\dot{A}}}},
 \end{split} \end{equation} for $ F =  f_{ {\mathbf{\dot{A}}}}  \omega^{ {\mathbf{\dot{A}}}} $. Operators $\mathfrak  d^{A'}$  also satisfy the Leibnitz law: if $|\mathbf{\dot{A}}|=\tau$ and
  $G= g_{\mathbf{\dot{B}}}\omega^{\mathbf{\dot{B}}}$, we have
\begin{equation}\label{eq:Leibnitz2}\begin{aligned}\mathfrak d^{A'}(F\wedge G)=& Z _{A }^{A'}(f_{\mathbf{\dot{A}}}g_{\mathbf{\dot{B}}})~\omega^A\wedge\omega^{\mathbf{\dot{A}}}\wedge\omega^{\mathbf{\dot{B}}}\\
=& Z _{A }^{A'} f_{\mathbf{\dot{A}}} ~\omega^A\wedge\omega^{\mathbf{\dot{A}}}\wedge   g_{\mathbf{\dot{B}}}\omega^{\mathbf{\dot{B}}}+(-1)^{\tau} f_{\mathbf{\dot{A}}}\omega^{\mathbf{\dot{A}}}\wedge Z _{ {A}
}^{A'} g_{\mathbf{\dot{B}}} \omega^A\wedge\omega^{\mathbf{\dot{B}}}\\
=&\mathfrak d^{A'} F\wedge G+(-1)^{\tau}F\wedge\mathfrak d^{A'} G.
\end{aligned}\end{equation} by $\omega^A\wedge\omega^{\mathbf{\dot{A}}} =(-1)^{\tau}
\omega^{\mathbf{\dot{A}}}\wedge\omega^A $. Denote
\begin{equation*}
  \Omega^{  A '} := d^{ A '} \varrho ,\qquad
   \mathcal E  :=- {d}^{0'} {d}^{1'}\varrho
\end{equation*} By definition,
  \begin{equation}\label{eq:d-Omega0}\begin{split}
   d^{ 0'}\Omega^{ 1' } &  =  -d^{1'}\Omega^{ 0' }=  -\mathcal E,\qquad d^{A'}\Omega^{ A'  }=d^{A'} d^{A'}\varrho=0.
\end{split}\end{equation}
 \begin{prop}\label{prop:d-b} We have
 \begin{equation}\label{eq:d-b}
    d^{ A'}= \Omega^{ A'} \wedge\partial_{4n+1}+d^{ A'}_b,
 \end{equation}
  where
$
 d^{ A'}_b:=  \mathfrak{ d}^{ {A} '}f   +\Omega^{ B'}\wedge   {\mathbf T}_{ B'}^{ A '}
$ is the part     only involving vector fields tangential to the boundary $bD$, and
 \begin{equation*}
  {\mathbf T}_{ B'}^{ A '} :=R_{ B'C'} \mathbf{N}_{C'}^{A'}- \delta^{ A'}_{  B'}\partial_{4n+1} ,
\end{equation*}
with $R_{ B' C '}:= \varepsilon_{B'D'}  (\mathbf{N}\varrho)^{-1} _{D' C' } $.
 \end{prop}
 \begin{proof}
If we set  $
     \Theta^{ C' }  :={\omega}^{ 2n+  o(C')  } +   n_{B C'}\omega^{  B } ,
$ then
\begin{equation}\label{eq:OmegaA'}\begin{split}\Theta^{      C '}&=  {\omega}^{ 2n+  o(C')  } + \nabla_{B D'}\varrho\cdot (\mathbf{N}\varrho)^{-1}_{D'C' }\omega^{  B }\\&= \left[{\omega}^{ 2n+ o(B') }\mathbf{N}_{  B ' D'}\varrho+   \nabla_{BD'}\varrho\cdot\omega^{  B }   \right](\mathbf{N}\varrho)^{-1}_{D'C' } =
\\ &   =d_{D'} \varrho\, (\mathbf{N}\varrho)^{-1}_{D'C' } = \Omega^{ B' }R_{ B' C'}
\end{split}\end{equation}by \eqref{eq:n}, \eqref{eq:dA'0} and lowering index.
So for
$
   f =    f_{ \mathbf{\dot {A}} }      \omega^{  \mathbf{\dot {A}} }\in \Gamma(D,\wedge^\tau\mathbb{C}^{2n+2})
$, we can write
 \begin{equation}\label{eq:dA'-Omega}\begin{split}
  d^{ {A} '}f   =&   \left\{ Z_{  {A} }^{ A '}f_{ \mathbf{\dot {A}}  }  \omega^{  A}   +\mathbf{N}_{  C' }^{ A '} f_{ \mathbf{\dot {A}}  } \omega^{  2n+  o(C')  }
       +   n_{{A} B'}\mathbf{N}_{ B' }^{ A '}  f_{ \mathbf{\dot {A}} }  \omega^{ A} \right\} \wedge   \omega^{\mathbf{\dot {A}} } \\=&  \mathfrak{ d}^{ {A} '}f   +\Theta^{   C'  }\wedge \mathbf{N}_{  C' }^{ A '}   f \\
= &\mathfrak{ d}^{ {A} '}f   +\Theta^{   C'  }\wedge (\mathbf{N}_{  C' }^{ A '}-\mathbf{N}_{  C' }^{ A '}\varrho \,\partial_{4n+1}) f+\Theta^{  C '  }\mathbf{N}_{  C' }^{ A '}\varrho \wedge   \partial_{4n+1}f  \\
= &\mathfrak{ d}^{ {A} '}f   +\Omega^{ B'}\wedge   {\mathbf T}_{ B'}^{ A '} f+ \Omega^{  A '}\wedge\partial_{4n+1} f
\end{split}\end{equation} by  the formula \eqref{eq:dA'0} of $ d^{ {A} '}$,   the expression   \eqref{eq:Z} of $ Z_{  {A} }^{ A '}$,  \eqref{eq:OmegaA'}    and using
\begin{equation*}
   R_{ B'C '}\mathbf{N}_{  C' }^{ A '}\varrho= \varepsilon_{B'D'} (\mathbf{N}\varrho)^{-1}_{D'C' }(\mathbf{N}\varrho)_{  C'E' }\varepsilon^{E' A '}= \varepsilon_{B'D'} \varepsilon^{D' A '}= \delta^{ A'}_{  B'}.
\end{equation*}
The proposition is proved.
 \end{proof}

\begin{rem} (1) By definition, we have
\begin{equation}\label{eq:Omega-N}\left(\mathbf{N}_{A'B'}\varrho (\mathbf{q})\right)= \left(
                                      \begin{array}{rr}
                                    1& 0 \\
                         0  & 1    \end{array}
                                    \right)+O(|\mathbf{q}| ), \qquad
   \Omega^{  A '}= \omega^{  2n+  o(A')  }+O(|\mathbf{q}| ).
\end{equation}
    If the hypersurface \eqref{eq:rigid-hypersurface} is rigid, then $\overline\partial_{q_{n+1}}\varrho=1$. Thus $( \mathbf{N}_{A'B '   }\varrho)$ is the identity matrix,  and so $ R_{ B' C '} =\varepsilon_{B'C'} $. Hence
\begin{equation}\label{eq:T-AB}
   ( {\mathbf T}_{ B'}^{ A '} )=\left(
                                      \begin{array}{rr}-\textbf{i}\partial_{4n+2}&- \partial_{ 4n+3}+ \textbf{i}\partial_{4n+4}
\\ \partial_{4n+3} +\textbf{i}\partial_{4n+4} & \textbf{i}\partial_{4n+2}                                                                               \end{array}
                                    \right).
\end{equation}
\\
(2) Vector fields $Z_{  {A} }^{ A '}$'s   span  the tangential space perpendicular to the quaternionic line of the normal vector, while $ {\mathbf T} _{  {A}' }^{ B '}$'s   span  $3$-dim tangential space contained in  the quaternionic line of the normal vector.
\end{rem}

The symbol of the differential operator  $\partial_{A'} d^{ {A} '}$ at the direction $   {\operatorname{grad} \varrho} $ is
\begin{equation}\label{eq:DD}
    \mathbb{D}:=\Omega^{A'} \wedge \partial_{A'},
 \end{equation}
  because the symbol of $d^{ {A} '}$ is easily seen to be $d^{ {A} '} \varrho \wedge=\Omega^{  A '}\wedge$ by definition if we replace $\partial_{x_j}$ by $\partial_{x_j}\varrho$.
Set
  \begin{equation}\label{eq:Omega-a-j}\begin{split}
     \check{\mathbb{D}}:&=\frac 1{2 }\left(\Omega^{0'} \wedge \partial_{0'}-\Omega^{1'} \wedge \partial_{1'}\right) ,\\
    \mathbf{ \Omega}_{\sigma_j  }^{a   } :&= \check{\mathbb{D}} \mathbf{S}_{\sigma_j+1 }^{a   }  =\frac 1{2 } \left(  \mathbf{S}_{\sigma_j }^{a }  \Omega^{0'} - \mathbf{S}_{\sigma_j }^{a  -1  }  \Omega^{1'}\right),
  \end{split}  \end{equation}
where $ a=   1,\ldots,  \sigma_{j} $.
  When acting on $ \mathcal{{V}}^{\sigma,\tau} = \odot^{\sigma}\mathbb{C}^{2}\otimes
\wedge^\tau\mathbb{C}^{2n +2}$ in \eqref{eq:Vab},
  $   \mathbb{D}$ and  $ \check{  \mathbb{D}}$ are both linear transformations.
     $ \mathcal{J} _{j} $'s are determined by the following proposition.

\begin{prop} \label{prop:JW0}   Let $k$ be fixed and $j=0,\ldots,k-1$. Then,
\\
  (1)  $ \mathcal{J} _{j} (U)$   consists of
 \begin{equation}\label{eq:J-j}\begin{split}&
\mathbb{D}\mathbf{S}_{\sigma_j +1}^{a    }  \wedge f_{  a} + \varrho f,
 \end{split} \end{equation}
 where $f_a\in \Gamma( U ,\wedge^{\tau_j-1}\mathbb{C}^{2 n +2})$, $a=0,\ldots, \sigma_j+1$ and $f \in \Gamma( U ,\wedge^{\tau_j }\mathbb{C}^{2 n +2})$.
 \\
(2)  The $j$-th vector space   of the boundary complex is
  $\mathscr{V}_ j =\mathscr{V}_ j^{(1)}\oplus\mathscr{V}_ j^{(2)}$ with
   \begin{equation*} \begin{split}
       \mathscr{V}_ j^{(1)}&\cong \odot^{ \sigma_j } \mathbb{C}^{2} \otimes
\wedge^j\mathbb{C}^{2n } , \qquad
\mathscr{V}_ j^{(2)} \cong \odot^{ \sigma_{j+1}  } \mathbb{C}^{2} \otimes
\wedge^{j-1}\mathbb{C}^{2n },
  \end{split}\end{equation*}and $\mathscr{V}_0^{(2)}=\emptyset$. As  $\Gamma(U,  \mathbb{C})$-modules,
    $\Gamma(U,  \mathscr{V}_ j^{(1)})$  and $\Gamma(U,  \mathscr{V}_ j^{(2)})$  are  spanned by
  \begin{equation}\label{eq:W-j}\begin{split}&
     \mathbf{ {S}}_{\sigma_j }^{a   }   \mathbf{\omega}^{  \mathbf{{A}} },
  \end{split}\end{equation}  $a= 0,\ldots, {\sigma_j }, |   \mathbf{A}|=j,$
  and\begin{equation}\label{eq:W-j2}\begin{split}
      \mathbf{ \Omega}_{\sigma_j  }^{  b+1  }    \wedge\omega^{  \mathbf{ {B}} },
            \end{split}\end{equation}  $ b= 0,\ldots, {\sigma_j }-1=\sigma_{j+1} , $  $  |    \mathbf{B}|=j-1 $, respectively.
   \end{prop}
\begin{proof} (1) Let $\langle\cdot,\cdot\rangle$ be the   inner product    of $  \mathcal{V}_j$ by choosing $\mathbf{S}_{\sigma_j }^{a    } \mathbf{\omega}^{  \dot{\mathbf{{A}}} }$ to be an orthonormal basis.  For $f\in \Gamma(  U ,  V_j)$ and $\psi\in \Gamma(  U ,  V_{j+1})$ with compact   support, by integration by part and using Proposition \ref{prop:d-b}, we get
\begin{equation*}\begin{split}
\int_D\left\langle  \partial_{A'} d^{ {A} '}f,\psi\right\rangle dV  & =  \int_D\left\langle \partial_{A'}\left(\Omega^{A'}\wedge \partial_{4n+1}+d^ {A'}_b\right)f , \psi\right\rangle dV\\
&=-\int_{bD}\left\langle    \mathbb{D} f,\psi\right\rangle \frac {dS}{|\operatorname{grad} \varrho|}+\int_D\left\langle f,\left(\partial_{A'} d^{ {A} '}\right)^*\psi\right\rangle dV,
\end{split}
\end{equation*}
   where  $-  {\operatorname{grad} \varrho}/{|\operatorname{grad} \varrho|}$ is the unit vector outer normal to $bD$  with
    \begin{equation*}
       \operatorname{grad} \varrho(x)=\left(-\phi_{x_1}(x),\ldots,-\phi_{x_{4n}}, 1,-\phi_{x_{4n+2}},-\phi_{x_{4n+3}},-\phi_{x_{4n+4}}\right ).
    \end{equation*}    Since $d^{ A'}_b$ only involves vector fields tangential to the boundary $bD$, there is no boundary term after integration by part.
Thus, $f \in \mathcal{J} _{j}(U) $ if and only if
\begin{equation*}
   \mathbb{D} f|_{bD\cap  U }=0.
\end{equation*}

It is direct to see that
\begin{equation}\label{eq:DD-omega-0}\begin{split}
      \mathbb{D} &\mathbb{D}\mathbf{S}_{\sigma_j +1}^{a    } = \Omega^{A'} \wedge \Omega^{B'}   \partial_{A'}\partial_{B'} \mathbf{S}_{\sigma_j +1}^{a    } =0,
\end{split}\end{equation}  since the term is skew-symmetric in superscripts $ A', B'$ and is symmetric in subscripts $ A', B'$. Thus, forms in  \eqref{eq:J-j} belongs to
$ \mathcal{J} _{j} (U) $.
On the other hand,
 \begin{equation}\label{eq:omega-0}\begin{split}
           \mathbb{D}& ( \mathbf{ \Omega}_{\sigma_j  }^{b+1   } \wedge\mathbf{\omega}^{  \mathbf{B} })  =-\mathbf{S}_{\sigma_j-1}^{b   }  \Omega^{ 0'} \wedge \Omega^{ 1'}  \wedge\mathbf{\omega}^{  \mathbf{B} } \neq 0,
\end{split}\end{equation}  for  $ b=0,\ldots, \sigma_j-1 $, by \eqref{eq:partial-S}, and
\begin{equation}\label{eq:omega-0'}
\mathbb{D}(\mathbf{S}_{\sigma_j }^{b    } \mathbf{\omega}^{  \mathbf{{A}} })  =\left(\mathbf{S}_{\sigma_j-1 }^{b }  \Omega^{0'} + \mathbf{S}_{\sigma_j -1}^{b  -1 }  \Omega^{1'}\right)\wedge\mathbf{\omega}^{  \mathbf{{A}} }.
\end{equation}
It is direct to see that terms in the right hand sides of \eqref{eq:omega-0}-\eqref{eq:omega-0'} are linearly independent. So
the linear combination of $\mathbf{ \Omega}_{\sigma_j  }^{b+1   }  \wedge\mathbf{\omega}^{  \mathbf{B} }$ and $\mathbf{S}_{\sigma_j }^{b   } \mathbf{\omega}^{  \mathbf{{A}} }$ with coefficients in $\Gamma(U,  \mathbb{C})$
can not be annihilated by $\mathbb{D}$.

 On the other hand, since
$ \{\Omega^{ 0'}, \Omega^{1'},  \omega^{0} ,\ldots ,\omega^{2n-1}\}$ locally is also a basis of $ \mathbb{C}^{2 (n+1) }$ by \eqref{eq:Omega-N}, we see that $ \Gamma(U,  \odot^{\sigma_j } \mathbb{C}^{2} \otimes
\wedge^j\mathbb{C}^{2n+2})$ has a basis consisting of
\begin{equation}\label{eq:basis}
    \mathbf{S}_{\sigma_j }^{a  }   \omega^{\mathbf{ {A}}}, \qquad  \mathbf{S}_{\sigma_j }^{a  } \,\Omega^{A'} \wedge \omega^{  \mathbf{  {B}} }, \qquad  \mathbf{S}_{\sigma_j }^{a  } \,  \Omega^{0'} \wedge\Omega^{1'} \wedge \omega^{ \mathbf{  {C}} },
\end{equation}
where  $| \mathbf{  {A}}|=| \mathbf{  {B}}|+1= | \mathbf{  {C}}|+2=j   $, $a=0,1,\ldots, \sigma_j $. But
\begin{equation}\label{eq:Omega-Omega}\mathbf{S}_{\sigma_j }^{a  } \Omega^{0'} \wedge \Omega^{ 1'} \wedge\omega^{ \mathbf{C}}  \in\mathcal{J} _{j} (U),
\end{equation}by \eqref{eq:omega-0},
    and
    \begin{equation}\label{eq:Omega-Omega01}
   \mathbf{S}_{\sigma_j }^{a  } \Omega^{A'}\wedge\omega^{  \mathbf{ {B}} }=   ( -1)^{ o(A')   }\mathbf{ \Omega}_{\sigma_j  }^{a +o(A')   }  \wedge\omega^{  \mathbf{ {B}} },\qquad {\rm mod }\quad    \mathcal{J}_{j } (U).
\end{equation}
    Thus $ \Gamma(U,  \odot^{\sigma_j } \mathbb{C}^{2} \otimes
\wedge^j\mathbb{C}^{2n+2})$  mod $\mathcal{J} _{j } (U)$ is spanned by terms in  \eqref{eq:W-j}-\eqref{eq:W-j2}, and so $\mathcal{J} _{j } (U)$ is spanned by terms in  \eqref{eq:J-j}.  \eqref{eq:Omega-Omega01} holds because
    \begin{equation*}\label{eq:Omega-Omega1}
   \mathbf{S}_{\sigma_j }^{a  } \Omega^{1'}\wedge\omega^{  \mathbf{ {B}} }= \Omega^{ 1'}\partial_{1'}\mathbf{S}_{\sigma_j+1 }^{a +1 } \wedge\omega^{  \mathbf{ {B}} } = -\mathbf{ \Omega}_{\sigma_j  }^{a+1    }  \wedge\omega^{  \mathbf{ {B}} },\qquad {\rm mod }\quad    \mathcal{J}_{j } (U),
\end{equation*}
for $a=0,\ldots, \sigma_j-1$,    and similarly,
$
   \mathbf{S}_{\sigma_j }^{a  } \Omega^{0'}\wedge\omega^{  \mathbf{ {B}} } = \mathbf{ \Omega}_{\sigma_j  }^{a    }  \wedge\omega^{  \mathbf{ {B}} } $ mod  $\mathcal{J}_{j } (U).
$
The proposition is proved.
 \end{proof}

\begin{rem}\label{rem:D-checkD} The indices $b$ in $\mathbf{ \Omega}_{\sigma_j  }^{b+1   } $ \eqref{eq:W-j2}
    are only taken over     $   1,\ldots, {\sigma_j } -1  $,  because
    \begin{equation*}
      2 \mathbf{ \Omega}_{\sigma_j  }^{0 } = \mathbf{S}_{\sigma_j }^{0   }\Omega^{0'}  = {\mathbb{D}}\mathbf{S}_{\sigma_j+1 }^{0    } , \qquad 2\mathbf{ \Omega}_{\sigma_j  }^{ \sigma_j +1   }    = -  \mathbf{S}_{\sigma_j }^{\sigma_j     } \Omega^{1'} =   -  {\mathbb{D}}\mathbf{S}_{\sigma_j+1 }^{\sigma_j+1   }\in  \mathcal{J}_{j }(U).
    \end{equation*}
\end{rem}

  \begin{lem}\label{lem:partial-Omega0} For $j=0,1,\cdots,k-1 $,
     \begin{equation}\label{eq:partial-Omega0}\begin{split}
     \partial_{0'}  \mathbf{ \Omega}_{\sigma_j   }^{a+1   }   &=   \mathbf{ \Omega}_{\sigma_j -1 }^{a+1   }  , \qquad \partial_{1'} \mathbf{ \Omega}_{\sigma_j  }^{a+1   }   =  \mathbf{ \Omega}_{\sigma_j -1 }^{a    }  , \qquad  \partial_{A'}  d^{A'} \mathbf{ \Omega}_{\sigma_j  }^{a+1   }   =  \mathcal E  \mathbf{S}^{a   }_{ \sigma_{j+1 }}.
 \end{split} \end{equation}
  \end{lem}
\begin{proof} The first two identities follows from definition \eqref{eq:Omega-a-j}, while for the third one, we have
\begin{equation*}
    \partial_{A'}  d^{A'} \mathbf{ \Omega}_{\sigma_j  }^{a+1   }  =\frac 1{2 } \left(  \partial_{1'} \mathbf{S}_{\sigma_j }^{a+1 }  d^{1'}\Omega^{0'} - \partial_{0'} \mathbf{S}_{\sigma_j }^{a    }  d^{0'} \Omega^{1'}\right)=  \mathcal E  \mathbf{S}^{a   }_{ \sigma_{j }-1},
\end{equation*}by \eqref{eq:d-Omega0}   and $  \sigma_{j+1 }= \sigma_{j }-1$.
 \end{proof}
 \subsection{The induced action of operators $ \mathcal{D }_{j}$'s on  the boundary complex} By Proposition \ref{prop:JW0}, an element of $\Gamma(U, \mathcal{  V}_ j  )$ modulo $ \mathcal{J} _{j }(U)$ has the form
 \begin{equation*}
    f_{  a  }  \mathbf{S}^{a}_{\sigma_j }   +  \mathbf{ \Omega}_{\sigma_j  }^{ b+1   }    \wedge G_{ b  },
 \end{equation*}for some $f_a     \in  \Gamma(U,   \wedge^{\tau_j}\mathbb{C}^{2n  } ), G_c     \in  \Gamma( U,     \wedge^{{\tau_j}-1}\mathbb{C}^{2n } )  $.
We need to calculate the action of $ \partial_{A'}d^{A'} $ on such elements modulo $ \mathcal{J} _{j+1 }(U)$.  Here and in  the sequel, we will use the convention that  the summation of repeated indices $a$ in $ \mathbf{S}^{a}_{\sigma_j }   $  is taken over
 $a= 0,\ldots, {\sigma_j }$, while the summation of repeated indices $b$ in $ \mathbf{ \Omega}_{\sigma_j  }^{ b+1   }   $  is taken over
 $b= 0,\ldots, {\sigma_j }-1=\sigma_{j+1}$.    We can write
 \begin{equation}\label{eq:E=}
 \mathcal{  E}=\mathcal{  E}_0+\Omega^{ A'} \wedge \mathcal{  E}_{A'}+ \mathcal{  E}_{ 0'1'}\Omega^{ 0'} \wedge \Omega^{ 1'}
 \end{equation}
for some $\mathcal{  E}_0, \mathcal{  E}_{ 0'},\mathcal{  E}_{ 1'}$ and $\mathcal{  E}_{ 0'1'}$ to be $2$-, $1$-, $1$- and $0$-forms only involving $\omega^A$.
Denote
\begin{equation}\label{eq:RT+-}\begin{split}\mathscr{R}_{c  }^\pm :&=
\mathfrak{ d}^{A '} f_{  c\pm  o(A')   }        + \mathcal E_0  \wedge G_{ c  },\\
 \mathscr{T}_c^{\pm} :&=  -\mathfrak{ d}^{A '} G_{ c\pm  o(A') }-\mathcal{  E}^{A'} \wedge G_{c\pm  o(A')} - \mathbf{ T }^{A ' B '} f_{ c\pm  o(A'B ')  } ,
\end{split}\end{equation}where $ o(A'B ')= o(A' )+ o( B ')$. By raising indices,
\begin{equation*}\label{eq:T-T}\begin{split}
  \mathbf{ T }^{0' 0 '}  = {\mathbf T}_{1'}^{ 0 '},\qquad \mathbf{ T }^{ 0 '1'}= - {\mathbf T}_ { 0'}^{ 0 '}  ,\qquad   \mathbf{T}^{ 1'  0 '}  ={\mathbf T}_{1'}^{ 1 '} ,\qquad \mathbf{ T }^{1' 1 '}   =-{\mathbf T}_ { 0'}^{ 1 '} .
\end{split}\end{equation*}

\begin{prop} \label{prop:JZ 0} If $j=0,1,\cdots, k-2,$
we have
    \begin{equation*} \begin{split}
 \partial_{A'}d^{A'}& \left[    f_{  a  }  \mathbf{S}^{a}_{\sigma_j }   + \mathbf{ \Omega}_{\sigma_j  }^{ b+1   }  \wedge G_{ b  } \right]
= \mathscr{R}_{b  }^+\mathbf{S}^{b  }_{{  \sigma_{j+1 }  }}   + \mathbf{ \Omega}_{\sigma_{j+1}  }^{ c+1   }  \wedge   \mathscr{T}_{c }^+
      \quad {\rm mod}  \quad \mathcal{J} _{j+1} (U).
\end{split}\end{equation*}
 \end{prop}
 \begin{proof}     Note that
\begin{equation*}\begin{split}
   \partial_{A'}     d^{A'} \mathbf{ \Omega}_{\sigma_j  }^{ b+1   }&  = \mathbf{S}^{b  }_{\sigma_{j+1}} \, \mathcal E
     =   \mathcal E_0  \mathbf{S}^{b  }_{ \sigma_{j+1 }}+ \mathbf{\Omega}_{\sigma_{j+1 }  }^{b  } \wedge \mathcal{  E}_{ 0'}-  \mathbf{\Omega} _{\sigma_{j+1 }  }^{b +1   }  \wedge \mathcal{  E}_{ 1'}\qquad  {\rm mod}  \quad \mathcal{J} _{j+1} (U).
\end{split}\end{equation*}by  using Lemma \ref{lem:partial-Omega0}, \eqref{eq:Omega-Omega01} and \eqref{eq:E=}.
By using the formula \eqref{eq:dA'-Omega} of $ d^{ {A} '}$  and Lemma \ref{lem:partial-Omega0} again, we get
  \begin{equation*} \begin{split}
  \partial_{A'}d^{A'}\left[  \mathbf{ \Omega}_{\sigma_j  }^{ b+1   } \wedge   G_{b }   \right]
  = &\partial_{A'}     d^{A'}\mathbf{ \Omega}_{\sigma_j  }^{ b+1   }  \wedge   G_{b }   -  \partial_{A'}   \mathbf{ \Omega}_{\sigma_j  }^{ b+1   }  \wedge d^{A'}  G_{b }    \\
 =&  \mathbf{S}^{b   }_{\sigma_{j+1}} \, \mathcal E_0 \wedge G_{b } -  \mathbf{ \Omega}_{\sigma_{j+1 }  }^{ b  +1  }  \wedge\left(\mathfrak{ d}^{ 0 '} G_{b } + \mathcal{  E}_{ 1'} \wedge G_{b } \right) - \mathbf{ \Omega}_{\sigma_{j+1 } }^{ b   }  \wedge\left(\mathfrak{ d} ^{ 1 '} G_{b } -\mathcal{  E}_{ 0'} \wedge G_{b } \right)
\end{split}\end{equation*}   mod $\mathcal{J} _{j+1}  (U)   $, since
\begin{equation}\label{eq:Omega-Omega=0}
  \mathbf{ \Omega}_{\sigma_{j+1 }  }^{ b    }\wedge\Omega^{ A'}   \wedge\cdot =0  \qquad  {\rm mod}  \quad \mathcal{J} _{j+1} (U),
   \end{equation}by
 (\ref{eq:Omega-Omega}). Now
  by  relabeling    and raising indices, we get
 \begin{equation*}  \begin{split}
 \partial_{A'}d^{A'}\left[  \mathbf{ \Omega}_{\sigma_{j+1}   }^{ b+1   }  \wedge G_{ b   }    \right]=  &    \mathbf{{S}}^{b  }_{ {\sigma_{j+1} } }\mathcal E_0 \wedge G_{b }-   \mathbf{ \Omega}_{\sigma_{j+1}  }^{ c+1   }  \wedge  \left[\mathfrak{ d}^{A '} G_{ c+ o(A') }+\mathcal{  E}^{A'} \wedge G_{c+  o(A')}\right]\quad  {\rm mod}  \quad \mathcal{J} _{j+1} (U).
\end{split}\end{equation*}

On the other hand,
 \begin{equation}\label{eq:4n+1}
  \partial_{A'}\Omega^{  A '}\wedge\partial_{4n+1}  (f_{  a  }  \mathbf{S}^{a}_{\sigma_j }   ) = \mathbb{D}\mathbf{S}_{\sigma_j }^{a    }\wedge \partial_{4n+1}   f_{  a  } \in \mathcal{J} _{j+1} (U),
 \end{equation}
   and so
 \begin{equation*}\label{eq:Dj-2}\begin{split} \partial_{A'}d^{A'} & \left[    f_{  a  }  \mathbf{S}^{a}_{\sigma_j }     \right]= \partial_{A'} \mathbf{S}^{a}_{\sigma_j }  \left(\mathfrak{ d}^{ A '} f_{  a  }     +    \Omega^{ B'} \wedge  {\mathbf T}_{  B'}^{ A '} f_{ a   } \right)\\& =   \mathbf{S}^{a- o(A')}_{\sigma_j-1 }  \mathfrak{ d}^{ A '} f_{  a  }     + \mathbf{S}^{a }_{ \sigma_{j  }-1 }\left ( \Omega^{ 0'}  \wedge {\mathbf T}_{0'}^{ 0 '} f_{  a  } +     \Omega^{ 1'}  \wedge {\mathbf T}_{1'}^{ 0 '} f_{ a   }\right) +\mathbf{S}^{a-1 }_{ \sigma_{j  }-1 }   \left (\Omega^{ 0'}  \wedge{\mathbf T}_{0'}^{ 1 '} f_{ a  } +      \Omega^{ 1'} \wedge{\mathbf T}_{1'}^{ 1 '} f_{  a  } \right)\\
 &=  \mathbf{S}^{a}_{\sigma_{j+1 } }  \mathfrak{ d}^{ A '} f_{  a + o(A') }     +  \mathbf{ \Omega}_{\sigma_{j+1 }  }^{    a  }   \wedge \left({\mathbf T}_{ 0 '} ^{ 0 '} -{\mathbf T}_{1'}^{ 1 '}\right ) f_{  a   }  -  \mathbf{ \Omega}_{\sigma_{j+1 }  }^{   a+1 }   \wedge {\mathbf T}_{1'}^{ 0 '}  f_{  a   }     +  \mathbf{ \Omega}_{\sigma_{j+1 }  }^{   a-1 } \wedge {\mathbf T}_{0'}^{ 1 '} f_{  a  }       \\
 &=  \mathbf{S}^{a}_{\sigma_{j+1 } }  \mathfrak{ d}^{ A '} f_{  a + o(A') }   - \mathbf{ \Omega}_{\sigma_{j+1 }  }^{c  +1    }    \wedge \left[\left({\mathbf T} ^{0 ' 1 '} +{\mathbf T} ^{ 1 '0'}\right ) f_{ c +1  }  + {\mathbf T} ^{ 0 '0'}  f_{ c }     +   {\mathbf T} ^{ 1 '1'} f_{  c +2 } \right]  \quad {\rm mod}   \quad\mathcal{J} _{j+1} (U)
\end{split}\end{equation*}
by using (\ref{eq:dA'-Omega}),   \eqref{eq:Omega-Omega01},   Proposition  \ref{prop:JW0} (1),  and relabeling indices. Where $c$  is taken over
 $  0,\ldots,  \sigma_{j+1}-1=\sigma_{j+2 }$. The result follows from the sum of the above identities.
\end{proof}
 \subsection{Proof of theorem \ref{thm:bd-complex-main:1}   for $j=0,1,\cdots,k-1$ }  To write down   operators  ${\mathscr{D} }_{j}$ in  the boundary complex,     define    isomorphisms
  \begin{equation}\label{eq:Pi-0}\begin{split}
     \mathbf{\Pi}_j  :  \Gamma(U, {V}_ j  ) / \mathcal{J} _{j }(U)& \longrightarrow \Gamma( bD\cap U,\mathscr{V}_ j^{(1)}  )\oplus \Gamma( bD\cap U, \mathscr{V}_ j^{(2)}  ),\\
     f_{  a  }  \mathbf{S}^{a}_{\sigma_j }   + \mathbf{ \Omega}_{\sigma_{j  }  }^{  b +1 }      \wedge G_{b } &\mapsto \left(f_{  a  }  \mathbf{S}^{a}_{\sigma_j } ,   G_{ b}\mathbf{S}^{b }_{{  \sigma_{j+1 }  }}\right),
 \end{split} \end{equation}where the first summation  is taken over
 $a= 0,\ldots, {\sigma_j }$, while the second one is taken over
 $b= 0,\ldots,  \sigma_{j+1 } $.

    The operator ${\mathscr{D} }_{j}: \Gamma(bD,\mathscr{V}_ j ) \longrightarrow  \Gamma(bD,\mathscr{V}_{j+1} ) $ is given by
  \begin{equation}\label{eq:Pi-Q1}
     {\mathscr{D} }_{j}=  \mathbf{\Pi}_{j+1 }\circ \partial_{A'}d^{A'}\circ \mathbf{\Pi}_j^{-1} \qquad {\rm mod} \quad \mathcal{J} _{j+1 }(U).
  \end{equation}
  By Proposition \ref{prop:JZ 0},
      we get
 \begin{equation*}\begin{split}
 {\mathscr{D} }_{j}\left(f_{  a  }  \mathbf{S}^{a}_{\sigma_j } ,   G_{ b}\mathbf{S}^{b }_{{  \sigma_{j+1 }  }}\right)&=  \mathbf{\Pi}_{j+1 }\left(\partial_{A'}d^{A'}  \left[    f_{  a  }  \mathbf{S}^{a}_{\sigma_j }   +  \mathbf{ \Omega}_{\sigma_{j  }  }^{   b +1    }     \wedge G_{b} \right]\right)\\&
=  \mathbf{\Pi}_{j+1 }\left( \mathscr{R}_{b  }^+ \mathbf{S}^{b  }_{{  \sigma_{j+1 }  }}  +  \mathbf{ \Omega}_{\sigma_{j+1 }  }^{    c  +1  }   \wedge   \mathscr{T}_c^+\right)\\&
= \left ( \mathscr{R}_{b  }^+ \mathbf{S}^{b  }_{{  \sigma_{j+1 }  }} ,  \mathscr{T}_{c  }^+\mathbf{S}^{c }_{{  \sigma_{j+2 }  }}\right)
      ,
 \end{split} \end{equation*}where $c$  is taken over
 $  0,\ldots,  \sigma_{j+2}$. On the other hand, we have
 \begin{equation}\label{eq:RT+}\begin{split} \mathscr{R}_{b  }^+
  \mathbf{S}^{b  }_{{  \sigma_{j+1 }  }}& =  \partial_{A'}\mathfrak  d^{A '}  \left(f_{  a  }  \mathbf{S}^{a}_{\sigma_j } \right)  +\mathcal E_0\wedge  \left ( G_{ b}\mathbf{S}^{b}_{{  \sigma_{j+1 }  }} \right),\\
  \mathscr{T}_{c  }^+\mathbf{S}^{c }_{{  \sigma_{j+2 }  }} &= - \partial_{A'} \mathfrak  d^{A'}  \left (G_{ b}\mathbf{S}^{b }_{{  \sigma_{j+1 }  }}\right)-\mathcal{  E}^{A'} \wedge \partial_{A'}\left  (G_{b }  \mathbf{S}^{b }_{{  \sigma_{j+1 }  }}\right)-\partial_{A'}\partial_{B'} \mathbf{T }^{A' B '} \left(f_{  a  }  \mathbf{S}^{a}_{\sigma_j } \right)
 \end{split} \end{equation}by direct differentiation  with respect to $s^{A'}$,
and formulae \eqref{eq:RT+-} of $\mathscr{R}_{b  }^+$ and $
 \mathscr{T}_c^+$.
 Thus, if taking $\mathbb{ F}_1   =f_{  a  }  \mathbf{S}^{a}_{\sigma_j } $ and $\mathbb{ F}_2   =G_{ b}\mathbf{S}^{b}_{{  \sigma_{j+1 }  }}$, we see that
  $ {\mathscr{D} }_{j}$ is given by  \begin{equation}\label{eq:bdoperator<}\begin{split}{\mathscr{D} }_{j}\mathbb{ F}   =&\left(
 \partial_{A'}\mathfrak  d^{A '} \mathbb{F}_1  +\mathcal E_0\wedge    \mathbb{F}_2 ,  -\partial_{A'} (\mathfrak  d^{A'}+\mathcal{{  E}}^{A'}\wedge ) \mathbb{F}_2-\partial_{A'}\partial_{B'} \mathbf{T }^{A' B '} \mathbb{F}_1
\right) .
\end{split} \end{equation}

\section{The boundary complex  for $j\geq k-1$   }
\subsection{Vector spaces of the boundary complex   }
 For $j=k+1,\cdots,2n+1$,  let
  \begin{equation}\label{eq:Omega-check-D}
    \widetilde{\mathbb{D}}:=s_{A'} \Omega^{A'}\wedge
  \end{equation} which is the symbol of the differential operator  $s_{A'} d^{ {A} '}$ at the direction $   {\operatorname{grad} \varrho} $,
and
 \begin{equation}\label{eq:Omega-a-p}
     \mathbf{\widetilde{{\Omega}}}^{b }_{\sigma_j} := \frac 12\left(  \widetilde{ \mathbf{{S}}}_{\sigma_j }^{b -1  } \Omega^{ 0'}-      \widetilde{\mathbf{{S}}}_{\sigma_j }^{b }  \Omega^{1'} \right),
  \end{equation}where  $ b=   0,\ldots, \sigma_j +1 . $ Unlike $\mathbf{  \Omega }^{a }_{\sigma_j}$, elements $ \mathbf{\widetilde{{\Omega}}}^{0 }_{\sigma_j}:=- \frac 12  \widetilde{ {S}}_{\sigma_j }^{0  } \Omega^{1'}$ and
  $ \mathbf{\widetilde{{\Omega}}}^{\sigma_j +1 }_{\sigma_j}:= \frac 12  \widetilde{ \mathbf{{S}}}_{\sigma_j }^{\sigma_j    } \Omega^{ 0'}$ are not in $\ker\widetilde{\mathbb{D}}$.   It is direct to check  that
  \begin{lem}\label{lem:partial-Omega}
     \begin{equation}\label{eq:partial-Omega}\begin{split}
        s_{0'} \mathbf{\widetilde{ {\Omega}}}^{b}_{\sigma_j } &=  \mathbf{\widetilde{{\Omega}}}^{b }_{\sigma_{j+1 } }, \qquad s_{1'} \mathbf{\widetilde{ {\Omega}}}^{b }_{\sigma_j } = \mathbf{ \widetilde{{\Omega}}}^{b+1}_{\sigma_{j+1 } },  \qquad
     s_{A'}  d^{A'}\mathbf{\widetilde{ {\Omega}}}^{b }_{\sigma_j }   =  \mathcal E  \widetilde{ \mathbf{S}}^{b  }_{ \sigma_{j+1 }}.
 \end{split} \end{equation}
  \end{lem}

\begin{prop} \label{prop:JW} For fixed $k$, we have
  (1)
 $ \mathcal{J} _{k+1} (U)$   consists of
 \begin{equation}\label{eq:J-j'}\begin{split}& \Omega^{ 0'} \wedge \Omega^{ 1'}\wedge  f   + \varrho F  ,
 \end{split} \end{equation}where $f  \in \Gamma( U ,\wedge^{k}\mathbb{C}^{2n  })$ and $F \in \Gamma(U,\wedge^{k+2 }\mathbb{C}^{2n+2 })$.
\\
 (2) if $j=k+2,\ldots,2n-1$,  $ \mathcal{J} _{j}(U) $   consists of
 \begin{equation}\label{eq:J-j''}\begin{split}& \widetilde{\mathbb{D}}   \widetilde{ \mathbf{{S}}}_{\sigma_j-1 }^{b   }  \wedge f_{  b} + \varrho F,
 \end{split} \end{equation} where  $f_{  b} \in \Gamma(U ,\wedge^{\tau_j-1}\mathbb{C}^{2n+2 })$, $F \in \Gamma(U ,\wedge^{\tau_j }\mathbb{C}^{2n +2})$, and $b=  0,\ldots, \sigma_{j}-1 $.
\\
 (3)     $ \mathcal{J} _{k } (U)$   consists of
 \begin{equation}\label{eq:J-k+1}\begin{split}& \Omega^{ A'}\wedge f_{ A'} +\Omega^{ 0'} \wedge\Omega^{ 1'}\wedge g
+  \varrho  \left( \mathfrak{d}^{ A'}  f_{ A'}  - \mathcal E_0 \wedge g \right ) ,
 \end{split} \end{equation}
where  $f_{ A'}\in \Gamma(U ,\wedge^{k-1}\mathbb{C}^{2n })$  and   $g   \in \Gamma(U ,\wedge^{k-2}\mathbb{C}^{2n })$. If $k=1$, then $g=0$.
 \end{prop}
\begin{proof} (1)
For $f \in \Gamma(  U ,  \mathcal{V}_j)$ and $ \psi\in \Gamma(  U , \mathcal{V}_{j+1})$  with compact   support, we have \begin{equation*}\begin{split}
\int_D\left\langle  s_{A'} d^{ {A} '}f,\psi\right\rangle dV  & =  \int_D\left\langle s_{A'}\left(\Omega^{A'}\wedge \partial_{ {4n+1}}+d^{ 0'}_b\right)f ,\psi\right\rangle dV\\
&=-\int_{bD}\left\langle  \widetilde{\mathbb{D}}  f,\psi\right\rangle \frac {dS}{|\operatorname{grad} \varrho|}+\int_D\left\langle f,\left(s_{A'} d^{ {A} '}\right)^*\psi\right\rangle dV,
\end{split}
\end{equation*} where $ (s_{A'} d^{ {A} '} )^*$ is the formal adjoint of $s_{A'} d^{ {A} '}$. Thus,  $f \in \mathcal{J} _{j} (U)$ if and only if
 \begin{equation}\label{eq:J-bD}
    \widetilde{\mathbb{D}}  f|_{bD\cap U}=0.
 \end{equation}

 When $j=k+ 1$, we have $\sigma_{k+ 1}=0$.
 The condition \eqref{eq:J-bD} is equivalent to $\Omega^{ 0'}    f =\Omega^{ 1'}  \wedge f =0$ on the boundary for $f  \in \Gamma(U,\wedge^{k+2}\mathbb{C}^{2n })$. Thus,  $ \mathcal{J} _{k+ 1} (U)$ consists of elements of the form (\ref{eq:J-j'}).

(2) For $j=k+2,\ldots,2n-1$, we have $\sigma_j=1,2,\ldots$, and it is direct to see that
\begin{equation}\label{eq:tilde-DDS}\begin{split}&
    \widetilde{\mathbb{D}} \widetilde{\mathbb{D}}     \widetilde{ \mathbf{{S}}}_{\sigma_j -1 }^{b } =0,\qquad \qquad\qquad  \quad\qquad b=0,\ldots, \sigma_j-1,\\&
    \widetilde{\mathbb{D}}   \mathbf{\widetilde{{\Omega}}}^{c }_{\sigma_j}=-\widetilde{\mathbf{S}}_{\sigma_j +1}^{c   }  \Omega^{ 0'} \wedge \Omega^{ 1'} \neq 0,\qquad  c=0,\ldots, \sigma_j+1,
\end{split}\end{equation}
   by \eqref{eq:Omega-a-p} and $\widetilde{\mathbb{D}} \widetilde{\mathbb{D}}   =0$ as in \eqref{eq:DD-omega-0}. Thus $\widetilde{\mathbb{D}}     \widetilde{ \mathbf{{S}}}_{\sigma_j -1 }^{b } =\widetilde{ \mathbf{{S}}}_{\sigma_j }^{b   } \Omega^{ 0'}+     \widetilde{\mathbf{{S}}}_{\sigma_j }^{b +1}  \Omega^{1'} \in\mathcal{J} _{j}(U) $. Note that
 \begin{equation}\label{eq:tilde-Omega-Omega01}
   \widetilde{\mathbf{S}}_{\sigma_{j  }}^{b  } \Omega^{A'}\wedge\omega^{  \mathbf{ {B}} }=   ( -1)^{ o(A')   }\widetilde{\mathbf{ \Omega}}_{\sigma_{j }  }^{b+1-o(A')   }  \wedge\omega^{  \mathbf{ {B}} },\qquad {\rm mod }\quad    \mathcal{J} _{j}(U),
\end{equation}because
  \begin{equation*}\label{eq:S-Omega}\begin{split}
     \widetilde{\mathbf{{S}}}^{b }_{\sigma_{j  } }  \Omega^{ 0'}=  \mathbf{\widetilde{{\Omega}}}^{ b+1 }_{  \sigma_{j  }  },\qquad   \widetilde{\mathbf{{S}}}^{b }_{\sigma_{j  } }  \Omega^{ 1'}= - \mathbf{\widetilde{{\Omega}}}^{ b }_{  \sigma_{j  }  } \qquad {\rm mod}  \quad  \mathcal{J} _{j}(U) ,
\end{split}  \end{equation*}by definition.
      Thus $ \mathcal{J} _{j}(U) $ consists of elements in (\ref{eq:J-j''}) as in the case $j<k$ in Proposition \ref{prop:JW0} .

(3) Since  $\partial_{ {4n+1}}\Omega^{A'}=0$, we have that for $f =    f_{ \mathbf{\dot {A}} }      \omega^{  \mathbf{\dot {A}} }\in \Gamma(U,\wedge^{k }\mathbb{C}^{2n+2 })$,
\begin{equation*}
 d^{ 0'}d^{ 1'}f=  \Omega^{ 0'}\wedge\Omega^{ 1'} \wedge\partial_{ {4n+1}}^2 f +d^{ 0'}_b (\Omega^{ 1'}\wedge\partial_{4n+1}f) +\Omega^{ 0'} \wedge d^{ 1'}_b \partial_{4n+1}f+ d^{ 0'}_b  d^{ 1'}_bf
\end{equation*}by  using \eqref{eq:d-b} twice.   Recall that there is no boundary term for $d^{ 0'}_b$ after integration by part. We get
\begin{equation*}\begin{split}
\int_D\left\langle d^{ 0 '}  d^{ 1 '} f,\psi\right\rangle &dV= \hskip 2mm  \int_D\left\langle  f, (d^{ 0 '}d^{ 1 '})^*\psi\right\rangle dV -\int_{bD}\left\langle   d^{ 0'}_b (\Omega^{ 1'}\wedge f) +\Omega^{ 0'} \wedge d^{ 1'}_b   f,\psi\right\rangle \frac {dS}{|\operatorname{grad} \varrho|}\\& +\int_{bD}\left\langle\Omega^{ 0'}\wedge\Omega^{ 1'} \wedge f,\partial_{ {4n+1}} \psi\right\rangle \frac {dS}{|\operatorname{grad} \varrho|}-\int_{bD}\left\langle  \Omega^{ 0'}\wedge\Omega^{ 1'}  \wedge \partial_{ {4n+1}}f, \psi\right\rangle \frac {dS}{|\operatorname{grad} \varrho|},
\end{split}
\end{equation*}
by integration by part twice.
Thus $f \in \mathcal{J} _{k} (U)$ if and only if
\begin{equation}\label{eq:condition-Jk}\begin{split}
\left.\Omega^{ 0'}\wedge\Omega^{ 1'} \wedge f\right|_{bD} &=0,\\
 \left. d^{ 0'} (\Omega^{ 1'}\wedge f)  \right|_{bD} +\left.\Omega^{ 0'} \wedge d^{ 1'}_b    f    \right|_{bD} & =0,
\end{split}\end{equation}
by definition of $d^{ 0'}_b$.
By the first equation, we can write
\begin{equation*}
   f=\Omega^{ A'}\wedge f_{ A'}+\Omega^{ 0'} \wedge\Omega^{ 1'}\wedge g +\varrho F+O(\varrho^2)
\end{equation*}
for some $f_{ A'} ,g, F$ independent of $\varrho$ and valued in $\wedge^{*}\mathbb{C}^{2n }$.  Then,
\begin{equation*}\begin{split}
    \left. d^{ 0'} (\Omega^{ 1'}\wedge f)\right|_{bD}  &= \left. {d}^{ 0'}  (  \Omega^{ 1'}\wedge \Omega^{ 0'}\wedge f_{ 0'}+\varrho \Omega^{ 1'}\wedge F )\right|_{bD}\\&=  -\mathcal E\wedge \Omega^{ 0'}\wedge f_{ 0'}-\Omega^{ 0'}\wedge\Omega^{ 1'}\wedge \mathfrak{d}^{ 0'} f_{ 0'} +\Omega^{ 0'}\wedge   \Omega^{ 1'}\wedge F,
 \end{split}\end{equation*}by   using the Leibnitz law and \eqref{eq:dA'-Omega}. Similarly,
 \begin{equation*}\begin{split}
   \left. \Omega^{ 0'} \wedge d^{ 1'}_b    f\right|_{bD}&=\left. \Omega^{ 0'} \wedge d^{ 1'}  f\right|_{bD}-\left.\Omega^{ 0'}\wedge \Omega^{ 1'}\wedge \partial_{{4n+1}}f\right |_{bD}\\&=\Omega^{ 0'} \wedge d^{ 1'} \left.\left(\Omega^{ A'} \wedge f_{ A'}  +\Omega^{ 0'} \wedge\Omega^{ 1'}\wedge g +\varrho F\right)\right|_{bD}- \Omega^{ 0'}\wedge \Omega^{ 1'}\wedge F
    \\&=  \Omega^{ 0'}\wedge \mathcal E\wedge f_{ 0'} -\Omega^{ 0'}\wedge\Omega^{ 1'}\wedge \mathfrak{d}^{1'} f_1+ \Omega^{ 0'}\wedge\Omega^{ 1'} \wedge\mathcal E\wedge g .
 \end{split}\end{equation*}Their sum gives us that the second equation in \eqref{eq:condition-Jk} is equivalent to
\begin{equation*}\left.
    \Omega^{ 0'} \wedge \Omega^{ 1'}\wedge \left( F-\mathfrak{d}^{ A'}  f_{ A'} +\mathcal E \wedge g \right)\right|_{bD}  =0,
\end{equation*}
i.e. $F=\mathfrak{d}^{ A'}  f_{ A'}   - \mathcal E_0 \wedge g $ mod $\Omega^{ 0'}, \Omega^{ 1'}$. The result follows.
\end{proof}

By Proposition \ref{prop:JW} and its proof, we know  vector spaces of the boundary complex.
\begin{cor}\label{cor:W-bundle}
(1)
for $j=k+1,\ldots,2n-1$,
 \begin{equation*} \begin{split}
  \mathscr{V}_ j^{(1)}\cong \odot^{ \sigma_j } \mathbb{C}^{2} \otimes
\wedge^{j+1}\mathbb{C}^{2n } , \qquad
  \mathscr{V}_ j^{(2)}\cong \odot^{ \sigma_{j+1} } \mathbb{C}^{2} \otimes
\wedge^j\mathbb{C}^{2n } ,
\end{split}\end{equation*}
 are   spanned by\begin{equation}\label{eq:W-j2'}\begin{split}&
    \widetilde{ {\mathbf{{S}}}}_{\sigma_j }^{a   }   \omega^{  \mathbf{{B}} } ,   \qquad  \mathbf{\widetilde{{\Omega}}}^{b }_{\sigma_j} \wedge\omega^{  \mathbf{ {C}} },
            \end{split}\end{equation}respectively,  where $a=0,\ldots, {\sigma_j }$, $b= 0,  \ldots,\sigma_{j+1}$, $|    \mathbf{B}|=j+1$, $|    \mathbf{C}|=j $.

  (2)    $\mathscr{V}_ k^{(1)}\cong  \wedge^k\mathbb{C}^{2n } $ and  $\mathscr{V}_ k^{(2)}\cong \wedge^k\mathbb{C}^{2n } $ are  spanned by\begin{equation}\label{eq:W-k}\begin{split}&
         \omega^{  \mathbf{ {A}} }, \qquad \varrho\, \omega^{  \mathbf{ {B}} },
            \end{split}\end{equation}respectively, where $|    \mathbf{A}|= |    \mathbf{B}|=k$.
   \end{cor}

\subsection{The induced action of operators $ \mathcal{D }_{j}$'s on the   boundary complex}
\begin{prop} \label{prop:JZ0}
If $j=k+1,\cdots,2n,$  we have\begin{equation*} \begin{split}
s_{A'}d^{A'} & \left[    f_{ a }\widetilde{\mathbf{{S}}}^{a}_{\sigma_j }   + \mathbf{\widetilde{{\Omega}}}^{b}_{\sigma_j } \wedge  G_{ b }   \right]
= \mathscr{R}_{b  }^-\widetilde{\mathbf{{S}}}^{b  }_{ \sigma_{j+1 }  }
   + \mathbf{\widetilde{{\Omega}}}^{c }_{ \sigma_{j+1 } }  \wedge  {\mathscr{T}}_c^-
      \quad {\rm mod}  \quad \mathcal{J} _{j+1}(U);
\end{split}\end{equation*}
  \end{prop}
\begin{proof}
 Note that
\begin{equation*}\begin{split}
   s_{A'}     d^{A'}\widetilde{\mathbf{{\Omega}}}^{b }_{\sigma_j }&=  \widetilde{ \mathbf{S}}^{b   }_{\sigma_{j+1}} \, \mathcal E=  \widetilde{ \mathbf{S}}^{b   }_{\sigma_{j+1}}\left(\mathcal E_0+ \Omega^{A'} \wedge\mathcal{  E}_{ A'}\right)
    \\&=   \mathcal E_0  \widetilde{ \mathbf{S}}^{b   }_{ \sigma_{j+1 }}+\widetilde{\mathbf{ {\Omega}}}^{b+1 }_{\sigma_{j+1} }\wedge \mathcal{  E}_{ 0'}-\widetilde{\mathbf{ {\Omega}}}^{b}_{\sigma_{j+1} }\wedge \mathcal{  E}_{ 1'}\qquad  {\rm mod}  \quad \mathcal{J} _{j+1}(U).
\end{split}\end{equation*}by \eqref{eq:tilde-Omega-Omega01}.  By using \eqref{eq:partial-Omega}, we find that
\begin{equation*}\label{eq:00}\begin{split}
  s_{A'}d^{A'}\left[\mathbf{\widetilde{{\Omega}}}^{b }_{\sigma_j }    \wedge  G_b  \right]
 = &  s_{A'}d^{A'} \mathbf{\widetilde{{\Omega}}}^{b }_{\sigma_j }    \wedge  G_b  - s_{A'} \mathbf{\widetilde{{\Omega}}}^{b}_{\sigma_j }    \wedge d^{A'} G_b \\
= &  \widetilde{ \mathbf{S}}^{b   }_{  \sigma_{j+1 }  }\,   \mathcal E_0 \wedge  G_b- \mathbf{ \widetilde{ {\Omega}}}^{b }_{  \sigma_{j+1 }  }\wedge\left (  \mathfrak{ d}^{ 0 '} G_b+\mathcal{  E}_{ 1'}\wedge G_b\right)  - \mathbf{\widetilde{ {\Omega}}}^{b+1 }_{  \sigma_{j+1 }  }\wedge \left( \mathfrak{ d}^{ 1 '} G_b -\mathcal{  E}_{ 0'}\wedge G_b\right)\\= &   \widetilde{ \mathbf{S}}^{c  }_{ {\sigma_{j+1} }}\, \mathcal E_0 \wedge  G_c-\mathbf{\widetilde{ {\Omega}}}^{b   }_{\sigma_{j+1 } }\wedge \left[    \mathfrak{ d}^{A '} G_{ b-  o(A') }+\mathcal{  E}^{A'} \wedge G_{b- o(A')}\right] \qquad {\rm mod}  \quad \mathcal{J} _{j+1}(U),
\end{split}\end{equation*}   by
$
   \mathbf{\widetilde{{\Omega}}}^{b }_{\sigma_{j+1}  }  \wedge\Omega^{A'}\wedge*\in\mathcal{J} _{j+1} (U)
$ (cf. \eqref{eq:tilde-DDS}).

As in \eqref{eq:4n+1}, $
 s_{A'}\Omega^{  A '}\wedge\partial_{4n+1}  (f_{  a  }  \widetilde{ \mathbf{S}}^{a}_{\sigma_j }   ) = \widetilde{\mathbb{D}}\widetilde{ \mathbf{S}}_{\sigma_j }^{a    } \wedge\partial_{4n+1}   f_{  a  } \in \mathcal{J} _{j+1}(U)$.
 So  we have
 \begin{equation*}\begin{split} s_{A'} d^{ {A} '} (f_{ a  }  \widetilde{ \mathbf{S}}^{a}_{\sigma_{j  }  }   )  &=  s_{A'}\widetilde{ \mathbf{S}}^{a}_{\sigma_{j  }  }  \left ( \mathfrak d^{ A '} f_{  a  } +  \Omega^{ B'} \wedge {\mathbf T}_ { B'}^{ A '} f_{  a  } \right) \\& = \widetilde{ \mathbf{S}}^{a+o(A')}_{\sigma_{j +1 }  }   \mathfrak d^{ A '} f_{  a  } +  \widetilde{ \mathbf{S}}^{b }_{\sigma_{j+1 } }\left( \Omega^{ 0'}\wedge {\mathbf T}_{0'}^{ 0 '}  +  \Omega^{ 1'}\wedge {\mathbf T}_{1'}^{ 0 '} \right) f_{ b   }  +\widetilde{ \mathbf{S}}^{b +1 }_{\sigma_{j+1 } } \left ( \Omega^{ 0'}\wedge {\mathbf T}_{0'}^{ 1 '}  + \Omega^{ 1'}\wedge {\mathbf T}_{1'}^{ 1 '} \right)  f_{ b  }\\
 &=  \widetilde{ \mathbf{S}}^{a}_{\sigma_{j +1 }  }   \mathfrak d^{ A '} f_{  a - o(A') } + \left[  \mathbf{\widetilde{{\Omega}}}^{ b+1}_{  \sigma_{j+1 }  }\wedge\left ( {\mathbf T}_{0'}^{ 0 '}  -    {\mathbf T}_{1'}^{ 1 '}\right)   -\mathbf{\widetilde{{\Omega}}}^{ b  }_{  \sigma_{j+1 }  } \wedge {\mathbf T}_{1'}^{ 0 '} +\mathbf{\widetilde{{\Omega}}}^{ b +2 }_{  \sigma_{j+1 }  }  \wedge{\mathbf T}_{0'}^{ 1 '}  \right ]f_{ b  }\\
 &=\widetilde{ \mathbf{S}}^{a}_{\sigma_{j +1 }  }   \mathfrak d^{ A '} f_{  a - o(A') }    - \mathbf{\widetilde{{\Omega}}}^{ b }_{  \sigma_{j+1 }  }\wedge\left[ \left ({\mathbf T}^{0'1 '}+ {\mathbf T}^{1'0 '}\right) f_{ b -1 } +   {\mathbf T}^{0'0 '}  f_{ b  } +   {\mathbf T}^{1'1 '} f_{ b-2  }   \right ]\\
 &=\widetilde{ \mathbf{S}}^{a}_{\sigma_{j +1 }  }   \mathfrak d^{ A '} f_{  a - o(A') }   - \mathbf{\widetilde{{\Omega}}}^{ b }_{  \sigma_{j+1 }  }\wedge {\mathbf T}^{A'B '}   f_{ b-o(A'B ')  }
   \qquad {\rm mod}  \quad \mathcal{J} _{j+1}(U),
\end{split}\end{equation*}
  by using (\ref{eq:dA'-Omega}), \eqref{eq:tilde-Omega-Omega01} and relabeling indices. The result follows from the sum of these two identities.
 \end{proof}

\begin{prop} \label{prop:JZ>}
We have
 (1) if $j=k-1 ,$  \begin{equation*} \begin{split}
   \partial_{A'}d^{A'}    [  f_{  a  }  \mathbf{S}^{a}_{1 } &+ \mathbf{\Omega}_{1}^{1    }  \wedge G   ]=  \mathfrak  d^{  A'} f_{o(A')}  + \mathcal E_0 \wedge G \\&+\varrho \left[
     \mathfrak{ d} ^{ [0 '} \mathfrak{ d} ^{ 1 ']} G -\mathfrak{ d} ^{ A ' } \left( \mathcal{  E}_{ A'}\wedge  G\right)+\mathcal E_0\wedge \left(    \mathbf{T} ^{[0 '1']}+\mathcal{  E}_{ 0'1'}  \right) G -\mathfrak{ d} ^{ B '} {\mathbf T}_ {B'  }^{   A'   }   f_{ o(A')  }\right] \quad {\rm mod}  \quad \mathcal{J} _{k}(U)
\end{split}\end{equation*}
for $  f_{  a  }\in  \Gamma (U,     \wedge^{k-1}\mathbb{C}^{2n  } )$ and $  G\in  \Gamma (U,     \wedge^{k-2}\mathbb{C}^{2n  } )$. Here  $ {\mathbf T}^{[0 '1']}= \frac 12 ({\mathbf T}^{ 0 '1' } -{\mathbf T}^{1' 0 ' })$ is   skew-symmetrization;
\\
 (2)  if $j=k ,$   \begin{equation*} \begin{split}
d^{ 0 '} d^{ 1 '} (f_0+\varrho f_1)= &  \mathfrak  d^{  0'} \mathfrak  d^{  1'} f_0 - \mathcal E\wedge \left ( {\mathbf T}^{1'0 '} f_0+  f_1 \right)  + \Omega^{   A'   }\wedge \left(2{\mathbf T}_ {  A'  }^{[ 0'} \mathfrak  d^{ 1']}     f_0-   \mathfrak  d_{A ' }   f_1\right)
\end{split}\end{equation*} holds   mod $ \mathcal{J} _{k+1}(U)$ for $  f_{  0  }, f_{  1 }\in  \Gamma (U,     \wedge^{k }\mathbb{C}^{2n  } )$   independent of $\varrho$.
 \end{prop}
 \begin{proof}
(1) Since $\sigma_{k-1}=1$,  $ \mathbf{\Omega}_{\sigma_{k-1}}^{a   }  $  has only one element:
$
  \mathbf{\Omega}_{1}^{1    } = \frac 1{2 }  ( s^{1'}   \Omega^{ 0'}-s^{0'}  \Omega^{ 1'}  ).
$  Here we can assume $  f_{  a  } $ and $  G $ are  independent of $\varrho$, because $\partial_{A'}d^{A'}(\varrho \omega)$ for $  \omega\in  \Gamma (U,     \wedge^{k-1}\mathbb{C}^{2n +2 } )$ belongs to $ \mathcal{J} _{k}(U)$ automatically by \eqref{eq:inclusion}.
 Thus
 \begin{equation*}\label{eq:d-G2}\begin{split}
     \partial_{A'}d^{A'}\left[ \mathbf{\Omega}_{1}^{1    }  \wedge G    \right] =& \mathcal E\wedge G -  \partial_{A'}    \mathbf{\Omega}_{1}^{1    } \wedge  d^{A'}  G  \\   = & \mathcal E_0\wedge G+ \Omega^{A'} \wedge\mathcal{  E}_{ A'}\wedge    G+\Omega^{ 0'}\wedge \Omega^{ 1'}   \wedge\mathcal{  E}_{ 0'1'}      G \\&+\frac 12 \Omega^{ 1'}\wedge  \mathfrak{ d}  ^{ 0 '} G -\frac 12 \Omega^{ 0'}\wedge  \mathfrak{ d}  ^{ 1 '} G+  \frac 12\Omega^{ 1'}   \wedge \Omega^{ 0'}\wedge\left (\mathbf{T}_{  0'}^{0 '}+\mathbf{T}_{ 1' }^{1 '}\right) G
     \\=  &  \mathcal E_0\wedge G + \varrho \left[- \mathfrak{ d} ^{ A ' } \left( \mathcal{  E}_{ A'}\wedge  G\right)+ \mathfrak{ d} ^{[0 '} \mathfrak{ d} ^{ 1 ']}  G + \mathcal E_0\wedge \mathbf{T} ^{[0 '1']} G+\mathcal{  E}_{ 0 } \wedge  \mathcal{  E}_{ 0'1'}  G  \right]   \quad  {\rm mod}  \quad \mathcal{J} _{k}(U),
 \end{split} \end{equation*} by Proposition \ref{prop:JW} (3).
On the other hand,  $f_{  a  }  \mathbf{S}^{a}_{1 }= f_{  0  }  s^{0' }+f_{  1  }  s^{1'}$ by definition, and
\begin{equation*}\begin{split}
    \partial_{A'} d^{ {A} '} (f_{  a  }  \mathbf{S}^{a}_{1 }   )= & \mathfrak  d^{  0'} f_0+ \mathfrak  d^{  1'}  f_1 +\Omega^{B'}\wedge\left({\mathbf T}_ { B' }^{ 0 '}  f_{ 0 }+{\mathbf T}_ { B' }^{1 '}  f_{1 }\right )\\
   = & \mathfrak  d^{  A'} f_{o(A')}   -\varrho \left(  \mathfrak  d^{B'}  {\mathbf T}_ { B' }^{ 0 '}  f_{  0}+\mathfrak  d^{B'}{\mathbf T}_ { B' }^{1 '}  f_{ 1}  \right)   \quad {\rm mod}  \quad \mathcal{J} _{k}(U).
\end{split}\end{equation*}Their sum gives us  the result.

 (2) We have \begin{equation} \begin{split}
   d^{ 0 '}d^{ 1 '}(f_0+\varrho f_1)= &   d^{ 0'}\left[\mathfrak  d^{ 1  '} f_0\        + \Omega^{  A'  }  \wedge   \mathbf T^{1 '}_{A'}  f_0+\Omega^{1'}\wedge    f_1+\varrho \left(\mathfrak  d^{  1'} f_1        + \Omega^{  A'  }  \wedge   {\mathbf T}_ {  A'  }^{ 1 '}  f_1\right) \right]  \\=&\mathfrak  d^{  0'}\mathfrak  d^{  1'} f_0\ + \Omega^{  A'  }\wedge {\mathbf T}_ {  A'  }^{ 0'} \mathfrak  d^{ 1'}   f_0 - \mathcal E  \wedge{\mathbf T}_{1'}^{ 1 '} f_0  - \Omega^{  A'  }  \wedge  \mathfrak  d^{  0'} {\mathbf T}_{  A'  }^{ 1'} f_0\\& - \mathcal{{ E}}\wedge    f_1 +(\Omega^{0 '}\wedge  \mathfrak   d^{1'}   - \Omega^{ 1'  }\wedge  \mathfrak  d^{ 0 '} )  f_1 ,\qquad {\rm mod}  \quad \mathcal{J} _{k+1}(U)
\end{split}\end{equation}by using  (\ref{eq:d-Omega0}), (\ref{eq:dA'-Omega}) and the  characterization of $ \mathcal{J} _{k+1} (U)$    in \eqref{eq:J-j'}. The result follows.
  \end{proof}

\subsection{The operator  ${\mathscr{D} }_{j}$ in the  boundary complex     }
For $j=k+1,\cdots,2n,$   define    isomorphisms
  \begin{equation}\label{eq:Pi-j>}\begin{split}
     \mathbf{\Pi}_j  :  \Gamma(U, {V}_ j  ) / \mathcal{J} _{j }& \longrightarrow \Gamma( bD\cap U,\mathscr{V}_ j^{(1)}  )\oplus \Gamma(bD\cap U, \mathscr{V}_ j^{(2)}  ),\\
     f_{  a  }  \widetilde{ \mathbf{S}}^{a}_{\sigma_j }   + \widetilde{\mathbf{{\Omega}}}^{c}_{\sigma_j }  \wedge G_{ c  } &\mapsto \left(f_{  a  }  \widetilde{ \mathbf{S}}^{a}_{\sigma_j } ,   G_{ c}\widetilde{ \mathbf{S}}^{c }_{{  \sigma_{j+1 }  }}\right),
 \end{split} \end{equation}where the first summation  is taken over
 $a= 0,\ldots, {\sigma_j }$, while the second one is taken over
 $c= 0,\ldots, {\sigma_{j+1 }  }$.

\begin{thm} \label{thm:bd-complex-main:2}   For $\mathbb{ F}=(\mathbb{{F}} _{ 1 }  ,   \mathbb{{F}} _{ 2 } ) \in \Gamma(bD,\mathscr{V}_ j )$, we have
\begin{equation}\label{eq:bdoperator>=}\begin{split}{\mathscr{D} }_{j}\mathbb{ F}   =&\left(
 s_{A'}\mathfrak  d^{ A '}\mathbb{{F}} _{ 1 }   +\mathcal E_0\wedge    \mathbb{{F}} _{ 2 }  , -s_{A'}(\mathfrak  d^{A '} +\mathcal{  E}^{A'} \wedge)  \mathbb{F}_2 - s_{A'} s_{B'} \mathbf{ {T}}^{A ' B }  \mathbb{{F}} _{ 1 } \right),\qquad j=k+1,\cdots,2n ,\\ {\mathscr{D} }_{k-1}\mathbb{ F}   =&\left ( \partial_{A'}\mathfrak  d^{  A'} \mathbb{F}_1   +\mathcal E_0 \wedge  \mathbb{F}_2 , \mathfrak{ d} ^{ [0 '} \mathfrak{ d} ^{ 1 ']} \mathbb{F}_2  -\mathfrak{ d} ^{ A ' } \left( \mathcal{  E}_{ A'}\wedge  \mathbb{F}_2 \right) +\mathcal E_0\wedge \left(\mathbf{T} ^{[0 '1']} +\mathcal{  E}_{ 0'1'}\right) \mathbb{F}_2-\partial_{A'} \mathfrak{ d} ^{ B '} {\mathbf T}_ {B'  }^{   A'   }   \mathbb{F}_1\right),
\\{\mathscr{D} }_{k}\mathbb{ F}   =&\left(
\mathfrak  d^{   0'} \mathfrak  d^{  1' }\mathbb{F}_1 -\mathcal E_0  \wedge \left ( {\mathbf T}^{1'0 '} \mathbb{F}_1 +  \mathbb{F}_2\right)    , -\mathfrak   d_{  A'}  \mathbb{F}_2-  \mathcal{  E}_{A'} \wedge \left ( {\mathbf T}^{1'0 '} \mathbb{F}_1 + \mathbb{F}_2 \right) + 2{\mathbf T}_ {  A'  }^{[ 0'} \mathfrak  d^{ 1']}      \mathbb{F}_1\right)
.
\end{split} \end{equation}  \end{thm}
   \begin{proof} (1)
 For $j=k+2,\cdots,2n,$
      we have
 \begin{equation*}\begin{split}
 {\mathscr{D} }_{j}\left(f_{  a  }  \widetilde{ \mathbf{S}}^{a}_{\sigma_j } ,   {G}_{ b}\widetilde {\mathbf{S}}^{b }_{{  \sigma_{j+1 }  }}\right)&=  \mathbf{\Pi}_{j+1 }\left(s_{A'}d^{A'}  \left[    f_{ a }\widetilde{ \mathbf{S}}^{a}_{\sigma_j }   + \mathbf{\widetilde{{\Omega}}}^{b}_{\sigma_j } \wedge  G_{ b }   \right]
\right)\\&
=  \mathbf{\Pi}_{j+1 }\left( \mathscr{R}_{b  }^- \widetilde{ \mathbf{S}}^{b  }_{ \sigma_{j+1 }  }
   + \mathbf{\widetilde{{\Omega}}}^{c }_{ \sigma_{j+1 } }  \wedge  {\mathscr{T}}_c^- \right)\\&
=  \mathscr{R}_{b }^- \widetilde{ \mathbf{S}}^{b }_{{  \sigma_{j+1 }  }}   +   \mathscr{T}_{c }^-\widetilde{ \mathbf{S}}^{c }_{{  \sigma_{j+2 }  }}
      \quad {\rm mod}  \quad \mathcal{J} _{j+1},
 \end{split} \end{equation*}by   Proposition \ref{prop:JZ0}, where
 \begin{equation}\label{eq:RT-}\begin{split} \mathscr{R}_{b  }^-
\widetilde{\mathbf{ S}}^{b  }_{{  \sigma_{j+1 }  }}  & =  s_{A'}\mathfrak  d^{A '}   (f_{  a  }  \widetilde{ \mathbf{S}}^{a}_{\sigma_j }  )  +\mathcal E_0\wedge  \left  (G_{ b}\widetilde{ \mathbf{S}}^{b }_{{  \sigma_{j+1 }  }}\right) ,\\
  \mathscr{T}_{c }^-\widetilde{ \mathbf{S}}^{c }_{{  \sigma_{j+2 }  }}  &=
   -s_{A'} \mathfrak  d^{A'}   \left (G_{b}\widetilde{ \mathbf{S}}^{b}_{{  \sigma_{j +1 }  }}\right )-\mathcal{  E}^{A'} \wedge s_{A'} \left(G_{b  }\widetilde{ \mathbf{S}}^{b }_{{  \sigma_{j+1  }  }}\right )-s_{A'}s_{B'} \mathbf{T }^{A' B '} \left (f_{  a  }  \widetilde{ \mathbf{S}}^{a}_{\sigma_j } \right ),
 \end{split} \end{equation}by multiplying $s_{A'}$.
  Thus
  $ {\mathscr{D} }_{j}$  is given by the first formula.

(2)   Since $\sigma_{k-1}=1$,  let $ \mathbb{F}_1  =f_{  a  }  \mathbf{S}^{a}_{1 }= f_{  0  } s^{0'}+f_{  1  } s^{1' }$ and $ G   =\mathbb{F}_2$. We have
 \begin{equation*} \begin{split}
 \mathfrak  d^{ A '}  f_{o(A') }&= \partial_{A'}\mathfrak  d^{  A'} \mathbb{F}_1,\\
     \mathfrak{ d} ^{ B '} {\mathbf T}_ {B'  }^ { A '}   f_{o(A') }&  =\partial_{A'}\mathfrak{ d} ^{ B '} {\mathbf T}_ {B'  }^{   A'   } \mathbb{F}_1.
\end{split}\end{equation*}
The result follows Proposition \ref{prop:JZ>}
  (1).

 (3) Note that for $j=k $,  $ \mathbb{F}_1  =f_{  0  }  $ and $  \mathbb{F}_2 =f_{  1  } $. The result follows Proposition \ref{prop:JZ>}
  (2).
  \end{proof}
\begin{rem}
$\mathscr{D  }_{j}$'s for  $j>k+1 $  in   \eqref{eq:bdoperator>=}  are exactly operators for  $j<k-1$ in  \eqref{eq:bdoperator<} with  $\partial_{A'}$ replaced by $s_{A'}$.

(2) By comparing  \eqref{eq:RT+} with \eqref{eq:RT-}, we see that $+o(A')$ in \eqref{eq:RT+-} comes from $\partial_{A'}$, while $-o(A')$ in \eqref{eq:RT+-} comes from multiplying $s_{A'}$.
  \end{rem}
\section{The Hartogs-Bochner extension  of   $k$-regular functions }

\begin{prop}\label{prop:extension} For fixed $k$, if $f\in \Gamma(bD,\odot^k \mathbb{C}^2)$ is $k$-CF, then there exists a representative $ \widehat{ f}\in \Gamma(\overline{D}, \odot^k \mathbb{C}^2)$ such that $\widehat{f}|_{bD}=f$ and
$
  \mathcal{ D}_0\widehat f
$ is flat on $bD$.
 \end{prop}
\begin{proof} Denote also by $f$    any fixed extensions of $f$ to $U\supset D$
as a $ C^\infty$ function.

(1) {\it The case $k> 1$}.  Since $f $ is $k$-CF, $\partial_{A'} d^{ {A} '} f\in \mathcal{J} _{1}(U) $ and so
 \begin{equation}\label{eq:D0-f}\begin{split}\partial_{A'} d^{ {A} '} f=&
      \mathbb{D}\mathbf{S}_{\sigma_1 }^{a    } \cdot  f_{  a} + \varrho F,
 \end{split} \end{equation}
for some functions $f_a\in \Gamma(U, \mathbb{C}), F\in \Gamma(U,\mathcal{V}_1)$, $a=0,\ldots, \sigma_1 $, by the characterization of $ \mathcal{J} _{1} (U) $ in Proposition \ref{prop:JW0}. Note that   $\mathbb{D}$ given by \eqref{eq:DD} is globally defined since $\Omega^{ {A} '}$'s are. But
\begin{equation*}
   \partial_{A'} d^{ {A} '}
     \left(\varrho  \mathbf{S}_{\sigma_1 }^{a    }    f_{  a}\right) = \Omega^{A'}  \partial_{A'} \mathbf{S}_{\sigma_1 }^{a    } \cdot  f_{  a} + O(\varrho)  = \mathbb{D}\mathbf{S}_{\sigma_1 }^{a    } \cdot f_{  a} + O(\varrho).
\end{equation*}
We see that
\begin{equation}\label{eq:D0-f1}\begin{split}\partial_{A'} d^{ {A} '}  \left( f-\varrho  \mathbf{S}_{\sigma_1 }^{a    }   f_{  a}\right)=&
       \varrho F_1^{(1)},
 \end{split} \end{equation}for some $F_1^{(1)}\in \Gamma(U, \mathcal{V}_1)$. As  $\partial_{A'} d^{ {A} '}$  applied to the left hand side gives zero  by \eqref{eq:DD1}, we get
\begin{equation*}
  0= \partial_{A'} d^{ {A} '} (\varrho F_1^{(1)})=\mathbb{D}F_1^{(1)}|_{bD}
+ O(\varrho),
\end{equation*}Namely, $\mathbb{D}F_1^{(1)}|_{bD}=0$. Hence
\begin{equation*}
   F_1^{(1)}=\mathbb{D} \mathbf{S}_{\sigma_1 }^{a    }  \cdot F_{  a} + \varrho F_1 ^{(2)}
\end{equation*}for some functions $F_a\in \Gamma(U, \mathbb{C})$ by Proposition \ref{prop:JW0} again, and so
\begin{equation}\label{eq:D0-f2}\begin{split}\partial_{A'} d^{ {A} '}  \left( f-\varrho  \mathbf{S}_{\sigma_1 }^{a    }   f_{  a}\right)=&
       \varrho \mathbb{D}  \mathbf{S}_{\sigma_1 }^{a    }\cdot  F_{  a} + \varrho^2 F_1 ^{(2)}.
 \end{split} \end{equation}Repeating this procedure, we get
 \begin{equation*}
    \partial_{A'} d^{ {A} '}(f+\varrho f^{(1)}+\varrho^2 f^{(2)}+\cdots)\equiv O_{bD}^\infty
 \end{equation*}
 with a formal power series in $\varrho$ with coefficients $ C^\infty$ on $bD$, where $O_{bD}^\infty$ denotes functions
  vanishing of infinite order on $bD$. By using the Whitney extension theorem
(cf.  \cite[I, Proposition 22]{AH72}),  we get the conclusion.

(2) {\it The case $k =0$}. Write $f=u_0 + \varrho u_1$.  Since $(u_0 , u_1) $ is $0$-CF on $bD$, $ d^{0'}d^{1'} (u_0 + \varrho u_1)\in \mathcal{J} _{1}(U  )  $,
\begin{equation*}
   d^{0'}d^{1'}(u_0 + \varrho u_1)=\varrho \alpha_1+\Omega^{0'}  \wedge \Omega^{1'} \beta_1
\end{equation*}
  for some $\alpha_1\in \Gamma(U, \wedge^2 \mathbb{C}^{2n+2}), \beta_1\in \Gamma(U,  \mathbb{C} )$, by the characterization of $ \mathcal{J} _{1}(U  )  $ for $k =0$ in Proposition  \ref{prop:JW} (1). Then
\begin{equation*}
   d^{0'}d^{1'}\left(u_0 + \varrho u_1-\frac 12\varrho^2\beta_1\right)=\varrho \alpha_1'
\end{equation*}
  for some $\alpha_1'\in \Gamma(U, \wedge^2 \mathbb{C}^{2n+2}) $. As  $d^{0'} $ and $d^{1'}$  applied to the left hand side give zero by \eqref{eq:DD2}, we get
\begin{equation*}
  \Omega^{0'}  \wedge  \alpha_1'|_{bD}=0,\qquad \Omega^{1'}  \wedge  \alpha_1'|_{bD}=0,
\end{equation*} and so
\begin{equation*}
 \alpha_1'=\Omega^{0'} \wedge \Omega^{1'} \beta_2+ \varrho \alpha_2
\end{equation*}  for some $\alpha_2\in \Gamma(U, \wedge^2 \mathbb{C}^{2n+2}) $ and $\beta_2\in \Gamma(U,\mathbb{C} ) $. Thus,
\begin{equation*}
     d^{0'}d^{1'}\left(u_0 + \varrho u_1-\frac 12\varrho^2\beta_1-\frac 16\varrho^3\beta_2+\cdots\right)=\varrho^2 \alpha_2'.
 \end{equation*}
Repeating this procedure, we get
 \begin{equation*}
     d^{0'}d^{1'}\left(u_0 + \varrho u_1 +\cdots\right)\equiv O_{bD}^\infty
 \end{equation*}and get the conclusion as the case $k> 1$.

(3) {\it The case $k =1$}.
Since $f $ is $1$-CF, $ \partial_{A'} d^{ {A} '}f\in \mathcal{J} _{1}(U  ) $, and so
 \begin{equation*}
   \partial_{A'} d^{ {A} '}
        f    = \Omega^{ {A} '}
           f_{ A'}+\varrho   \mathfrak{d}^{ A'} f_{ A'}+ O(\varrho^2)
\end{equation*}for some functions $f_{ A'}\in \Gamma(U, \mathbb{C} )$, by the characterization of $ \mathcal{J} _{1} (U  ) $ for $k =1$ in Proposition  \ref{prop:JW} (3).
 Since
 \begin{equation*}
   \partial_{A'} d^{ {A} '}
     \left(  \varrho s^{B'}     f_{ B'}\right)  =\Omega^{ {A} '}
            f_{ A'}+\varrho  \left (\mathfrak{d}^{ A'} f_{ A'} +\Omega^{ B'}\wedge   \left({\mathbf T}_{ B'}^{ A '}+\partial_{4n+1}\right) f_{ A'}\right)+ O(\varrho^2)
\end{equation*}
we see that
\begin{equation*}
   \partial_{A'} d^{ {A} '}
     \left(    f -\varrho s^{A'}    f_{ A'} \right)  = \varrho \, \Omega^{ {A} '}
          \cdot  f^{(1)}_{ A'}+   \varrho^2 G
\end{equation*}  for some $f^{(1)}_{ A'}\in \Gamma(U, \mathbb{C}) $   and $G\in \Gamma(U,   \mathbb{C}^{2n+2} ) $.
 As  $d^{0'}  d^{1'}$  applied to the left hand side gives zero, we get
 \begin{equation*}
   \Omega^{1 '}
        \wedge \Omega^{ 0 '}
        \cdot \mathfrak{d}^{ 0'}  f^{(1)}_{0'}- \Omega^{0 '}
         \wedge \Omega^{ 1'}
         \cdot \mathfrak{d}^{ 1'}  f^{(1)}_{1'}+   2 \Omega^{0 '}
          \wedge \Omega^{ 1'}
        G=0
 \end{equation*}on the boundary $bD$, i.e.
 \begin{equation*}2G=
     \mathfrak{d}^{ A'}  f^{(1)}_{A'} +   O(\varrho ),\qquad {\rm mod}\quad \Omega^{A '}.
 \end{equation*}
 Therefore
 \begin{equation*}
   \partial_{A'} d^{ {A} '}
     \left(    f -\varrho s^{A'}   f_{ A'} \right)  = \varrho  \left[\Omega^{ {A} '}
       \cdot  f^{(1)}_{ A'}+  \frac  \varrho 2 \left( \mathfrak{d}^{ A'} f^{(1)}_{ A'}+\Omega^{ {A} '}
         \cdot  \widetilde{f}^{(2)}_{ A'}\right)+ O(\varrho^2)\right],
\end{equation*} for some $\widetilde{f}^{(2)}_{ A'}\in \Gamma(U, \mathbb{C}) $,
and so
 \begin{equation*}
   \partial_{A'} d^{ {A} '}
     \left(    f -\varrho s^{A'}   f_{ A'} -  \frac  {\varrho^2} 2    s^{A'} f^{(1)}_{ A'}\right)  = \varrho^2  \left[\Omega^{ {A} '}
         \cdot  f^{(2)}_{ A'} + O(\varrho )\right].
\end{equation*} for some $f^{(2)}_{ A'}\in \Gamma(U, \mathbb{C}) $. Repeating this procedure, we get
  the conclusion as above.
  \end{proof}
\begin{rem} Such extension for CR functions was constructed by Andreotti-Hill \cite{AH72}, while
  extension for pluriharmonic functions satisfying $\partial\overline{\partial}$-equation was constructed by  Andreotti-Nacinovich in \cite{Andreotti2}.
\end{rem}

To show the Hartogs-Bochner extension, we need to solve   nonhomogeneous
 $k$-Cauchy-Fueter equation
 $
\mathcal{D}_0 u=f,$
 for $f$ satisfying the compatibility
condition  $\mathcal{D}_1f=0$.
\begin{thm}\label{thm:CF-D0}\cite[Theorem 5.3]{Wang} For  $f\in C_0 (\mathbb{R}^{4n+4},
\mathcal{ V}_1)$ such that $\mathcal{D}_1f=0$ in the sense of distributions, then there exists a
function $u\in C_0(\mathbb{R}^{4n+4},  \mathcal{V}_0 )\cap
W^{1,2}(\mathbb{R}^{4n+4},  \mathcal{V}_0   )$ ($u\in
C_0(\mathbb{R}^{4n+4}, \mathbb{C}  )\cap W^{2,2}(\mathbb{R}^{4n+4},
\mathbb{C} )$ if $k=0$) satisfying $
\mathcal{D}_0 u=f $ and vanishing on
the unbounded connected component of $ \mathbb{R}^{4n+4}\setminus
{\rm supp}  f$.
\end{thm}
Let us recall the construction of  solutions in \cite{Wang}.
 Consider  the associated Hodge-Laplacian on $\Gamma(D, \mathcal{{V}}_1)$:
\begin{equation}\label{eq:laplacian0} \begin{split} & \square_{ 1}=  \mathcal{D}_{0 }   \mathcal{D}_{0   }^* +
\left(\mathcal{ D}_{ 1}^{ * } \mathcal{D}_{ 1} \right)^2  , \qquad \square_1= \left( \mathcal{D}_{0} \mathcal{D}_{0 }^{ * }
\right)^2 +  \mathcal{D}_1^{ * } \mathcal{D}_1,
\end{split}
\end{equation} if $k=0$ and $k=1$, respectively, and
\begin{equation} \label{eq:laplacian} \begin{split}  & \square_1=
\left(\mathcal{D}_{0 }   \mathcal{D}_{0 }^{ * } \right)^2+ \left( \mathcal{D}_{1
}^{ * } \mathcal{D}_{1} \right)^2,
\end{split}
\end{equation}
 if $k\geq 2$. They are all uniformly elliptic differential operators of
 $4$-th order with constant coefficients. Its
 inverse $ \mathbf{G}_1$ in
$L^2(\mathbb{R}^{4n+4}, \mathcal{V}_j )$
  is a convolution operator with matrix kernel   of
$C^\infty(\mathbb{R}^{4n+4}\setminus\{0\})$ homogeneous functions of
degree   $ -4n$. The solution is given by $\mathcal{D}_0^*\mathcal{D}_{0 }  \mathcal{ D}_{0 }^{ * }\mathbf{G}_1 f$ for $k\geq1$, i.e.
 \begin{equation*}
    \mathcal{D}_0 ( \mathcal{D}_0^*\mathcal{D}_{0 }  \mathcal{ D}_{0 }^{ * }\mathbf{G}_1 f)=f.
 \end{equation*}
 The solution is   $\mathcal{D}_0^* \mathbf{G}_1 f$ for $k=0$.

\vskip 3mm

 {\it Proof  of Theorem \ref{thm:Hartogs-Bochner}}. By Proposition \ref{prop:extension}, we can extend $f$ to a smooth function $ \widehat{ f}$ on $\overline{D}$
and extend
  $\mathcal{ {D}}_0 \widehat{ f}$ by $0$ outside of $\overline{D}$ to get a  $ \mathcal{{D}}_1$-closed element $ {{F}}\in \Gamma(\mathbb{H}^{n+1}, \mathcal{{V}}_1)$
supported in  $\overline{D}$. Then, by Theorem \ref{thm:CF-D0},
  there exists $H \in C (\mathbb{H}^{n+1}, \mathcal{ V} _0)$ with compact support   such that $ {F} =  \mathcal{{D}}_0 H$.
Note that
  $H$    is $k$-regular on $ \mathbb{H}^{n+1}\setminus \overline{D}$. Since a  $k$-regular function is harmonic \cite{Wang} (see this fact in \eqref{eq:diag} for right-type
  groups),  by analytic continuation, $H$ vanishes on   the connected open set $\mathbb{H}^{n+1}\setminus \overline{D}$. Then
$F = f-H$ gives us  the required extension.
\qed
 \section{  right-type
  groups  and  the quaternionic Monge-Amp\`{e}re operator }
\subsection{The nilpotent Lie groups of step two associated to  rigid  quadratic hypersurfaces}
Consider rigid model domain $D$ in \eqref{eq:rigid-hypersurface} and the projection
$\pi
 :  D    \longrightarrow    \mathbb{H}^n\times \mathbb{H}_+
 $ given by
 \begin{equation*}
    (\mathbf{q}', q_{n+1} ) \longmapsto (\mathbf{q}'  , q_{n+1}- \phi(\mathbf{q}')),
 \end{equation*}
  where $\mathbf{q}' \in \mathbb{H}^n$, which maps $bD$ to $\mathbb{H}^n\times \operatorname{Im} \mathbb{H}$.   The push  forward vector field $\pi_*\overline{Q}_l$ is exactly $ \overline{\partial}_{q_{l+1}}+\overline{\partial}_{q_{l+1}}\phi\cdot
\overline{\partial}_{\mathbf{t}} $, where  $\mathbf{t}=t_1\textbf{i}+t_2\textbf{j}+t_3\textbf{k}\in \operatorname{Im} \mathbb{H} $ (cf. \cite[Subsection 5.1]{shi-wang}).
Denote
\begin{equation}\label{eq:Xj}
   X_{4l+1}+\mathbf{i}X_{4l+2}+\mathbf{j}X_{4l+3}+\mathbf{k}X_{4l+4}:=\overline{\partial}_{q_{l+1}}+\overline{\partial}_{q_{l+1}}\phi\cdot
\overline{\partial}_{\mathbf{t}},
\end{equation}where $l=0,\ldots,n -1$.
Then
\begin{equation}\label{eq:X-b}
   X_{b}=\partial_{x_{b}}+2\sum_{\beta=1}^3\sum_{a=1}^{4n}\left(\mathbb S\mathbb I^\beta\right)_{ab}x_a\partial_{t_\beta}
\end{equation}
 by direct calculation \cite[Proposition 5.1]{shi-wang}, and so
\begin{equation}\label{eq:X-B}
   [X_{a},X_{b}]=2\sum_{\beta=1}^{3} B^\beta_{ab}\partial_{t_\beta},
\end{equation} where
\begin{equation}\label{eq:B-S}   B^\beta:=
   \mathbb{S}\mathbb I^\beta+\mathbb I^\beta\mathbb{S}
\end{equation}
 (\cite[Subsection 5.1]{shi-wang} \cite{wang13}). Here $\mathbb{I}^{\beta}={\rm diag}(I^\beta,I^\beta,\ldots)$ with
 \begin{equation}\label{eq:I}
    I^{1}:=\left(\begin{array}{cccc} 0 &  1 & 0 &0\\ -1& 0& 0& 0\\ 0& 0&0& -1\\0 &0& 1 &0\end{array}\right),   \qquad I^{2}:=\left(\begin{array}{cccc} 0 & 0 &
1 &0\\ 0& 0& 0& 1\\- 1& 0&0& 0\\0 &-1& 0 &0\end{array}\right),\qquad
I^{3}:=\left(\begin{array}{cccc} 0 & 0 & 0 &1\\ 0& 0& -1& 0\\ 0& 1&0& 0\\-1 &0& 0 &0\end{array}\right)
 \end{equation}
satisfying  commutating relation of quaternions
$(I^{1})^{2}=(I^{2})^{2}=(I^{3})^{2}=-I_{4 },$ $ I^{1}I^{2}=I^{3}.$
So ${\rm span}_\mathbb{R} \{X_1,\cdots$, $X_{4n},\partial_{t_1},
\partial_{t_2},\partial_{t_3} \}$ is a nilpotent Lie algebra with center ${\rm span}_\mathbb{R}\left\{\partial_{t_1},
\partial_{t_2},\partial_{t_3}\right\}.$ The corresponding nilpotent Lie group  of step two has
the multiplication given by (\ref{eq:mul}).

Recall that  the Lie algebra $\mathfrak{so}(4)$ of skew symmetric $4\times4$ matrices  has the decomposition
\begin{equation*}
   \mathfrak{so}(4)\cong\mathfrak{sp}(1)\oplus\mathfrak{sp}(1)
\end{equation*}
 with one $\mathfrak{sp}(1)$   spanned by (\ref{eq:I}) and the other one spanned by
 \begin{equation}\label{eq:J}\begin{split}
&J^1=\left(
\begin{array}{cccc}0&1&0&0\\-1&0&0&0\\0&0&0&1\\0&0&-1&0\end{array}\right),\qquad
J^2=\left(\begin{array}{cccc}
0&0&1&0\\0&0&0&-1\\-1&0&0&0\\0&1&0&0\end{array}\right),\qquad
J^3=\left(\begin{array}{cccc}
0&0&0&1\\0&0&1&0\\0&-1&0&0\\-1&0&0&0\end{array}\right),\end{split}  \end{equation}
which also satisfy  the
quaternionic  commutating relation.
If we write $\mathbb{S}$ as a block matrix $(\mathbb{S}^{(lm)})$ with $ \mathbb{S}^{(lm)} $ to be $4\times4$ matrices, then
\begin{equation}\label{eq:block-B}
   B_{lm}^\beta=I^\beta \mathbb{S}^{(lm)}+\mathbb{S}^{(lm)}I^\beta,
\end{equation} where $l,m=0,\ldots, n-1$, and
$
  B^\beta  =(B^\beta_{lm})
$. The following is a direct characterization of  right-type  groups in terms of matrices $B^\beta$.

\begin{prop}\label{prop:right-type}
  The   group $ \mathcal{{N}}_{\mathbb{S}}$   is
   right-type if and only if $ B_{lm}^\beta\in {\rm span}\, \{J^1,J^2,J^3,I_{4 }\}$ for each $l,m ,\beta $.
\end{prop}
\begin{proof} Note that $\mathbb{S}^{(lm)}_{ab}=\frac 12\partial_{4l+a}\partial_{4m+b}\phi= s_{(4l+a)(4m+b)}=\mathbb{S}^{(m l)}_{b a}$, i.e. $\mathbb{S}^{(lm)}=(\mathbb{S}^{(m l)})^t$, which implies $B_{lm}^\beta=-(B_{m l}^\beta)^t$. Write $\mathcal E =  d^{0 '}  d^{ 1'}\phi=\mathcal E_{AB }\omega^A\wedge\omega^B$  and note that $\mathcal E = \mathcal E_0$ in the rigid case.
  Then by the expression \eqref{eq:k-CF-raised-flat} of $\nabla_{ {A}}^{A'}$, we get
\begin{equation*}\label{eq:mathcal-E}\begin{split}
  \mathcal E_{(2l)(2m) }&  =  \frac 12\left(  \nabla_{ 2l}^{0'}   \nabla_{ 2m}^{ 1'}- \nabla_{2m}^{0'}   \nabla_{ 2l }^{ 1'}  \right)\phi\\&= \frac 12 (\partial_{4l+3}+\mathbf i\partial_{4l+4})(\partial_{4m+1}+\mathbf i\partial_{4m+2})\phi -\frac 12(\partial_{4l+1}+\mathbf i\partial_{4l+2})(\partial_{4m+3}+\mathbf i\partial_{4m+4}) \phi\\
&= \mathbb{S}^{(lm)}_{31}-\mathbb{S}^{(lm)}_{13}-\mathbb{S}^{(lm)}_{42}+\mathbb{S}^{(lm)}_{24}-\mathbf i\left(\mathbb{S}^{(lm)}_{14}-\mathbb{S}^{(lm)}_{41  } +\mathbb{S}^{(lm)}_{23}-\mathbb{S}^{(lm)}_{32}\right),\\
  \mathcal E_{(2l+1)(2m+1) }&  = \frac 12\left (  \nabla_{ 2l+1}^{0'}   \nabla_{ 2m+1}^{ 1'}- \nabla_{2m+1}^{0'}   \nabla_{ 2l+1 }^{ 1'} \right )\phi =   \overline{\mathcal E_{(2l)(2m) }},\\
  \mathcal E_{(2l)(2m+1) }&  =  \frac 12\left(  \nabla_{ 2l}^{0'}   \nabla_{ 2m+1}^{ 1'}- \nabla_{2m+1}^{0'}   \nabla_{ 2l }^{ 1'}  \right )\phi\\&= \frac 12(\partial_{4l+3}+\mathbf i\partial_{4l+4})(\partial_{4m+3}-\mathbf i\partial_{4m+4})\phi+\frac 12(\partial_{4l+1}+\mathbf i\partial_{4l+2})(\partial_{4m+1}-\mathbf i\partial_{4m+2})\phi  \\
&= \mathbb{S}^{(lm)}_{ 1 1 }+\mathbb{S}^{(lm)}_{22}+\mathbb{S}^{(lm)}_{33}+\mathbb{S}^{(lm)}_{44} +\mathbf i\left(\mathbb{S}^{(lm)}_{43}-\mathbb{S}^{(lm)}_{34}-\mathbb{S}^{(lm)}_{12}+\mathbb{S}^{(lm)}_{21}\right),
\end{split}\end{equation*}by $\overline{\nabla_{2l+1}^{0'}}=-{\nabla_{2l}^{1'}},$ $\overline{\nabla_{2m+1}^{1'}}={\nabla_{2m }^{0'}}$.
Thus $\mathcal E =0$ if and only if
\begin{equation}\label{eq:E=0}\begin{split}& \mathbb{S}^{(lm)}_{ 1 1 }+\mathbb{S}^{(lm)}_{22}+\mathbb{S}^{(lm)}_{33}+\mathbb{S}^{(lm)}_{44}=0,\\
   &\mathbb{S}^{(lm)}_{12}-\mathbb{S}^{(lm)}_{21} +\mathbb{S}^{(lm)}_{34} -\mathbb{S}^{(lm)}_{43}=0,\\&\mathbb{S}^{(lm)}_{13}-
   \mathbb{S}^{(lm)}_{31}- \mathbb{S}^{(lm)}_{24}+\mathbb{S}^{(lm)}_{42}=0,\\&\mathbb{S}^{(lm)}_{14}-\mathbb{S}^{(lm)}_{41  } +\mathbb{S}^{(lm)}_{23}-\mathbb{S}^{(lm)}_{32}=0.
\end{split}\end{equation}
By direct calculation, we have
\begin{equation}\label{eq:S-B} \begin{split}
& B_{lm}^1 :=\left(
\begin{array}{rrrr}\mathbb{S}^{*}_{21}-\mathbb{S}^{*}_{12}&\mathbb{S}^{*}_{22}+ \mathbb{S}^{*}_{11}&\mathbb{S}^{*}_{23} + \mathbb{S}^{*}_{14}&\mathbb{S}^{*}_{24}-\mathbb{S}^{*}_{13}\\
- \mathbb{S}^{*}_{11}-\mathbb{S}^{*}_{22} &-\mathbb{S}^{*}_{12}+\mathbb{S}^{*}_{21}&-\mathbb{S}^{*}_{13}+\mathbb{S}^{*}_{24}&-  \mathbb{S}^{*}_{14}-\mathbb{S}^{*}_{23}
\\-  \mathbb{S}^{*}_{4 1}-\mathbb{S}^{*}_{3 2} & -\mathbb{S}^{*}_{42} +\mathbb{S}^{*}_{31}&-\mathbb{S}^{*}_{4 3}+\mathbb{S}^{*}_{34}&-\mathbb{S}^{*}_{44}-  \mathbb{S}^{*}_{33}
\\ \mathbb{S}^{*}_{3 1}-\mathbb{S}^{*}_{4 2}&   \mathbb{S}^{*}_{32}+\mathbb{S}^{*}_{41}  & \mathbb{S}^{*}_{33} + \mathbb{S}^{*}_{44} &\mathbb{S}^{*}_{34}-\mathbb{S}^{*}_{4 3}\end{array}\right)
  \end{split} \end{equation}where $*=(lm)$. By (\ref{eq:I})-(\ref{eq:J}), it is direct to see that \eqref{eq:E=0} holds   if and only if
  \begin{equation}\label{eq:right-B} \begin{split}
B_{lm}^1 : =& (\mathbb{S}^{*}_{11}+\mathbb{S}^{*}_{22})J^1+(\mathbb{S}^{*}_{14}+\mathbb{S}^{*}_{23})J^2  +
(\mathbb{S}^{*}_{24}-\mathbb{S}^{*}_{13})J^3+( \mathbb{S}^{*}_{21}-\mathbb{S}^{*}_{12})I_4. \end{split} \end{equation}
  Similarly, we have
\begin{equation}\label{eq:right-B2} \begin{split}
 B_{lm}^2:= &  (\mathbb{S}^{*}_{3 2}-\mathbb{S}^{*}_{14} )J^1+( \mathbb{S}^{*}_{11}+\mathbb{S}^{*}_{33})J^2 +(\mathbb{S}^{*}_{12}+\mathbb{S}^{*}_{34})J^3+( \mathbb{S}^{*}_{31}-\mathbb{S}^{*}_{13})I_4,\\
 B^3_{lm}:=&  (\mathbb{S}^{*}_{13}+\mathbb{S}^{*}_{24})J^1+( \mathbb{S}^{*}_{34}-\mathbb{S}^{*}_{12})J^2 +(\mathbb{S}^{*}_{11}+\mathbb{S}^{*}_{44} )J^3+( \mathbb{S}^{*}_{41}-\mathbb{S}^{*}_{14})I_4.\end{split} \end{equation}
For $l=m$, coefficients of $I_4$ vanish since $B^\beta_{ ll }$ is skew-symmetric.  The Proposition is proved.
\end{proof}

 The  right quaternionic Heisenberg group \eqref{eq:right} is associated to the rigid quadratic hypersurface \eqref{eq:rigid-hypersurface}
with
$
   \phi =\sum_{l=0}^{n-1}\left(-3x^2_{4l+1}+x^2_{4l+2}
+x^2_{4l+3}+x^2_{4l+4}\right)
$   and $B^{\beta}={\rm diag}(-J^\beta,-J^\beta,\ldots)$ \cite{shi-wang}. This group  is right-type by Proposition \ref{prop:right-type}, and the tangential $k$-Cauchy-Fueter complex   constructed in \cite{shi-wang}  is a special case of  the subcomplex
\eqref{eq:subcomplex}.
The    left quaternionic Heisenberg group \eqref{eq:left} is associated to the rigid quadratic  hypersurface
with
 $
   \phi(\mathbf{q}') =|\mathbf{q}'|^2.
$
 The domain is   the {\it quaternionic Siegel upper half space} \cite{CDLWW,CMW}.
  In this case,    $B^{\beta}={\rm diag}(I^\beta,I^\beta,\ldots)$. This group  is not right-type  by Proposition \ref{prop:right-type}.

\begin{prop}\label{prop:d'2} On rigid hypersurfaces,
\begin{equation}\label{eq:dd-E}
   \mathfrak{d}^{(A'}\mathfrak{d}^{B')}  =  -\mathcal E\wedge \mathbf{{T}}^{(A' B ')}.
\end{equation}
 \end{prop}
\begin{proof} For functions independent of $x_{4n+1}$, we have  \begin{equation*}\label{eq:d-d} \begin{split}
 d^{ B'}   d^{ A'} &=d^{ B'}( \mathfrak{ d}^{ {A} '}    +\Omega^{ C'}\wedge   {\mathbf T}_{ C'}^{ A '} )\\&
 =\mathfrak{ d}^{ B '} \mathfrak{ d}^{ {A} '}+\Omega^{ C'}\wedge   {\mathbf T}_{ C'}^{B '}\mathfrak{ d}^{ {A} '}- \varepsilon^{ B'C'}\mathcal E\wedge   {\mathbf T}_{ C'}^{ A '}-\Omega^{ C'}\wedge  \mathfrak{ d}^{ B '} {\mathbf T}_{ C'}^{ A '}-\Omega^{ C'}\wedge   \Omega^{D'}\wedge   {\mathbf T}_{ D'}^{ B'} {\mathbf T}_{ C'}^{ A '}
\end{split}\end{equation*}by using the   the Leibnitz law \eqref{eq:Leibnitz2}.
Hence $d^{(A'} d^{B')}=0$ in Proposition \ref{prop:dd}  (1) gives us
\begin{equation*}
   \mathfrak{d}^{(A'}\mathfrak{d}^{B')}  = -\mathcal E\wedge \mathbf{{T}}^{(A' B ')} -\frac 12\Omega^{ C'}\wedge ( [  {\mathbf T}_{ C'}^{A' },\mathfrak{ d}^{  B '}] +[  {\mathbf T}_{ C'}^{ B '},\mathfrak{ d}^{A'}])+\Omega^{ C'}\wedge   \Omega^{D'}\wedge   {\mathbf T}_{ [D'}^{ (A' } {\mathbf T}_{ C']}^{ B ')} .
\end{equation*}
The result follows from $[  {\mathbf T}_{ C'}^{A' },\mathfrak{ d}^{  B '}] =0 ={\mathbf T}_{ [D'}^{ (A' } {\mathbf T}_{ C']}^{ B ')} $, by expressions of  $ \mathbf{{T}}^{A' B '}$  in \eqref{eq:T-AB}.
\end{proof}

Thus the anti-commutativity  \eqref{eq:delta20} of $\mathfrak{d}^{ A'}$'s holds on  right-type  groups.
 It is equivalent to the following brackets between  $Z _{  A}^{ A '}$'s, generalizing the result for  the right quaternionic Heisenberg group \cite{shi-wang}.

\begin{cor}\label{cor:Z -bracket} On the   group $ \mathcal{{N}}_{\mathbb{S}}$,
$
Z _{ [A}^{(A '} Z _{ B]}^{B')} = -\mathcal E_{ A B} \mathbf{ {T}}^{(A' B ')}
$. In particular, $
Z _{ [A}^{(A '} Z _{ B]}^{B')} = 0
$ if the group is right-type.
   \end{cor}\begin{proof} It follows from Proposition \ref{prop:d'2} by
$
    \mathfrak{d}^{(A'}\mathfrak{d}^{B')}u  =    Z _{  A}^{(A '} Z _{ B }^{B')}u\, \omega^A\wedge \omega^B
$ for a scalar function  $u$.
\end{proof}
Although a $2n$-form is not an authentic differential form and we cannot integrate it, we can define a functional on $ L^1(\Omega,\wedge^{2n}\mathbb{C}^{2n})$ \cite{wan-wang} \cite{wang21}. For $F=f~\Omega_{2n}\in L^1(\Omega,\wedge^{2n}\mathbb{C}^{2n})$, let
\begin{equation}\label{2.24}\int_\Omega F:=\int_\Omega f dV,
\end{equation}where   $dV$ is the Lebesgue measure on $\mathbb{R}^{4n+3}$, which is invariant  on the
  group $ \mathcal{ N}_{\mathbb{S}}$,
$\Omega_{2n}:=\omega^0\wedge\omega^{  1 }\cdots\wedge
\omega^{2n-1}\in \wedge^{2n}_{\mathbb{R}+}\mathbb{C}^{2n}.$ Then
$\beta_n^n=n!\Omega_{2n}$ for
$\beta_n:=\sum_{l=0}^{n-1} \omega^{2l}\wedge\omega^{ 2l +1} .$

 \begin{lem}\label{lem:stokes} {\rm (Stokes-type formula)}   Let $\Omega$ be a bounded   domain on the   group $ \mathcal{{N}}_{\mathbb{S}}$  with smooth boundary and defining function $\rho$ (i.e. $\rho=0$ on $\partial \Omega$ and $\rho<0$ in $\Omega$) such that $|{\rm grad}\, \rho|=1$.
Assume that $T=\sum_AT_A\omega^{\widehat{A}}$ is a smooth $(2n-1)$-form in $\Omega$,
where $\omega^{\widehat{A}}=\omega^A  \rfloor \Omega_{2n}$. Then for   $h\in C^1(\overline{\Omega})$, we have
\begin{equation}\label{stokes}\int_\Omega h\mathfrak  d^{A'} T=-\int_\Omega \mathfrak  d^{A'} h\wedge T+\sum_{A=0}^{2n-1} \int_{\partial\Omega}hT_A
 Z_{A}^{A'}\rho\, dS,\end{equation}
where    $dS$ denotes the surface measure of $\partial\Omega$. There is no boundary term if $h=0$ on
$\partial\Omega$.
\end{lem}
\begin{proof} The proof is the same as that on the Heisenberg group \cite{wang21}. Note that \begin{equation*}\mathfrak  d^{A'}(hT)=\sum_{B,A} Z_{B }^{A'}(hT_A)\omega^B\wedge\omega^{\widehat{A}}=\sum_{A} Z_{A}^{A'}(hT_A) \Omega_{2n}.\end{equation*} Then
 \begin{equation*}\int_\Omega \mathfrak  d^{A'}(hT)=\int_\Omega\sum_{A} Z_{A}^{A'}(hT_A) dV=\int_{\partial\Omega}\sum_{A} hT_A
 Z_{A}^{A'}\rho~dS,\end{equation*} by definition (\ref{2.24}) and integration by part,
 \begin{equation}\label{eq:Stokes}
    \int_\Omega  X_j f \, dV=\int_{\partial\Omega} f X_j \rho\,dS,
 \end{equation}
 for $j=1,\ldots 4n$. (\ref{eq:Stokes}) holds because   coefficients  of $\partial_{t_\beta}$'s in $X_j $ are independent of $\mathbf{t}$.
  (\ref{stokes}) follows from the above formula and
 $\mathfrak  d^{A'}(hT)=\mathfrak  d^{A'} h\wedge T+h\mathfrak  d^{A'} T$.
\end{proof}

\subsection{The   quaternionic Monge-Amp\`{e}re operator on right-type
  groups  }Recall that \cite{wan-wang,wang21}
a $ 2p $-form $\omega$ is said to be \emph{elementary strongly positive} if there exist linearly
independent right $\mathbb{H}$-linear mappings $\eta_j:\mathbb{H}^n\rightarrow \mathbb{H}$ , $j=1,\ldots,p$, such that
\begin{equation*}\omega=\eta_1^*\widetilde{\omega}^0\wedge
\eta_1^*\widetilde{\omega}^1\wedge\ldots\wedge\eta_p^*\widetilde{\omega}^0\wedge \eta_p^*\widetilde{\omega}^1,\end{equation*}where
$\{\widetilde{\omega}^0,\widetilde{\omega}^1\}$ is a basis of $\mathbb{C}^{2}$ and $\eta_p^*:~\mathbb{C}^{2}\rightarrow\mathbb{C}^{2n}$
is the induced $\mathbb{C}$-linear pulling back transformation of $\eta_p$. It is called \emph{strongly positive} if it belongs to the convex cone
$\text{SP}^{2p}\mathbb{C}^{2n}$  generated by elementary strongly positive $2p$-elements.
A  $2p$-element $\omega$ is said to be \emph{positive} if for any   strongly positive element
$\eta\in \text{SP}^{2n-2p}\mathbb{C}^{2n}$, $\omega\wedge\eta$ is positive. For a domain $\Omega$  in a   right-type
  group $ \mathcal{ N}_{\mathbb{S}}$, let
 $
\mathscr{D}^{p}(\Omega)=C_0^\infty(\Omega,\wedge^{p}\mathbb{C}^{2n})$.  An element of the dual space
$[\mathscr {D}^{2n-p}(\Omega)]' $ is called a~\emph{$p$-current}. A $2p$-current $T$ is said to be \emph{positive} if we have
$T(\eta)\geq0$ for any strongly positive form $\eta\in\mathscr {D}^{2n-2p}(\Omega)$.  Now for the $p$-current $F$, we define a $(p+1)$-current $\mathfrak d_{A'} F $ as
$(\mathfrak d_{A'}  F)(\eta):=-F(\mathfrak d_{A'} \eta),
$ for any test $(2n-p-1)$-form $\eta$, $  A'=0',1'$. We say a current $F$ is \emph{closed} if
$\mathfrak d_{0'} F=\mathfrak d_{1'} F=0 .$

A $[-\infty,\infty)$-valued  upper
semicontinuous function    on a   right-type
  group $  \mathcal{N}_{\mathbb{S}}$ is said to be \emph{plurisubharmonic}  if it is $L^1_{\rm loc} $ and  $\triangle u$ is
a closed positive $2$-current. For a $  C^2 $  plurisubharmonic functions $u $,  $\triangle u$ is a closed   strongly positive
$2$-form.

For      positive $(2n-2p)$-form  $T$ and
  an arbitrary compact subset $K$, define
$\|T\|_K:=\int_K T\wedge\beta_n^p.$
  For $u_1,\ldots,
u_n\in C^2$ on a   right-type
  group $  \mathcal{N}_{\mathbb{S}}$, we have\begin{equation}\label{eq:key-id}\begin{aligned}\triangle u_1\wedge \triangle
u_2\wedge\ldots\wedge\triangle u_n&=\mathfrak d_{0'}(\mathfrak d_{1'}u_1\wedge \triangle
u_2\wedge\ldots\wedge\triangle u_n)=-\mathfrak d_{1'}(\mathfrak d_{0'}u_1\wedge \triangle
u_2\wedge\ldots\wedge\triangle u_n)\\&=\mathfrak d_{0'}\mathfrak d_{1'}(u_1\triangle
u_2\wedge\ldots\wedge\triangle u_n)=\triangle (u_1
\triangle u_2\wedge\ldots\wedge\triangle u_n).
\end{aligned}\end{equation}
 It directly follows from $\mathfrak d_{A'} \triangle=0 $ by the anti-commutativity  \eqref{eq:delta20} of $\mathfrak{d}^{ A'}$'s on this kind of groups.

  \begin{thm}\label{thm:estimate-C0}   Let
$\Omega$ be a domain in    a   right-type
  group $ \mathcal{ N}_{\mathbb{S}}$. Let $K$ and $L$ be compact subsets
of $\Omega$ such that $L$ is contained in the interior of $K$. Then
there exists a constant $C$ depending only on $K,L$ such that for
any $C^2 $ plurisubharmonic functions $u_1,\ldots u_p$ on $\Omega $, we have
\begin{equation}\label{3.14}\|\triangle
u_1\wedge\ldots\wedge\triangle u_p \|_{L}\leq
C\prod_{i=1}^p\|u_i\|_{C^0(K)}.
\end{equation}  \end{thm}
\begin{proof} The proof is similar to that on the Heisenberg group \cite{wang21}. By definition, $\triangle
u_1\wedge\ldots\wedge\triangle u_p $ is already   closed and strongly positive.   Since $L$ is compact, there is a covering of $L$ by a
family of balls $B_j'\Subset B_j\subseteq K$. Let $\chi\geq0$ be a smooth function equals to 1 on $\overline{B_j'}$ with support in
$B_j$. We have
\begin{equation}\label{eq:int-part}\begin{aligned} \int_{B_j}\chi \triangle
u_1\wedge\ldots\wedge\triangle u_p \wedge \beta_n^{n-p}&=-\int_{B_j}\mathfrak  d_{0'}\chi\wedge\mathfrak  d_{1'} u_1\wedge\triangle
u_2\wedge\ldots\wedge\triangle u_p \wedge \beta_n^{n-p}\\&
=  -\int_{B_j} u_1\mathfrak  d_{1'}\mathfrak  d_{0'}\chi\wedge\triangle
u_2\wedge\ldots\wedge\triangle u_p \wedge \beta_n^{n-p}\\&=   \int_{B_j} u_1 \triangle\chi\wedge\triangle
u_2\wedge\ldots\wedge\triangle u_p \wedge \beta_n^{n-p}\end{aligned}\end{equation}by using
the   identity \eqref{eq:key-id}  and    Stokes-type formula in Lemma \ref{lem:stokes}.
Then
 \begin{equation*}\begin{aligned}\|\triangle
u_1 \wedge\ldots\wedge\triangle u_p\|_{L\cap\overline{B_j'}}=&
\int_{L\cap\overline{B_j'}}\triangle
u_1 \wedge\ldots\wedge\triangle u_p\wedge\beta_n^{n-p}\leq \int_{B_j}\chi\triangle
u_1\wedge  \ldots\wedge\triangle u_p\wedge\beta_n^{n-p}\\=&\int_{B_j}u_1\triangle\chi
\wedge\triangle
u_2\wedge\ldots\wedge\triangle u_p\wedge\beta_n^{n-p}\\\leq& \frac 1\varepsilon
 \|u_1\|_{L^{\infty}(K)}\|\triangle\chi \|_{L^{\infty}(K) } \int_{B_j}
 \triangle
u_2\wedge\ldots\wedge\triangle u_p\wedge\beta_n^{n-p+1},
\end{aligned}\end{equation*}for some $\varepsilon>0$, by using (\ref{eq:int-part}) and the positivity of  $-\varepsilon\triangle\chi
 +\|\triangle\chi \|_{L^{\infty}(K) }\beta_n$ for sufficiently small $\varepsilon>0$ (cf. \cite[Lemma 3.3]{wan-wang}).
The result follows by repeating this procedure.
\end{proof}

  It is standard to yield the existence of Monge-Amp\`{e}re measure (cf. e.g. \cite{alesker1} \cite{wang13}) from the Chern-Levine-Nirenberg type estimate in Theorem \ref{thm:estimate-C0}. We omit   details.
   Let $\{u_j\}$ be a sequence of $C^2$ plurisubharmonic
functions converging to $u$ uniformly on compact subsets of  a
domain $\Omega$ in    a   right-type
  group $ \mathcal{ N}_{\mathbb{S}}$.  Then $u $ be a continuous  plurisubharmonic function on $\Omega$. Moreover,  $(\triangle u_j )^n$ is a family of
uniformly bounded measures on each compact subset  $K$ of
$\Omega$ and weakly
converges to a non-negative measure on $\Omega$. This measure depends only on $u$ and not on
the choice of an approximating sequence  $\{u_j\}$.

  \section{The generalization of Malgrange's vanishing theorem and the Hartogs-Bochner extension  for  $k$-CF functions  on right-type
  groups}
\subsection{Subelliptic estimate }
For a nilpotent Lie group of step $2,$  its Lie algebra $\mathfrak{g}$  has decomposition:
$\mathfrak{g}=\mathfrak{g}_1\oplus\mathfrak{g}_2$
satisfying $[\mathfrak{g}_1,\mathfrak{g}_1]\subset\mathfrak{g}_2,\
[\mathfrak{g},\mathfrak{g}_2]=0.$    The group is called {\it stratified} if $[\mathfrak{g}_1,\mathfrak{g}_1]=\mathfrak{g}_2 $ (cf. \cite{BLU}).   Consider the  \emph{condition (H)}: for any $\lambda\in\mathfrak{g}_2^*\setminus\{0\},$ the skew-symmetric bilinear form
  $B_\lambda(Y,Y')=\langle\lambda,[Y,Y']\rangle,$
for  $Y,Y'\in \mathfrak{g}_1$,  is non-degenerate.

 \begin{proof} [Proof of Proposition \ref{prop:Subelliptic-estimate}] The estimate is proved in \cite{shi-wang} for the right quaternionic Heisenberg group. The proof can be adapted to right-type
  groups since we have similar  brackets by Corollary \ref{cor:Z -bracket} and the associated Hodge-Laplacian operator.

Vector fields in \eqref{eq:X-b} are not left invariant on $\mathcal{N}_{\mathbb{S}} $, but
   $X_{b}=\partial_{x_{b}}+2\sum_{\beta=1}^3\sum_{a=1}^{4n} B^\beta_{ab} x_a\partial_{t_\beta}$ is (cf. \cite[Subsection 5.1]{shi-wang}). Since they have the same brackets, we denote them also by  $X_{b}$'s by abuse of notations. The {\it SubLaplacian} on the
  groups  $\mathcal{N}_{\mathbb{S}} $ is
$
   \triangle_b :=-\sum_{a=1}^{4n} X_{ a}^2.
$

 By Theorem \ref{thm:bd-complex-main:1},  $f\in \Gamma\left(  \mathcal{N}_{\mathbb{S}} ,\mathscr{V}_{0}\right)$ on   a   right-type
  group  can be written as  $f=f_a\mathbf{S}_{k }^a$ for scalar functions $f_a$. Then
  $\mathscr{D  }_{0}f= (\mathscr{D  }_{0}^{(1)}f,\mathscr{D  }_{0}^{(2)}f) $ with
  \begin{equation}\label{eq:D-0-1}\begin{split}
      \mathscr{D  }_{0}^{(1)}(f_a\mathbf{S}_{k }^a)&=\partial_{A'}\mathfrak  d^{ A '}(f_a\mathbf{S}_{k }^a)=Z_{A}^{0'}f_a \mathbf{S}_{k-1 }^{a}\omega^A+Z_{A}^{1'}f_a \mathbf{S}_{k -1}^{a -1}\omega^A\\
    &= (Z_{A}^{0'}f_a +  Z_{A}^{1'}f_{a+1}  )\mathbf{S}_{k-1 }^{a }\omega^A\\
    &=Z_{A}^{A'}f_{b+o(A')}   \mathbf{S}_{k-1 }^{b }\omega^A.
\end{split}  \end{equation}

We introduce the inner products $\langle\cdot,\cdot\rangle$ of $\mathscr{V}_{0}$ and $\mathscr{V  }_{1}^{(1)}$ by requiring $\{\mathbf{S}_{k }^a\}$ and $\{ \mathbf{S}_{k -1}^{a }\omega^A\}$ to be orthonormal bases, respectively.
Then the operator formal adjoint to $\mathscr{D  }_{0}^{(1)}$ is given by
\begin{equation*}\begin{split}
      \mathscr{D  }_{0}^{(1)*} g
    &=- \overline{Z_{A}^{0'}}g_{A,0} \mathbf{S}_{k  }^{0 }-\overline{Z_{A}^{1'}}g_{A,k-1} \mathbf{S}_{k  }^{k }-\sum_{a=1 }^{k-1}\overline{Z_{A}^{B'}}g_{A,a-o(B')} \mathbf{S}_{k  }^{a },
\end{split}  \end{equation*}for $g=g_{A,b}\mathbf{S}_{k-1 }^b\omega^A\in C_0^\infty(\mathcal{N}_{\mathbb{S}}, \mathscr{V  }_{1}^{(1)})$, where $b$ is taken over $0,\ldots, k-1$.
This is because
\begin{equation*}\begin{split}\int_{\mathcal{N}_{\mathbb{S}}}\langle \mathscr{D  }_{0}^{(1)}f,g\rangle
         &= \int_{\mathcal{N}_{\mathbb{S}}}Z_{A}^{B'}f_{b+o(B')} \overline{g_{A,b}}dV=-\int_{\mathcal{N}_{\mathbb{S}}}f_{b+o(B')} \overline{\overline{Z_{A}^{B'}}g_{A,b}}dV
\end{split}  \end{equation*}by relabeling indices and integration by part (\ref{eq:Stokes}).
The {\it Hodge-Laplacian operator associated to $\mathscr{D  }_{0}^{(1)}$} is
\begin{equation*}\begin{split}
    \mathscr{D  }_{0}^{(1)*}  & \mathscr{D  }_{0}^{(1)}f=-\overline{Z_{A}^{0'}} Z_{A}^{0'}f_0\mathbf{S}_{k  }^{0 }-\overline{Z_{A}^{1'}} Z_{A}^{1'}f_k\mathbf{S}_{k  }^{k }-\sum_{a=1 }^{k-1} \sum_{ A, A',B' }\overline{Z_{A}^{B'}}Z_{A}^{A'}f_{a+o(A')-o(B')} \mathbf{S}_{k  }^{a }  \\&=-\overline{Z_{A}^{0'}} Z_{A}^{0'}f_0\mathbf{S}_{k  }^{0 }-\overline{Z_{A}^{1'}} Z_{A}^{1'}f_k\mathbf{S}_{k  }^{k }-\sum_{a=1 }^{k-1}  \left(\overline{Z_{A}^{0'}} Z_{A}^{0'} f_{a}+ \overline{Z_{A}^{1'} }Z_{A}^{1'} f_{a}+ \overline{Z_{A}^{0'}} Z_{A}^{1'} f_{a+1}+ \overline{Z_{A}^{1'}} Z_{A}^{0'} f_{a-1}\right)\mathbf{S}_{k  }^{a }.
\end{split}\end{equation*}
Note that
\begin{equation*}\begin{split}
    \sum_{A=2l,2l+1  }  \overline{ Z}_{A}^{0'}Z_{A}^{0'}  & =X_{ 4l+3}^2+  X_{4l+4}^2  + \mathbf{i}[X_{ 4l+3},  X_{4l+4}]
+X_{4l+1}^2+ X_{4l+2}^2-\mathbf{i}[X_{4l+1}, X_{4l+2}] \\&= X_{4l+1}^2+ X_{4l+2}^2+
                              X_{ 4l+3}^2+  X_{4l+4}^2 = \sum_{A=2l,2l+1  }   \overline{Z}_{A}^{1'}Z_{A}^{1'},
\end{split}\end{equation*}
by \eqref{eq:t-k-CF-raised}, since
\begin{equation}\label{eq:X-X}\begin{split}
   0=4Z _{ [2l}^{(0 '} Z _{ 2l+1]}^{1')} =[ Z _{  2l}^{ 0 '}, Z _{ 2l+1 }^{1' }]+[ Z _{  2l}^{ 1'}, Z _{ 2l+1 }^{0 ' }]=-[X_{ 4l+3},  X_{4l+4}] + [X_{4l+1}, X_{4l+2}]   ,
  \end{split}\end{equation}by using Corollary \ref{cor:Z -bracket}. Similarly, we have
  and
\begin{equation*}\begin{split}
       \sum_{A=2l,2l+1  }  \overline{ Z}_{A}^{0'}Z_{A}^{1'}&=Z_{2l+1}^{1'}Z_{2l}^{1'} -{Z_{2l }^{1'}}Z_{2l+1}^{1'}=2Z_{[2l+1}^{1'}Z_{2l]}^{1'}=0,
\\    \sum_{A=2l,2l+1  }  \overline{ Z}_{A}^{1'}Z_{A}^{0'} &=- {Z_{2l+1}^{0'}}Z_{2l}^{0'} +{Z_{2l }^{0'}}Z_{2l+1}^{0'}= 2Z_{[2l}^{0'}{Z_{2l+1]}^{0'}}=0,\end{split}  \end{equation*}by
$
   \overline{Z_{2l}^{1'}}=- {Z_{2l+1}^{0'}},$ $ \overline{Z_{2l+1}^{1'}}={Z_{2l }^{0'}}.
$
So we get
 \begin{equation*}\begin{split}
    \mathscr{D  }_{0}^{(1)*}  & \mathscr{D  }_{0}^{(1)}f= \triangle_b f_0\mathbf{S}_{k  }^{0 }+ 2\sum_{a=1 }^{k-1} \triangle_b f_a \mathbf{S}_{k  }^{a }+\triangle_b f_k\mathbf{S}_{k  }^{k }.
\end{split}\end{equation*} If  identify $f\in \Gamma(bD,\mathscr{V}_0)$  with a vector $f=(f_0,\ldots, f_k)^t\in  \mathbb{C}^{k+1}\cong \odot^{k}\mathbb{C}^{2}$,
we have
\begin{equation}\label{eq:diag}
      \mathscr{D  }_{0}^{(1)*}     \mathscr{D  }_{0}^{(1) } = \operatorname{diag} \left(
   \triangle_b ,     2\triangle_b , \ldots, 2\triangle_b  ,  \triangle_b
 \right).
\end{equation}

  Now we have
 \begin{equation*}\begin{split}
    \left\|\mathscr{D  }_{0}^{(1)}f\right\|_0^2&=\left\|\mathscr{D  }_{0}^{(1)}f\right\|_0^2+\left\|\mathscr{D  }_{0}^{(2)}f\right\|_0^2
    \geq \left\|\mathscr{D  }_{0}^{(1)}f\right\|_0^2\\&=\int_{\mathcal{N}_{\mathbb{S}}}\langle \mathscr{D  }_{0}^{(1)*}\mathscr{D  }_{0}^{(1)}f ,f\rangle \geq-\int_{\mathcal{N}_{\mathbb{S}}}\sum_{a=1}^{4n}\langle X_{ a}^2f ,f\rangle=\sum_{a=1}^{4n}\|X_{ a} f\|_0^2.
 \end{split}\end{equation*}
It is well known that the $ \frac 12$-subelliptic estimate is satisfied when $[X_a,X_b]$'s span $  \left\{ \partial_{t_1},
\partial_{t_2},\partial_{t_3}\right\}$, i.e. it is stratified (cf. e.g. \cite{F}).

If $k=0$, note that
\begin{equation*}
   \mathfrak d^{0'} \mathfrak d^{1'}u=Z_{A}^{0'}Z_{B}^{0'}\omega^A\wedge\omega^B=\sum_{l=0}^{2n-1}\bigtriangleup_l \omega^{2l}\wedge\omega^{2l+1}+
   \sum_{|A-B|\neq 1}Z_{A}^{0'}Z_{B}^{1'}\omega^A\wedge\omega^B,
\end{equation*}
with
\begin{equation*}\begin{split}
    \bigtriangleup_l&=Z_{2l }^{0'} Z_{2l+1}^{1'}-Z_{2l+1}^{0'} Z_{2l }^{1'}=X_{4l+1}^2+ X_{4l+2}^2+ X_{ 4l+3}^2+  X_{4l+4}^2-\mathbf{i}[X_{4l+3}, X_{4l+4}] + \mathbf{i}[X_{ 4l+1},  X_{4l+2}]\\&=
   X_{4l+1}^2+ X_{4l+2}^2+ X_{ 4l+3}^2+  X_{4l+4}^2,
\end{split}\end{equation*}by sing \eqref{eq:t-k-CF-raised} and  \eqref{eq:X-X}. So we have
\begin{equation*}
   \|\mathfrak d^{0'} \mathfrak d^{1'}u\|_0^2\geq  \|\triangle_b u\|_0^2\geq \frac 12\sum_{a=1}^{4n}\|X_{ a} f\|_0^2-\|  u\|_0^2
\end{equation*}by the Cauchy-Schwarz inequality.
The estimate follows similarly.
 \end{proof}

 \begin{cor}\label{cor:regularity} Suppose that  $\mathcal{N}_{\mathbb{S}}$ is stratified. For any   $s\in \mathbb{Z}$, if $u\in \mathscr E'(\mathcal{N}_{\mathbb{S}} ,\mathscr{V}_0)$ satisfies  $ \mathscr{D  }_{0}u \in W^s (\mathcal{N}_{\mathbb{S}} ,\mathscr{V}_1)$,  we have $  u \in W^{s+\frac 12} (\mathcal{N}_{\mathbb{S}} ,\mathscr{V}_0)$.
Moreover, if  $\operatorname{supp}( u ) \subset K$,
   there are constants $C_{s, K}\geq 0$, $c_{s, K}> 0$ such that
 \begin{equation}\label{eq:Subelliptic-estimate-r}
C_{s, K} \| u\|_s^2+\left\|\mathscr{D  }_{0} u\right\|_s^2\geq c_{s, K}\| u\|_{s+\frac 12}^2.
    \end{equation}
 \end{cor}Recall that a distribution in $  \mathcal{E}' $ always has   compact support.
This regularity follows from   subelliptic estimate by the standard procedure (cf. \cite{Nacinovich17}). We just mention that $ \mathscr{D  }_{0}u \in W^s (\mathcal{N}_{\mathbb{S}} ,\mathscr{V}_1)$ implies $  \triangle_b u \in W^s (\mathcal{N}_{\mathbb{S}} ,\mathscr{V}_1)$ and omit   details (cf. \cite{F} for a version of regularity  for SubLaplacians).

\subsection{Abstract duality theorem}
A cohomological complex of topological vector spaces
is a pair $(E^\bullet, d)$ where $ E^\bullet = (E^q)_{q\in\mathbb{ Z}}$ is a sequence of topological vector
spaces and $d = (d^q)_{q\in\mathbb{ Z}}$ is a sequence of continuous linear maps $d^q:E^q\rightarrow E^{q+1}$   satisfying $d^{q+1} \circ d^q= 0$. Its cohomology groups  $H^q(E^\bullet )$
are the quotient spaces $\ker d^{q }/\operatorname{Im } d^{q-1}$,
endowed with the  quotient topology.

A homological complex of topological vector spaces is a pair $(E_\bullet, d)$ where $ E_\bullet = (E_q)_{q\in\mathbb{ Z}}$ is a sequence of topological vector
spaces and $d = (d_q)_{q\in\mathbb{ Z}}$ is a sequence of continuous linear maps $d_q:E_{q+1}\rightarrow E_q $
   satisfying  $ d_{q-
  1} \circ d_q= 0$. Its  homology groups  $H_q(E_\bullet )$
are the quotient spaces $\ker d_{q-1}/\operatorname{Im }d_{q }$,
endowed with the quotient topology.
The dual complex of a cohomological complex $(E^\bullet, d)$ of topological
vector spaces is the homological complex $(E_\bullet', d')$ where $ E_\bullet' = (E_q')_{q\in\mathbb{ Z}}$  with
$  E_q' $
the strong dual of $  E_q $ and $d' = (d_q')_{q\in\mathbb{ Z}}$  with $ d_q' $
the transpose map of $ d_q $.

Recall that   a Fr\'echet-Schwartz space  is a topological vector space whose topology is defined by an increasing sequence of seminorms  such that   the unit ball
with respect to the seminorm   is relatively compact for the topology associated to the previous seminorm. A Fr\'echet-Schwartz space and the dual of a Fr\'echet-Schwartz space are both reflexive \cite{LL99}. We need the following abstract duality theorem.

\begin{thm} \label{thm:duality} \cite[Theorem 1.6]{LL99} Let $(E^\bullet, d)$ be a cohomological complex of  Fr\'echet-Schwartz spaces
or of dual of Fr\'echet-Schwartz spaces and let $(E_\bullet, d)$ be its dual complex. For each $  q\in\mathbb{ Z} $ , the following
assertions are equivalent :
\\
(i) $\operatorname{Im }  d^q = \{g\in E^{q+1} | \langle g, f\rangle = 0$ for any $f \in\ker d_q'\}$;
\\(ii) $H^{q+1}(E^\bullet )$ is separated;
\\(iii) $ d^q $
is a topological homomorphism;
\\(iv) $ d_q' $
is a topological homomorphism;
\\
(v) $H_{q }(E^\bullet )$ is separated;
\\
(vi) $\operatorname{Im }  d_q' = \{f\in E_q' | \langle f, g\rangle = 0$ for any $f \in\ker d^q \}$.
\end{thm}

A continuous linear map $\psi$ between topological vector spaces $L_1$ and $L_2$ is called a {\it topological homomorphism} if for each open subset $U\subset L_1$, the image $\psi(U)$ is an open subset of $\psi(L_1)$. It is known that if $L_1$ is a  Fr\'echet  space, $\psi$ is a topological homomorphism if and only if $\psi(L_1)$ is closed   \cite[P. 77]{SW99}. See e.g. \cite{Brinkschulte10,HN00,LL99} for applications of abstract duality theorem to $\overline{\partial}$- or $\overline{\partial}_b$-complex. We adapt their methods to the boundary complex   of   the $k$-Cauchy-Fueter complex.

For a complex vector space $V$, let $\mathscr E(\mathcal{N}_{\mathbb{S}},V)$ be the space of  smooth  $V $-valued functions with the topology of uniform convergence on compact
sets of the functions and all their derivatives. Endowed with this topology
$\mathscr E(\mathcal{N}_{\mathbb{S}},V)$  is a Fr\'echet-Schwartz space.
 Let $\mathscr D (\mathcal{N}_{\mathbb{S}},V)$ be
the space
of compactly supported elements of $\mathscr E(\mathcal{N}_{\mathbb{S}},V)$.
For   a compact subset $K$ of $\mathcal{N}_{\mathbb{S}}$, let  $\mathscr D_K(\mathcal{N}_{\mathbb{S}},V)$
the closed subspace of
 $\mathscr E(\mathcal{N}_{\mathbb{S}},V)$  with support in $K$ endowed with the induced topology.
Choose $\{K_n\}_{n\in N}$ an exhausting sequence of compact subsets of $\mathcal{N}_{\mathbb{S}}$. Then
$\mathscr D (\mathcal{N}_{\mathbb{S}},V)=\cup_{n=1}^\infty \mathscr D_{K_n}(\mathcal{N}_{\mathbb{S}},V)$. We put on $\mathscr D (\mathcal{N}_{\mathbb{S}},V)$ the strict inductive limit
topology defined by the Fr\'echet-Schwartz spaces  $ \mathscr D_{K_n}(\mathcal{N}_{\mathbb{S}},V)$ \cite{LL99}.    Denote  by $[\mathscr E(\mathcal{N}_{\mathbb{S}} ,V)]'$ the  dual of $\mathscr E(\mathcal{N}_{\mathbb{S}} ,V)$ and by $[\mathscr D(\mathcal{N}_{\mathbb{S}} ,V)]'$ the  dual of $\mathscr D(\mathcal{N}_{\mathbb{S}} ,V)$.

\subsection{The generalization of Malgrange's vanishing theorem and  Hartogs-Bochner extension for $k$-CF functions }
   On a right-type
  group, we have the dual differential complex $\widehat{\mathscr{V}}_\bullet $:
 \begin{equation} \label{eq:dual-complex}
0\longleftarrow \Gamma(\mathcal{N}_{\mathbb{S}} ,\widehat{\mathscr{V}}_0 )\xleftarrow{\widehat{\mathscr{D}} _0 } \cdots\longleftarrow \Gamma(\mathcal{N}_{\mathbb{S}} ,\widehat{\mathscr{V}}_j )\xleftarrow{\widehat{\mathscr{D}} _j }
    \cdots\xleftarrow{\widehat{\mathscr{D}} _{2n-2} }
   \Gamma(\mathcal{N}_{\mathbb{S}} ,\widehat{\mathscr{V }}_{2n-1} )\longleftarrow 0.
 \end{equation}
 where $\widehat{\mathscr{V}}_j=\widehat{\mathscr{V}}_j^{(1)}\oplus \widehat{\mathscr{V}}_j^{(2)}$ with
 \begin{equation}\label{eq:Q-dual}\begin{array}{llll}
\widehat{\mathscr{V}}_j^{(1)}:=&\mathscr{V} ^{\sigma_j,2n-\tau_j } ,\qquad\qquad & \widehat{\mathscr{V}}_j^{(2)}:=\mathscr{V} ^{\sigma_{j+1 },2n+1-\tau_j  }  ,\qquad    & {\rm if}\quad j \neq k  ,\\\widehat{\mathscr{V}}_k^{(1)}:=&\mathscr{V} ^{0,2n-k } ,\qquad\qquad & \widehat{\mathscr{V}}_k^{(2)}:=\mathscr{V} ^{0,2n-k }   ,
\end{array}\end{equation}
and $\widehat{\mathscr{V}}_0^{(2)}=\emptyset$. Then, $[\mathscr E(\mathcal{N}_{\mathbb{S}} ,\mathscr{V}_j)]'$ (respectively, $[\mathscr D(\mathcal{N}_{\mathbb{S}} ,\mathscr{V}_j)]'$) can be identified with $ \mathscr E'(\mathcal{N}_{\mathbb{S}} ,\widehat{\mathscr{V}}_j)$ (respectively, $ \mathscr D'(\mathcal{N}_{\mathbb{S}} ,\widehat{\mathscr{V}}_j) $).

  The dual can be realization as follows (we only need $j=0,1$ here).
For  $F\in\mathscr{E} (\mathcal{N}_{\mathbb{S}} , \odot^{k }\mathbb{C}^2\otimes\Lambda^{2n }\mathbb{C}^{2n }  )$,
 write it as
$
    (F^{\mathbf{A}' })
 $ with $F^{\mathbf{A}' }\in\mathscr{E} (\mathcal{N}_{\mathbb{S}} , \Lambda^{2n }\mathbb{C}^{2n }  )$ (see the Appendix for this notation).
 It  defines a functional on $\mathscr D(\mathcal{N}_{\mathbb{S}} ,\mathscr{V}_0)$ by
  \begin{equation}\label{eq:integral-functional}\langle F,\varphi\rangle :=\int  F^{\mathbf{A}' }\varphi_{  \mathbf{A}'}   ,
\end{equation}
 for $\varphi\in\mathscr D(\mathcal{N}_{\mathbb{S}} ,\mathscr{V}_0)$. Similarly, for  $\mathbb{F}=(\mathbb{F}_1,$ $\mathbb{F}_2)\in\mathscr{E} (M, \widehat{\mathscr{V}}_ 1 )$,
 write $\mathbb{F}=(      \mathbb{F}_1^{\mathbf{A}' } ,   \mathbb{F}_2^{\mathbf{B}' })
  $ with $     \mathbb{F}_1^{\mathbf{A}' }\in  \mathscr{E} (\mathcal{N}_{\mathbb{S}} ,$ $  \Lambda^{2n-1 }\mathbb{C}^{2n }  ),   \mathbb{F}_2^{\mathbf{B}' }\in  \mathscr{E} (\mathcal{N}_{\mathbb{S}} ,  \Lambda^{2n }\mathbb{C}^{2n }  )
  $,  $|\mathbf{A}'|=\sigma_1=k-1$, $|\mathbf{B}'|=\sigma_2=k-2$. It  defines a functional on $\mathscr D(\mathcal{N}_{\mathbb{S}} ,\mathscr{V}_1)$ by
  \begin{equation}\label{eq:integral}\langle  {\mathbb{F}}, \psi\rangle :=\int \left(\mathbb{F}_1^{\mathbf{A}' } \wedge  \psi^1_{\mathbf{A}' } +\mathbb{F}_2^{\mathbf{B}' }    \psi_{\mathbf{B}' }^2\right)  ,
\end{equation}for any $\psi=(  \psi^1_{\mathbf{A}' }, \psi_{\mathbf{B}' }^2)\in\mathscr D(\mathcal{N}_{\mathbb{S}} ,\mathscr{V}_1)$,
since $\mathbb{F}_1^{\mathbf{A}' } \wedge  \psi^1_{\mathbf{A}' } +\mathbb{F}_2^{\mathbf{B}' }   \psi_{\mathbf{B}' }^2$ is an element of $\mathscr{D} (\mathcal{N}_{\mathbb{S}} ,  \Lambda^{2n }\mathbb{C}^{2n }  )$.

Recall that on a right-type
  group,
$\mathscr D_0:\mathscr E(\mathcal{N}_{\mathbb{S}} ,\mathscr{V}_ 0) \rightarrow \mathscr E(\mathcal{N}_{\mathbb{S}} ,\mathscr{V}_1)$ is given by
\begin{equation*} \begin{split}{\mathscr{D} }_{0}f  =&\left(
 \partial_{A'}\mathfrak  d^{A '} f ,-\partial_{A'}\partial_{B'} \mathbf{T }^{A' B '}f \right)
\end{split} \end{equation*}
for $f\in
\mathscr E(\mathcal{N}_{\mathbb{S}} ,\mathscr{V}_ 0) $ by \eqref{eq:bdoperator<}.
Then for $\mathbb{F}=(\mathbb{F}_1,\mathbb{F}_2)\in
\mathscr D(\mathcal{N}_{\mathbb{S}} ,\widehat{\mathscr{V}}_ 1) $
 \begin{equation*}\begin{split}  \langle\widehat{{\mathscr{D} }}_{0}{\mathbb{F}}, f \rangle&=\langle{\mathbb{F}},  \mathscr{D}_{0}f \rangle = \int  \mathbb{F}_1^{\mathbf{A}' } \wedge  \mathfrak d^{A'}f_{ A'\mathbf{A}'} -\int  \mathbb{F}_2^{\mathbf{B}' }  \mathbf{T }^{A' B '}f _{ A'B '\mathbf{B}'}    \\
 &=-\int   \mathfrak d^{(A'}\mathbb{F}_1^{\mathbf{A}') }     f_{ A'\mathbf{A}'}+\int   \mathbf{T }^{(A' B '}\mathbb{F}_2^{\mathbf{B}') }f_{ A'B '\mathbf{B}'}
           \end{split} \end{equation*}by  using Stokes-type formula in Lemma \ref{lem:stokes} and symmetrization.
         So we get
           \begin{equation}\label{eq:widehat-D}
            \left(\widehat{ \mathscr{D }}_{0}{\mathbb{F}}\right)^{ A' B ' \mathbf{B}'  } =-  \mathfrak  d^{(A'}\mathbb{F}_1^{B'\mathbf{B}') }  +  \mathbf{T }^{(A' B '}\mathbb{F}_2^{\mathbf{B}') }.
           \end{equation}

  M\'etivier proved the following theorem for analytic hypoellipticity.
\begin{thm}\label{elliptic}{\rm (\cite[Theorem 0]{Metivier2})}
Let $P$ be a homogeneous  left invariant differential operator on a nilpotent Lie group  satisfying condition (H). Then  $P$ is analytic hypoelliptic if and only if
  $P$ is $C^\infty$ hypoelliptic.
 \end{thm}

\begin{cor}\label{cor:anal} On a   right-type
  group $ \mathcal{ N}_{\mathbb{S}}$ satisfying condition (H),
$\Delta_b$ is analytic hypoelliptic.
\end{cor}It follows from Theorem \ref{elliptic}
and subelliptic estimate \eqref{eq:Subelliptic-estimate-r} in Corollary \ref{cor:regularity}.

 \begin{thm}\label{thm:closed} On a   right-type
  group $ \mathcal{ N}_{\mathbb{S}}$ satisfying condition (H), $\mathscr{D  }_{0}:\mathscr E'(\mathcal{N}_{\mathbb{S}} ,\mathscr{V}_0)\rightarrow \mathscr E'(\mathcal{N}_{\mathbb{S}} ,\mathscr\mathcal {V}_1)$ and  $\mathscr{D  }_{0}:\mathscr D(\mathcal{N}_{\mathbb{S}} ,\mathscr\mathcal {V}_0)\rightarrow \mathscr D(\mathcal{N}_{\mathbb{S}} ,\mathscr\mathcal {V}_1)$ have closed
range.
 \end{thm}
\begin{proof}
Let $\{f_\nu \}$ be a sequence in
$ \mathscr E'(\mathcal{N}_{\mathbb{S}} ,\mathscr{V}_0) $ such that all $\mathscr{D  }_{0}f_\nu\rightarrow g$ in $ \mathscr E'(\mathcal{N}_{\mathbb{S}} ,\mathscr{V}_0) $. Then $\mathscr{D  }_{0}f_\nu $ are all supported in a fixed compact subset $K   $ for all $\nu$ and there is $s\in  \mathbb{Z}$
such that $\mathscr{D  }_{0}f_\nu\in W^s (\mathcal{N}_{\mathbb{S}} ,\mathscr{V}_0)$,  and $\mathscr{D  }_{0}f_\nu\rightarrow g $ in $ W^s (\mathcal{N}_{\mathbb{S}} ,\mathscr{V}_1)$ (cf. \cite{Schwartz}). Consequently,  $\mathscr{D  }_{0}f_\nu\equiv 0$   outside of $  K$, i.e. $f_\nu$ is $k$-CF on $\mathcal{N}_{\mathbb{S}} \setminus K$. We
can assume that $\mathcal{N}_{\mathbb{S}} \setminus K$ has no compact connected component. Since
$\Delta_b$ is analytic hypoelliptic by Corollary \ref{cor:anal}, and  each component of a $k$-CF function annihilalted by $\Delta_b$ by \eqref{eq:diag},
   $k$-CF functions $f_\nu |_{\mathcal{N}_{\mathbb{S}} \setminus K } $   vanish on   each connected component of
$\mathcal{N}_{\mathbb{S}} \setminus K$, and thus  $\{f_\nu \} $  are also supported in  $  K$.

This argument also implies that the estimate \eqref{eq:Subelliptic-estimate-r} in Corollary \ref{cor:regularity} holds with $C_{s, K}= 0$. If this is not true, there exist a sequence $h_\nu\in W^s(\mathcal{N}_{\mathbb{S}} ,\mathscr{V}_0)$ such that
\begin{equation*}
    \left\|\mathscr{D  }_{0} h_\nu\right\|_s^2<\frac 1\nu  \| h_\nu\|_{s+ \frac 12}^2.
\end{equation*}
By rescaling we can assume that $\| h_\nu\|_s  = 1$ for each $\nu$. By \eqref{eq:Subelliptic-estimate-r}
\begin{equation*}
   C_{s, K}  \geq \left(c_{s, K}-\frac 1\nu\right)\| h_\nu\|_{s+\frac 12}^2.
\end{equation*}Thus $\{h_\nu \}$ is bounded in the Sobolev space $  W^{s+\frac 12}(\mathcal{N}_{\mathbb{S}} ,\mathscr{V}_0)$.
By the well known compactness  of the inclusion $  W^{s+ \frac 12}(\mathcal{N}_{\mathbb{S}} ,\mathscr{V}_0)\subset  W^s(\mathcal{N}_{\mathbb{S}} ,\mathscr{V}_0)$, there is a subsequence
  that   converges to a function $h_\infty$ in $  W^s(\mathcal{N}_{\mathbb{S}} ,\mathscr{V}_0)$. We have
\begin{equation*}
   \| h_\infty\|_s  = 1,\qquad  \mathscr{D  }_{0}h_\infty =0.
\end{equation*}  Then $\triangle_b h_\infty=0$ and $h_\infty$ is compactly supported. So $h_\infty\equiv 0$ by analytic continuation, which contradicts to $\| h_\infty\|_s  = 1$.

By the estimate \eqref{eq:Subelliptic-estimate-r}  in Corollary \ref{cor:regularity}  with $C_{s, K}= 0$, we see that
  $\{f_\nu \}$ is uniformly bounded in $  W^{s+\frac 12}  (\mathcal{N}_{\mathbb{S}} ,\mathscr{V}_0)$, and hence contains a subsequence which
converges to a compactly   supported weak solution $ f\in W^{s }  (\mathcal{N}_{\mathbb{S}} ,\mathscr{V}_0)$ of  $\mathscr{D  }_{0}f=g$. Namely, the image of $\mathscr{D  }_{0} $ in $\mathscr E'(\mathcal{N}_{\mathbb{S}} ,\mathscr{V}_1)$ is  closed.  The closedness of the image of $\mathscr{D  }_{0} $ in $\mathscr D(\mathcal{N}_{\mathbb{S}} ,\mathscr{V}_1)$ follows from the   proved result for
$ \mathscr E'(\mathcal{N}_{\mathbb{S}} ,\mathscr{V}_1)$  and the hypoellipticity.
\end{proof}

 \begin{proof}[Proof of Theorem \ref{thm:Malgrange}] By Theorem \ref{thm:closed},   the sequences
  \begin{equation}\label{eq:D0-closed}\begin{split} 0\longrightarrow&  \mathscr  D(\mathcal{N}_{\mathbb{S}} ,\mathscr{V}_0)\xrightarrow{\mathscr{D  }_{0}} \mathscr D(\mathcal{N}_{\mathbb{S}} ,\mathscr{V}_1),\\
  0\longrightarrow&  \mathscr E'(\mathcal{N}_{\mathbb{S}} ,\mathscr{V}_0)\xrightarrow{\mathscr{D  }_{0}} \mathscr E'(\mathcal{N}_{\mathbb{S}} ,\mathscr{V}_1),
           \end{split} \end{equation}
both are exact and  have closed ranges.   Thus $\mathscr{D  }_{0}$'s in \eqref{eq:D0-closed} are topological homomorphisms \cite[P. 77]{SW99}. Now we can apply   abstract duality theorem
\ref{thm:duality} (vi) to sequences in \eqref{eq:D0-closed} to get exact sequences
 \begin{equation}\label{eq:D0-closed-dual}\begin{split} 0\longleftarrow&  \mathscr  D'(\mathcal{N}_{\mathbb{S}} ,\widehat{\mathscr{V}}_0)\xleftarrow{\widehat{\mathscr{D  }}_{0}} \mathscr D'(\mathcal{N}_{\mathbb{S}} ,\widehat{\mathscr{V}}_1),\\
  0\longleftarrow&  \mathscr E (\mathcal{N}_{\mathbb{S}} ,\widehat{\mathscr{V}}_0)\xleftarrow{\widehat{\mathscr{D  }}_{0}} \mathscr E (\mathcal{N}_{\mathbb{S}} , \widehat{\mathscr{V}}_1),
           \end{split} \end{equation}i.e.   $\widehat{\mathscr{D  }}_{0}$'s are  surjective. The result follows.
\end{proof}

We  have the following   Hartogs-Bochner extension   for $k$-CF functions.

\begin{thm}\label{thm:Hartogs-Bochner-CF} Suppose that the   right-type
  group $ \mathcal{ N}_{\mathbb{S}}$ satisfies  condition (H) and $k>0$. Let
 $\Omega  $ be a domain of $\mathcal{N}_{\mathbb{S}}$ with   smooth boundary such that  $ \mathcal{ N}_{\mathbb{S}}\setminus \overline{\Omega}$ connected, and let $\rho$ be a defining function (i.e. $\rho=0$ on $\partial \Omega$ and $\rho<0$ in $\Omega$) such that $|{\rm grad}\, \rho|=1$.  Suppose that $f $ is the restriction to $\partial\Omega$
 of a  $ C^ 2(\Omega,\mathscr{V}_0)$ function,
with $\mathscr{D  }_{0}f $ vanishing to the second order on $\partial\Omega$, and satisfies the momentum condition
 \begin{equation}\label{eq:momentum}
   \int_{\partial\Omega} \left[f_{ A'\mathbf{A}'}(G_1^{\mathbf{A}' })_A  Z_A^{A'}\rho - f_{ A'B '\mathbf{B}'}  \mathbf{T }^{(A' B '}\rho\,G_2^{\mathbf{B}') }\right]  dS=0
 \end{equation}for any $G\in\mathscr E(\mathcal{N}_{\mathbb{S}} ,\widehat{\mathscr{V}}_1)$.
Then there exists  a $k$-CF function $\widetilde{f}\in  C^ 2(\overline{\Omega},\mathscr{V}_0)$  such that $\widetilde{f} =f $ on $\partial\Omega$.
 \end{thm}
 \begin{proof}
 Extending $\mathscr{D}_0 f$ by $0$ outside of $\overline{\Omega}$, we get a  $\mathscr{D}_1$-closed continuous element $  F \in \mathscr E'(\mathcal{N}_{\mathbb{S}} ,\mathscr{V}_1)$
supported in  $\overline{\Omega}$.  Since $H_0(\mathscr E (\mathcal{{N}}_{\mathbb{S}},  \widehat{\mathscr{V}}_\bullet))$ vanish by Theorem \ref{thm:Malgrange},
so it is separated. Thus,
we can apply    abstract duality theorem
\ref{thm:duality} (i) to the second sequences  in \eqref{eq:D0-closed} and \eqref{eq:D0-closed-dual} to see that for
 \begin{equation*}
     \operatorname{Im }   \mathscr{D}_0= \{\mathbb{F}\in \mathscr E'(\mathcal{N}_{\mathbb{S}} , {\mathscr{V}}_1 )| \langle \mathbb{F},G\rangle = 0 \text{ for any } G \in\ker  \widehat{\mathscr{D}}_0\}.
 \end{equation*}Consequently, we have $F\in  \operatorname{Im }   \mathscr{D}_0$, because   for any $G\in \ker  \widehat{\mathscr{D}}_0\subset \mathscr E (\mathcal{N}_{\mathbb{S}} ,\widehat{\mathscr{V}}_1 )$, we have
  \begin{equation*} \begin{split} \langle   F , G\rangle  =& \langle   \mathscr{D}_0 f , G\rangle =   \int \left( \mathfrak d^{A'}f_{ A'\mathbf{A}'}  \wedge G_1^{\mathbf{A}' }    +\mathbf{T }^{A' B '}f_{ A'B '\mathbf{B}'}  G_2^{\mathbf{B}' } \right)  \\
  =& \langle   f , \widehat{\mathscr{D}}_0 G\rangle +\int_{\partial\Omega} \left[f_{ A'\mathbf{A}'}(G_1^{\mathbf{A}' })_A  Z_A^{A'}\rho\, dS -f_{ A'B '\mathbf{B}'}  \mathbf{T }^{(A' B '}\rho\, G_2^{\mathbf{B}') }\right]  dS=0.
\end{split} \end{equation*}by using \eqref{eq:widehat-D} and Stokes-type formula \eqref{stokes}.
Hence, there exists a distribution $H \in \mathcal E'(\mathcal{N}_{\mathbb{S}} , \mathscr{V }_0)$ such that $ {F} = \mathscr{D}_0 H$.

Recall that a distribution in $  \mathcal{E}' $ always has   compact support. Now by the regularity of the $\mathscr{D}_0 $ operator in Corollary \ref{cor:regularity}, $H \in W^{1+\epsilon} (\mathcal{N}_{\mathbb{S}} , {\mathscr{V}}_0)$.
Then $H$  is $k$-regular  on the connected  open  set   $ \mathcal{N}_{\mathbb{S}} \setminus \overline{\Omega}$,
  since supp $F \subset \overline{\Omega}$. By real analyticity of  $k$-CF functions on this kind of groups, $H$ vanishes on $ \mathcal{N}_{\mathbb{S}} \setminus \overline{\Omega}$. Hence,
$F = f-H$ gives us  the required extension.
\end{proof}

\section{Appendix. }
The
{\it $\sigma$-th symmetric power} $ \odot^{\sigma
}\mathbb{C}^{2 } $ is a subspace of $ \otimes^{\sigma
}\mathbb{C}^{2 } $, and an element of $\odot^{\sigma} \mathbb{C}^2 $ is given by a $2^\sigma$-tuple  $
(  f_{\mathbf{A} ' })\in \otimes^{\sigma} \mathbb{C}^2$ with $\mathbf{A} '=A_1' \ldots A_\sigma'$ ($A_1' ,\ldots, A_\sigma'   =0',1' $)   such that
$  f_{\mathbf{A} '} $ is invariant under   permutations of
subscripts, i.e.
\begin{equation*}
     f_{A_1' \ldots A_\sigma'} =  f_{   A_{\pi(1)}'\ldots A_{\pi(\sigma)} '}
 \end{equation*} for any $\pi\in  {S}_\sigma$, the group of permutations on $\sigma$ letters.
The  symmetrization    of indices is   defined as
\begin{equation}\label{eq:sym-0}
     f_{\cdots (A_1'\ldots A_\sigma')\cdots}: =\frac 1 {\sigma!}\sum_{\pi\in  {S}_\sigma}  f_{\cdots  A_{\pi(1)}'\ldots A_{\pi(\sigma)}' \cdots}.
     \end{equation}
In particular, if $(f_{A_1'\ldots A_\sigma'})\in \otimes^\sigma \mathbb{C}^{2 }$ is symmetric in $A_2'\ldots A_\sigma'$, then we have
 \begin{equation}\label{eq:sym-1}
   f_{(A_1'\ldots A_\sigma')}=\frac 1k \left(f_{A_1'A_2'\ldots A_\sigma'}+\cdots+ f_{A_s'A_1'\ldots \widehat{A_s'}\ldots A_\sigma'}+\cdots+ f_{A_\sigma'A_1'\ldots A_{\sigma-1}' } \right).
 \end{equation}

For $j=0,1,\cdots,k-1 $, we write an element of $\Gamma(\mathbb{H}^{n+1}, \mathcal{{V}}_j)$ as a tuple $f=(f_{\mathbf{A}'})$ with symmetric indices $\mathbf{A}'=A_1'\cdots A_{ \sigma_j }'$  and $f_{  \mathbf{A}'}\in \Gamma(\mathbb{H}^{n+1},\wedge^{j}\mathbb{C}^{2(n+1)})$.
      $\mathcal{ {D}}_j $   is a differential operators of  first order  given by
 \begin{equation}\label{eq:D-j-1}
\left( \mathcal{{D}}_jf\right)_{\mathbf{A}'}= \sum_{A '=0',1'}d^{A '}f_{  A '\mathbf{A}'};
 \end{equation}
  while  for $j=k+1\cdots,2n+1 $,
 we write an element of $\Gamma(\mathbb{H}^{n+1}, \mathcal{{V}}_j)$ as a tuple $f=(f^{ \mathbf{A}'})$ with symmetric indices $\mathbf{A}' $  and  $f^{ \mathbf{A}'}\in \Gamma(\mathbb{H}^{n+1},\wedge^{j+1}\mathbb{C}^{2(n+1)})$.
 Then,
 \begin{equation}\label{eq:D-j-3}
    \left(\mathcal{ {D}}_{j}f\right)^{A '\mathbf{A}'}= d^{(A '}f^{\mathbf{A}')} ,
 \end{equation} where $(\cdots)$ is the symmetrization    of indices.

Define   isomorphisms
\begin{equation}\label{eq:Pi}
   \mathbf{ \dot{\Pi}}_j:  \Gamma(\mathbb{H}^{n+1},  \mathcal{{V}}_j) \longrightarrow \Gamma\left(\mathbb{H}^{n+1},  \mathcal{P}_{\sigma_j } (\mathbb{C}^2)  \otimes \wedge^{\tau_j}\mathbb{C}^{2n +2} \right).
\end{equation}
For $j=0,1,\cdots,k $,  the isomorphism $
   \mathbf{ \dot{\Pi}}_j$   is given by
 \begin{equation}\label{eq:dot-Pi1}
   \mathbf{ \dot{\Pi}}_j\left(f_{ \mathbf{A}  '  }\right) =  f_{ a } \mathbf{S}^{a}_{\sigma_j  }  ,
 \end{equation}
   where $f_{ a }:=f_{ \mathbf{A}  ' }$ with $o( \mathbf{A}  ')=a$, and the summation is taken over $ a=0,1,\ldots,\sigma_{j } $.
   If $j=k+1,\cdots,2n+1 $, the isomorphism $
   \mathbf{ \dot{\Pi}}_j$ is given by
\begin{equation}\label{eq:dot-Pi2}
    \mathbf{\dot{\Pi}}_j\left(f ^{\mathbf{A}  '  }\right)= f^{ \mathbf{A}  ' }s_{\mathbf{A}'}=f ^{a}  \widetilde{ \mathbf{S}}^{a}_{\sigma_j }\binom{\sigma_j  } a
\end{equation}
  where $f ^{a } =f^{ \mathbf{A}  ' }$ with $o( \mathbf{A}  ')=a$. Under this realization, operators $  \mathcal{D}_j$'s in  \eqref{eq:operator-k-CF}  in the $k$-Cauchy-Fueter complex  is the same as \eqref{eq:dot-Pi1}-\eqref{eq:dot-Pi2} by the following proposition.

\begin{prop} \label{prop:Pi} For $f\in \Gamma (\mathbb{H}^{n+1},  \mathcal{{V}}_j )$, we have
\begin{equation*}  \mathbf{ \dot{\Pi}}_{j+1}(  \mathcal{D}_jf)
 = \left\{
    \begin{array}{ll}  \partial_{A'} d^{ {A} '}\mathbf{ \dot{\Pi}}_j( f),\quad &{\rm if}\quad j=0, \ldots, k-1,\\
      s_{A'} d^{ {A} '}\mathbf{ \dot{\Pi}}_j( f) ,\quad &{\rm if}\quad  j=k+1,\ldots,2n+1.
    \end{array}\right.
\end{equation*}
\end{prop}
\begin{proof} Note that for $j=0, \ldots, k-1$,
 \begin{equation*} \begin{split} \partial_{A'}  d^{A'}\left(f_{ a } \mathbf{S}^{a}_{\sigma_j } \right)
 &  =    d^{0'} f_{ a } \mathbf{S}^{a}_{\sigma_{j+1 } }+ d^{1'} f_{ a } \mathbf{S}^{a-1}_{\sigma_{j+1 } } =   \left(d^{0'} f_{ b }+d^{1'} f_{b+1}\right ) \mathbf{S}^{b}_{\sigma_{j+1 }}  ,
\end{split}\end{equation*}by \eqref{eq:partial-S}, where $b$ is taken   over $0,1,\ldots,\sigma_{j+1}=\sigma_{j }-1$.
Apply the mapping  $\mathbf{ \dot{\Pi}}_{j+1}^{-1} $   to get
\begin{equation*}\begin{split}
  \left [\mathbf{ \dot{\Pi}}^{-1}_{j+1}\left(\partial_{A'} d^{ {A} '}\mathbf{ \dot{\Pi}}_j( f)\right)\right]_{ \mathbf{A}  '  } & = d^{0'} f_{0' \mathbf{A}  '  }+d^{1'} f_{1' \mathbf{A}  '  } .
\end{split}\end{equation*}

For $j=k+1,\cdots,2n+1$, noting that $\sigma_{j+1 }=\sigma_{j }+1$, we get
 \begin{equation*} \begin{split} s_{A'}  d^{A'}\left(f ^{ a } \widetilde{ \mathbf{S}}^{a}_{\sigma_j }  \binom{\sigma_j  } a\right)
 &  = \left(   d^{0'} f ^{ a } \widetilde{ \mathbf{S}}^{a}_{\sigma_j+1 }+    d^{1'} f ^{ a } \widetilde{ \mathbf{S}}^{a+1}_{\sigma_j+1 }\right) \binom{\sigma_j  } a \\&  =\left( \frac { \sigma_{j+1 }-(a+1)}{\sigma_{j+1 }}d^{0'} f ^{ a+1 } + \frac {a+1}{\sigma_{j+1 }}d^{1'} f ^{ a} \right) \widetilde{ \mathbf{S}}^{a+1}_{\sigma_{j+1 } } \binom{ {\sigma_{j+1 } }} {a+1}.
\end{split}\end{equation*}
Thus, for  $A'\mathbf{A}  '$ with $|A'\mathbf{A}  '|=\sigma_{j+1 }$ and $o(A'\mathbf{A}  ')=a+1$, we have
\begin{equation*}\begin{split}
     \left[\mathbf{ \dot{\Pi}}^{-1}_{j+1}\left(s_{A'} d^{ {A} '}\mathbf{ \dot{\Pi}}_j( f)\right)\right]^{ A'\mathbf{A}  '  } & =\frac { \sigma_{j+1 }-(a+1)}{\sigma_{j+1 }}  d^{0'}f^{\scriptstyle0'\ldots 0'\overbrace {\scriptstyle 1'  \ldots1'}^{a+1} }+ \frac {a+1}{\sigma_{j+1 }} d^{1'} f^{\scriptstyle0'\ldots 0'\overbrace {\scriptstyle 1' \ldots 1' }^{a } }  = d^{(A'}f^{  \mathbf{A}  ')  }
\end{split}\end{equation*}
by \eqref{eq:sym-1}.
\end{proof}

\end{document}